\documentclass[reqno]{amsart}

\usepackage[margin=1.4in]{geometry}
\usepackage{amssymb, amsmath, amsthm, amsfonts, amscd, mathtools, mathrsfs}
\usepackage[shortlabels]{enumitem}
\usepackage{tikz, tikz-cd}
\usepackage[all]{xy}
\usepackage{pst-node}
\usepackage{varwidth}
\usepackage{graphicx}
\usepackage{subfig}
\usepackage{tabularx, booktabs}
\usepackage{multirow, multicol}
\usepackage{xcolor}
\usepackage{romannum}
\usepackage{yfonts, euscript}
\usepackage{csquotes, dirtytalk}
\usepackage[pagewise]{lineno}%\linenumbers
\usepackage[colorlinks=true, allcolors=blue]{hyperref}

\usetikzlibrary{arrows, arrows.meta}

\makeatletter
\let\@wraptoccontribs\wraptoccontribs
\makeatother

\newtheorem{thm}{{\bf Theorem}}[section]
\newtheorem{lemma}[thm]{{\bf Lemma}}
\newtheorem{prop}[thm]{{\bf Proposition}}
\newtheorem{cor}[thm]{{\bf Corollary}}
\newtheorem*{defn}{{\bf Definition}}

\newtheorem*{rmk}{{\bf Remark}}
\newtheorem*{conv}{{\bf Convention}}

\newcommand{\m}{\mathfrak m}

\begin{document}

%%%%%%%%%%%%%%%%%%%%%%%%%%%%%%%%%%%%%%%%%%%%%%%%%%
%Document Information
%%%%%%%%%%%%%%%%%%%%%%%%%%%%%%%%%%%%%%%%%%%%%%%%%%

\title[Normal Subgroups of Twisted Chevalley Groups]{On Normal Subgroups of Twisted Chevalley Groups over Commutative Rings}

%%%%%%%%%%%%%%%%%%%%%%%%%%%%%%%%%%%%%%%%%%%%%%%%%%

\author{Shripad M. Garge}
\address{Shripad M. Garge. \newline \indent
        Department of Mathematics, 
	Indian Institute of Technology Bombay,\newline \indent
	Powai, Mumbai. 400076, India. \newline \indent
        School of Mathematics, 
	Harish Chandra Research Institute,\newline \indent
	Jhusi, Prayagraj 211019, India}
\email{\href{mailto:shripad@math.iitb.ac.in}{shripad@math.iitb.ac.in}, \href{mailto:smgarge@gmail.com}{smgarge@gmail.com}}
%\thanks{}

\author{Deep H. Makadiya}
\address{Deep H. Makadiya. \newline \indent
        Department of Mathematics, 
	Indian Institute of Technology Bombay,\newline \indent
	Powai, Mumbai 400076, India}
\email{\href{mailto:deepmakadia25.dm@gmail.com}{deepmakadia25.dm@gmail.com}}
%\thanks{}

% \contrib[with an appendix by]{Pavel Gvozdevsky}
\address{Pavel Gvozdevsky. \newline \indent
        Department of Mathematics, Bar-Ilan University, \newline \indent
        Ramat Gan 5290002, Israel}
\email{\href{mailto:gvozdevskiy96@gmail.com}{gvozdevskiy96@gmail.com}}

%%%%%%%%%%%%%%%%%%%%%%%%%%%%%%%%%%%%%%%%%%%%%%%%%%


%%%%%%%%%%%%%%%%%%%%%%%%%%%%%%%%%%%%%%%%%%%%%%%%%%

\begin{abstract}
    In this paper, we prove two structure theorems for twisted Chevalley groups $G_\sigma (R)$ over a commutative ring $R$ with unity. The first theorem concerns the normality of $E'_\sigma (R,J)$, the elementary congruence subgroups at level $J$, in the group $G_\sigma (R)$. The second theorem classifies all subgroups of $G_\sigma (R)$ normalized by its elementary subgroup $E'_\sigma (R)$. Along the way, we obtain several interesting results. For instance, when $R$ is a semilocal ring, we show that $G_\sigma(R)$ can be expressed as the (internal) product of $E'_\sigma (R)$ and the maximal torus $T_\sigma (R)$ of $G_\sigma (R)$.
\end{abstract}

%%%%%%%%%%%%%%%%%%%%%%%%%%%%%%%%%%%%%%%%%%%%%%%%%%

\pagenumbering{arabic} 
\maketitle 
%{\hfill \today}
\tableofcontents

%%%%%%%%%%%%%%%%%%%%%%%%%%%%%%%%%%%%%%%%%%%%%%%%%%
%Section: Introduction
%%%%%%%%%%%%%%%%%%%%%%%%%%%%%%%%%%%%%%%%%%%%%%%%%%

\setcounter{section}{0}

\section{Introduction}\label{sec:Intro}

\noindent Let $\Phi$ be a reduced irreducible root system and $R$ a commutative ring with unity. 
Let $\pi$ be a finite-dimensional faithful representation of the semisimple Lie algebra associated with $\Phi$. 
Consider the Chevalley group $G_\pi (\Phi, R)$ of type $\Phi$ over $R$, and let $E_\pi (\Phi, R)$ be its elementary subgroup (i.e., the subgroup generated by all elementary unipotent elements $x_\alpha (t)$ with $\alpha \in \Phi$ and $t \in R$).
For an ideal $J$ of $R$, the natural projection map $R \longrightarrow R/J$ induces a group homomorphism
\[
\phi: G_\pi (\Phi, R) \longrightarrow G_\pi (\Phi, R/J).
\]  
Define  
\[
G_\pi (\Phi, J) \coloneq \ker(\phi) \quad \text{and} \quad G_\pi (\Phi, R, J) \coloneq \phi^{-1}(Z(G_\pi (\Phi, R/J))),
\]  
where $Z(G_\pi (\Phi, R/J))$ denotes the center of $G_\pi (\Phi, R/J)$.  
The subgroup $G_\pi(\Phi, J)$ of $G_\pi (\Phi, R)$ is referred to as the \emph{principal congruence subgroup} of level $J$, while $G_\pi(\Phi, R, J)$ is called the \emph{full congruence subgroup} of level $J$.
Let $E_\pi (\Phi, J)$ denote the subgroup of $E_\pi (\Phi, R)$ generated by $x_\alpha(t)$ for all $\alpha \in \Phi$ and $t \in J$.  
Additionally, define $E_\pi (\Phi, R, J)$ as the normal subgroup of $E_\pi (\Phi, R)$ generated by $E_\pi (\Phi, J)$. The subgroup $E_\pi(\Phi, R, J)$ of $G_\pi (\Phi, R)$ is referred to as the \emph{elementary congruence subgroup} of level $J$. 

\medskip

We begin by recording two important structure theorems for Chevalley groups. Several variants of these theorems for different classical and exceptional groups over various rings are available in the literature (see N. A. Vavilov~\cite{NV1} for a historical perspective).

\begin{thm}[{L. N. Vaserstein \cite{LV}}]\label{thm:Intro1}
    Let $\Phi$ be an irreducible root system of rank $\geq 2$ and \(R\) a commutative ring with unity. If $J$ is an ideal of $R$, then the following commutator relations hold: 
    \[
        [E_\pi(\Phi, R, J), G_\pi(\Phi, R)] \subset E_\pi(\Phi, R, J) \quad \text{and} \quad [E_\pi(\Phi, R), G_\pi(\Phi, R, J)] \subset E_\pi(\Phi, R, J).
    \]
    Except in the cases where $\Phi = B_2$ or $G_2$ and $R$ has a residue field with two elements, these inclusions are equalities.
\end{thm}  

The first commutator relation in the above theorem is equivalent to saying that $E_\pi (\Phi, R, J)$ is a normal subgroup of $G_\pi (\Phi, R)$. In particular, $E_\pi (\Phi, R)$ is normal in $G_\pi (\Phi, R)$.
The second structure theorem characterizes all the subgroups of $G_\pi(\Phi, R)$ that are normalized by $E'_\sigma (\Phi, R)$.

\begin{thm}[{L. N. Vaserstein \cite{LV}}, E. Abe \cite{EA3}]\label{thm:Intro2}
    Let \(\Phi\) be an irreducible root system of rank $\geq 2$ and $R$ a commutative ring with unity. Assume $1/ 2 \in R$ if $\Phi \sim B_\ell, C_\ell, F_4$, and that $1/3 \in R$ and $R$ has no factor ring with two elements if $\Phi \sim G_2$. If $H$ is a subgroup of $G_\pi(\Phi, R)$ normalized by $E_\pi(\Phi, R)$, then there exists a unique ideal $J$ of $R$ such that 
    \[
        E_\pi(\Phi, R, J) \subset H \subset G_\pi (\Phi, R, J).
    \]
\end{thm}

Now, let $\sigma$ be the composition of a graph automorphism $\rho$ and a ring automorphism $\theta$ of $G_\pi (\Phi, R)$, where $\rho$ and $\theta$ have the same order. The twisted Chevalley group $G_{\pi, \sigma} (\Phi, R)$ is defined to be the set of all elements in $G_\pi (\Phi, R)$ that are fixed by the automorphism $\sigma$. 
Let $E'_{\pi, \sigma} (\Phi, R)$ be the subgroup of $G_{\pi, \sigma} (\Phi, R)$ generated by $x_{[\alpha]}(t)$ where $[\alpha] \in \Phi_\rho$ and $t \in R_{[\alpha]}$ (see Section~\ref{sec:TCG} for the notation).
Consider an ideal $J$ of $R$ that is invariant under $\theta$ (i.e., $\theta (J) \subset J$). 
The natural projection map $R \longrightarrow R/J$ induces a group homomorphism  
\[
\phi: G_{\pi, \sigma} (\Phi, R) \longrightarrow G_{\pi, \sigma} (\Phi, R/J).
\]  
Define the subgroups 
\[
G_{\pi, \sigma} (\Phi, J) \coloneq \ker(\phi) \quad \text{and} \quad G_{\pi, \sigma} (\Phi, R, J) \coloneq \phi^{-1} (Z(G_{\pi, \sigma} (\Phi, R/J))),
\]  
where $Z(G_{\pi, \sigma} (\Phi, R/J))$ is the center of $G_{\pi, \sigma} (\Phi, R/J)$.
The subgroup $G_{\pi, \sigma} (\Phi, J)$ of $G_{\pi, \sigma} (\Phi, R)$ is called a \emph{principal congruence subgroup} of level $J$, while $G_{\pi, \sigma} (\Phi, R, J)$ is referred to as a \emph{full congruence subgroup} of level $J$. 
Let $E'_{\pi, \sigma} (\Phi, J)$ be the subgroup of $E'_{\pi, \sigma} (\Phi, R)$ generated by elements $x_{[\alpha]}(t)$ for all $[\alpha] \in \Phi_\rho$ and $t \in R_{[\alpha]}$. Additionally, define $E'_{\pi, \sigma} (\Phi, R, J)$, called the \emph{elementary congruence subgroup} of level $J$, as the normal subgroup of $E'_{\pi, \sigma} (\Phi, R)$ generated by $E'_{\pi, \sigma} (\Phi, J)$. 

\medskip

The primary goal of this paper is to establish results analogous to Theorems~\ref{thm:Intro1} and \ref{thm:Intro2} for twisted Chevalley groups over commutative rings. Our main results are as follows.

\begin{thm}[Main Theorem 1]\label{mainthm1}
    Let $\Phi_\rho$ be one of the following types: ${}^2 A_n \ (n \geq 3), {}^2 D_n \ (n \geq 4), {}^2 E_6$, or ${}^3 D_4$. Assume that $1/2 \in R$ and in addition, $1/3 \in R$ if $\Phi_\rho \sim {}^3 D_4$.
    Let $J$ be a $\theta$-invariant ideal of $R$. Then 
    \begin{equation*}
        \begin{split}
            E'_{\pi,\sigma} (\Phi, R, J) = [E'_{\pi, \sigma} (\Phi, R), E'_{\pi,\sigma}(\Phi, J)] = [E'_{\pi,\sigma} (\Phi, R), G_{\pi,\sigma} (\Phi, R, J)] \\
            = [G_{\pi,\sigma} (\Phi, R), E'_{\pi,\sigma} (\Phi, R, J)].
        \end{split}
    \end{equation*}
\end{thm}

\begin{thm}[Main Theorem 2]\label{mainthm}
    Let $\Phi_\rho$ be one of the following types: ${}^2 A_n \ (n \geq 3), {}^2 D_n \ (n \geq 4), {}^2 E_6$, or ${}^3 D_4$. Assume that $1/2 \in R$, and in addition, $1/3 \in R$ if $\Phi_\rho \sim {}^3 D_4$. If $H$ is a subgroup of $G_{\pi, \sigma} (\Phi, R)$ normalized by $E'_{\pi, \sigma} (\Phi, R)$, then there exists a unique $\theta$-invariant ideal $J$ of $R$ such that 
    \[
        E'_{\pi, \sigma} (\Phi, R, J) \subset H \subset G_{\pi, \sigma} (\Phi, R, J).
    \]
\end{thm}

We acknowledge that Theorem~\ref{mainthm} was established by K. Suzuki~\cite{KS1} in the context of local rings. Our approach to proving these results is inspired by the works of L. N. Vaserstein~\cite{LV} and E. Abe~\cite{EA3}. However, the computations presented here are significantly more intricate, particularly in the case where $\Phi_\rho$ is of type ${}^2A_{2n} \ (n \geq 3)$.

The structure of the paper is as follows:  
Sections~\ref{sec:TCG}, \ref{sec:E(R)}, and \ref{sec:SKR} introduce definitions, fundamental properties, and key results related to Chevalley groups and twisted Chevalley groups.  
In Section~\ref{sec:E=UHV}, we examine certain subgroups of $G_{\pi, \sigma} (\Phi, R)$ and establish key properties of these subgroups.  
Section~\ref{sec:E(R,J)} investigates several important properties of the subgroup $E'_{\pi, \sigma} (\Phi, R, J)$ and provides a proof of Theorem~\ref{mainthm1}.
Sections~\ref{sec:Pf of main thm}, \ref{sec:Pf of prop 1}, \ref{sec:Pf of prop 2}, and \ref{sec:Pf of prop 3} are dedicated to proving Theorem~\ref{mainthm}.

Finally, the appendix presents an application of Theorem~\ref{mainthm1} by proving that $E'_{\pi, \sigma}(\Phi, R)$ is a characteristic subgroup of $G_{\pi, \sigma}(\Phi, R)$. 
This result holds great significance as it enables us to reduce the study of the automorphism group of $G_{\pi, \sigma} (\Phi, R)$ to the (presumably simpler) task of studying the automorphism group of $E'_{\pi, \sigma} (\Phi, R)$. 
While we had previously established this result in the case of Noetherian rings, the general version presented here is due to Pavel Gvozdevsky.

%%%%%%%%%%%%%%%%%%%%%%%%%%%%%%%%%%%%%%%%%%%%%%%%%%
%Section: Twisted Chevalley Groups
%%%%%%%%%%%%%%%%%%%%%%%%%%%%%%%%%%%%%%%%%%%%%%%%%%

\section{Twisted Chevalley Groups}\label{sec:TCG}

In this section, we provide a formal definition of twisted Chevalley groups. 
For a more comprehensive discussion of this topic, the reader is referred to \cite{EA1}, \cite{EB12:main}, \cite{RC}, \cite{RS} or \cite{RSTCG}.

%%%%%%%%%%%%%%%%%%%%%%%%%%%%%%%%%%%%%%%%%%%%%%%%%%

\subsection{Chevalley Groups}\label{Subsec:CG}

Let $\mathcal{L}$ be a semisimple Lie algebra over $\mathbb{C}$ with root system $\Phi$, and let its Cartan decomposition be given by
$$ 
\mathcal{L} = \mathcal{H} \oplus \coprod_{\alpha \in \Phi} \mathcal{L}_{\alpha}, 
$$
where $\mathcal{H}$ is a Cartan subalgebra of $\mathcal{L}$. 
Fix a simple system $\Delta$ of $\Phi$. Consider a Chevalley basis (for a precise definition, see, for instance, page 147 of \cite{JH} or page 7 of \cite{RS}) given by
$$
\{ H_i = H_{\alpha_i}, X_{\alpha} \mid \alpha_i \in \Delta, \, \alpha \in \Phi \}.
$$
Define $\mathcal{L}(\mathbb{Z})$ to be the $\mathbb{Z}$-span of this Chevalley basis.

Let $\mathcal{U} = \mathcal{U} (\mathcal{L})$ denote the universal enveloping algebra of $\mathcal{L}$, and let $\mathcal{U}_\mathbb{Z}$ denote the corresponding Kostant's $\mathbb{Z}$-form, generated by the elements $X_\alpha^m/m!$ for $m \in \mathbb{Z}_{\geq 0}$ and $\alpha \in \Phi$.

Consider a (finite-dimensional) faithful representation $\pi: \mathcal{L} \longrightarrow \text{GL}(V)$. 
This induces a natural action of $\mathcal{U}$ and, consequently, of $\mathcal{U}_\mathbb{Z}$ on $V$. 
Note that $V$ contains a lattice $M$ that is invariant under the action of $\mathcal{U}_\mathbb{Z}$, called an \textbf{admissible lattice}. 
Let $\mathcal{L}_{\pi} (\mathbb{Z})$ denote the stabilizer of $M$ in $\mathcal{L}$.
It follows that $\mathcal{L}_{\pi} (\mathbb{Z})$ is a lattice in $\mathcal{L}$ and can be expressed as 
$$
\mathcal{L}_{\pi} (\mathbb{Z}) = \mathcal{H}_{\pi} (\mathbb{Z}) \oplus \coprod_{\alpha \in \Phi} \mathbb{Z} X_{\alpha},
$$
where 
$$
\mathcal{H}_{\pi} (\mathbb{Z}) = \mathcal{H} \cap \mathcal{L}_{\pi} (\mathbb{Z}) = \{ H \in \mathcal{H} \mid \mu(H) \in \mathbb{Z} \text{ for all weights } \mu \text{ of the representation } \pi \}.
$$
This shows that $\mathcal{L}_{\pi} (\mathbb{Z})$ depends only on the weight lattice $\Lambda_\pi$, making the notation independent of the particular choice of $M$.

Let $R$ be a commutative ring with unity. 
Let $V(R) = M \otimes_\mathbb{Z} R$, $\mathcal{L} (R) = \mathcal{L} (\mathbb{Z}) \otimes_{\mathbb{Z}} R$ and $\mathcal{L}_{\pi} (R) = \mathcal{L}_{\pi} (\mathbb{Z}) \otimes_{\mathbb{Z}} R$.
Consider the automorphisms of $V(R)$ of the form $x_\alpha (t): = \exp {(t \pi (X_\alpha))} \ (t \in R, \alpha \in \Phi)$, where 
$$\exp {(t \pi(X_\alpha))} = \sum_{n=0}^{\infty} \frac{t^n \pi(X_\alpha)^n}{n!} .$$ 
The action of $x_\alpha (t)$ on $V(R)$ is the same as the action described in \cite[Chapter 3]{RS}. 
The subgroup of $\operatorname{Aut}(V(R))$ generated by all $x_\alpha(t) \ (t \in R, \alpha \in \Phi)$ is called an \textbf{elementary Chevalley group} and is denoted by $E_\pi (\Phi, R)$. 
For a representation $\pi$, let $\Lambda_\pi$ denote the weight lattice of $\pi$, i.e., the lattice generated by all weights of $\pi$. 
If $\pi$ and $\pi'$ are representations of $\mathcal{L}$ such that $\Lambda_\pi = \Lambda_{\pi'}$, then $E_\pi (\Phi, R) \cong E_{\pi'} (\Phi, R)$.
Let $\Lambda_r$ be the lattice generated by roots and $\Lambda_{sc}$ be the lattice generated by fundamental weights. 
If $\pi$ is such that $\Lambda_\pi = \Lambda_{r}$ (resp., $\Lambda_\pi = \Lambda_{sc}$), then $E_\pi(\Phi, R) = E_{\text{ad}}(\Phi, R)$ (resp., $E_\pi(\Phi, R) = E_{sc}(\Phi, R)$) is called an \textbf{adjoint elementary Chevalley group} (resp., \textbf{universal} (or \textbf{simply connected}) \textbf{elementary Chevalley group}). 

Let $U$ (resp., $U^{-}$) to be the subgroup of $E_\pi(\Phi, R)$ generated by all $x_\alpha(t), \alpha \in \Phi^+ \ (\text{resp., } \alpha \in \Phi^-), t \in R$. Let $H$ be the subgroup generated by all $h_\alpha (t) = w_\alpha (t) w_\alpha(1)^{-1},$ where $w_\alpha (t) = x_\alpha(t) x_{-\alpha}(-t^{-1}) x_\alpha(t), \ t \in R^*$ (the group of units in $R$). If $B$ is the subgroup generated by $U$ and $H$, then $U \cap H = 1$, $U$ is normal in $B$ and $B = UH$. Let $N$ be the subgroup generated by all $w_\alpha (t)$ and $W$ be the Weyl group $W(\Phi)$. Then $H$ is normal in $N$ and $W \cong N/H$ with the map $s_\alpha \mapsto H w_\alpha(1), \forall \alpha \in \Phi$. We sometimes use more precise notation such as $U_\pi (\Phi, R), U(\Phi, R)$ or $U(R)$ instead of just $U$. Similarly, this applies to $U^{-}, H, B$ and $N$.

Let $k$ be an algebraically closed field. Then the semisimple linear algebraic groups over the field $k$ are precisely the elementary Chevalley groups $E_\pi (\Phi, k)$ (see \cite[Chapter 5]{RS}). All these groups can be viewed as subgroups of $GL_n(k)$ defined as a common set of zeros of polynomials of matrix entries $x_{ij}$ with integer coefficients. Note that the multiplication map and the inverse map are also defined by polynomials with integer coefficients. Therefore, these polynomials can be considered as polynomials over an arbitrary commutative ring with unity. 

Let $E_\pi (\Phi, \mathbb{C})$ be an elementary Chevalley group viewed as a subgroup of $GL_n(\mathbb{C})$ defined by zero locus of polynomials $p_1(x_{ij}), \dots ,p_m(x_{ij})$. Note that these polynomials can be chosen to have integer coefficients. Let $R$ be a commutative ring with unity and let us consider the groups $$G(R) = \{ (a_{ij}) \in GL_n (R) \mid \widetilde{p}_1(a_{ij}) = 0, \dots, \widetilde{p}_m(a_{ij}) = 0\},$$
where $\widetilde{p}_1(x_{ij}), \dots , \widetilde{p}_m(x_{ij})$ are polynomials having the same coefficients as $p_1(x_{ij}), \dots ,p_m(x_{ij})$, but considered over a ring $R$. This group is called the \textbf{Chevalley group} $G_\pi (\Phi, R)$ of the type $\Phi$ over the ring $R$. If $\pi$ is a representation such that $\Lambda_{\pi} = \Lambda_r$ then $G_{\pi}(\Phi, R) = G_{ad}(\Phi, R)$ is called an \textbf{adjoint Chevalley group}.  If $\pi$ is a representation such that $\Lambda_\pi = \Lambda_{sc}$ then $G_\pi(\Phi, R) = G_{sc}(\Phi, R)$ is called a \textbf{universal} (or \textbf{simply connected}) \textbf{Chevalley group}.

Note that $E_\pi (\Phi, R) \subseteq G_\pi (\Phi, R)$. If $k$ is an algebraically closed field then $E_\pi (\Phi, k) = G_\pi (\Phi, k)$. But in general, equality may not hold (even for a field). 

%%%%%%%%%%%%%%%%%%%%%%%%%%%%%%%%%%%%%%%%%%%%%%%%%%

%\subsection{Standard Maximal Torus}

The subgroup of diagonal matrices (in the standard basis of weight vectors) of the Chevalley group $G_\pi (\Phi, R)$ is called the \textbf{standard maximal torus} of $G_\pi (\Phi, R)$ and it is denoted by $T_\pi (\Phi, R)$. 
This group is isomorphic to Hom$(\Lambda_\pi, R^*)$ where $R^*$ is the group of units in $R$ and the isomorphism is given as follows: Let $\chi \in $ Hom$(\Lambda_{\pi}, R^*)$ be a character of $\Lambda_\pi$. Let $V_\mu$ be a weight space corresponding to weight $\mu$ of $\pi$ and let $V_{\mu}(R) = (V_\mu \cap M) \otimes_\mathbb{Z} R$. 
Define an automorphism $h(\chi)$ of $V(R)$ given by $$ h(\chi) \cdot v = \chi (\mu) v, $$ where $\mu$ is a weight of $\pi$ and $v \in V_{\mu} (R)$. 
Note that, $h(\chi)$ can be extented to $V(R)$ as $V(R) = \coprod_{\mu \in \Omega_\pi} V_{\mu}(R),$ where $\Omega_\pi$ is the collection of weights corresponding to representation $\pi$. 
Therefore, $$T_\pi (\Phi, R) = \{ h(\chi) \mid \chi \in  \text{Hom}(\Lambda_\pi, R^*) \}.$$ 

Note that $H_\pi (\Phi, R)$ is contained in $T_{\pi}(\Phi, R)$. The element $h(\chi) \in H_\pi (\Phi, R) \subset E_\pi(\Phi, R)$ if and only if $\chi \in \text{Hom}(\Lambda_\pi, R^*)$ can be extented to a character $\chi'$ of $\Lambda_{sc}$, that is, $\chi' \in \text{Hom}(\Lambda_{sc}, R^*)$ such that $\chi'|_{\Lambda_{\pi}} = \chi$. Moreover, $h_\alpha(t) = h(\chi_{\alpha, t}) \ (t \in R^*, \alpha \in \Phi),$ where $$ \chi_{\alpha, t}: \lambda \mapsto t^{\langle \lambda, \alpha \rangle} \ (\lambda \in \Lambda_\pi).$$ Therefore $H_\pi (\Phi, R) = E_\pi (\Phi, R) \cap T_\pi(\Phi, R).$ Consider a subgroup $G_{\pi}^{0}(\Phi, R) = E_\pi (\Phi, R) T_\pi (\Phi, R)$ of $G_\pi(\Phi, R)$. If $R$ is a semilocal ring, then $G_\pi (\Phi, R) = G_\pi^0 (\Phi, R)$ (see \cite[Corollary 2.4]{EA2}). The element $h(\chi)$ acts on $\mathfrak{X}_{\alpha} = \{ x_{\alpha} (t) \mid t \in R \}$ by conjugation as follows: $$ h(\chi) x_\alpha (\zeta) h(\chi)^{-1} = x_\alpha (\chi (\alpha) \zeta).$$ 

A subgroup $H$ of a group $G$ is called \textit{characteristic}, if it is mapped into itself under any automorphism of $G$. In particular, any characteristic subgroup is normal. If the rank of $\Phi$ is $\geq 2$, then $E_\pi (\Phi, R)$ is a characteristic subgroup of $G_\pi (\Phi, R)$ (see \cite[Theorem 5]{LV}).

A group $G$ is said to be \textit{perfect} if $[G,G] = G$, where $[G,G]$ denotes the commutator subgroup of $G$. If the rank of $\Phi$ is $\geq 2$, then the elementary Chevalley group $E_\pi (\Phi, R)$ is perfect, i.e., $[E_\pi (\Phi, R), E_\pi (\Phi, R)]$ (see \cite[Theorem 5]{LV}). 

For abusive use of notations, we sometimes write $E(R)$ or $E(\Phi, R)$ instead of $E_\pi(\Phi, R)$, similar for $G_\pi (\Phi, R), G^0_\pi (\Phi, R)$ and $T_\pi (\Phi, R)$.

%%%%%%%%%%%%%%%%%%%%%%%%%%%%%%%%%%%%%%%%%%%%%%%%%%

\subsection{Twisted Root System}\label{Subsec:TRS}

Let $V$ be a finite-dimensional real Euclidean vector space and let $\Phi$ be a crystallographic root system. Let $\Delta$ and $\Phi^+$ be the simple and positive root systems, respectively, with respect to some fixed ordering on $V$. Let $\rho$ be a non-trivial angle preserving permutation of $\Delta$ (such a $\rho$ exists only when $\Phi$ is of type $A_l \ (l \geq 1), D_l \ (l \geq 4), E_6, B_2, F_4$ or $G_2$). Note that the possible order of $\rho$ is either $2$ or $3$, with the latter possible only when $\Phi$ is of type $D_4$. We define an isometry $\hat{\rho} \in GL (V)$ as follows:
\begin{enumerate}
    \item If $\Phi$ has one root length, then define $\hat{\rho}(\alpha)= \rho(\alpha)$ for each $\alpha \in \Delta$.
    \item If $\Phi$ has two root lengths. Then define $\hat{\rho}(\alpha)= \rho(\alpha)/ \sqrt{p}$ for each short root $\alpha \in \Delta$ and $\hat{\rho}(\alpha)= \sqrt{p} \hspace{1mm} \rho(\alpha)$ for each long root $\alpha \in \Delta$, where $p = ||\alpha||^2 / ||\beta||^2$, $\alpha$ is a long root and $\beta$ is a short root. 
\end{enumerate} 

Clearly, the order of $\hat{\rho}$ is the same as that of $\rho$ and $\hat{\rho}$ preserves the sign. Note that $\hat{\rho} w_{\alpha} \hat{\rho}^{-1} = w_{\rho(\alpha)}$, hence $\hat{\rho}$ normalizes $W$. Define $V_\rho = \{ v \in V \mid \hat{\rho}(v)=v \}$ and $W_\rho = \{ w \in W \mid \hat{\rho}w\hat{\rho}^{-1} =w \}$. Let $\hat{\alpha} = 1/o(\rho) \sum_{i=0}^{o(\rho)-1} \hat{\rho}^i (\alpha),$ the average of the elements in the $\hat{\rho}$-orbit of $\alpha$. Then $(\beta, \hat{\alpha})=(\beta, \alpha)$ for all $\beta \in V_\rho$. Hence the projection of $\alpha$ on $V_\rho$ is $\hat{\alpha}$. 

Note that $W_\rho$ acts faithfully on $V_\rho$. Let $J = J_\alpha \subset \Phi$ be the $\rho$-orbit of $\alpha$ and let $W_{J}$ be the group generated by all $w_\beta \ (\beta \in J_\alpha)$. Let $w_{J}$ be the unique element of $W_J$ such that $w_J (P_\alpha) = - P_\alpha$, where $P_\alpha$ is a positive system generated by $J_\alpha$ (such a $w_J$ exists and is of highest length element in $W_J$). Then $w_J|_{V_\rho} = w_{\hat{\alpha}}|_{V_\rho}$ and $w_J|_{V_\rho} \in W_{\rho}$. In fact, $\{ w_{\hat{\alpha}}|_{V_\rho} \mid \alpha \in \Delta \}$ forms a generating set of $W_\rho$. Therefore the group $W_\rho|_{V_\rho}$ is a reflection group. Define $\Tilde{\Phi}_\rho = \{ \hat{\alpha} \mid \alpha \in \Phi \}$ and $\Tilde{\Delta}_\rho = \{ \hat{\alpha} \mid \alpha \in \Delta \}$. Then $\Tilde{\Phi}_\rho$ is the (possibly non-reduced) root system corresponding to the Weyl group $W_\rho|_{V_\rho}$ and $\Tilde{\Delta}_\rho$ is the corresponding simple system. In order to make $\Tilde{\Phi}_\rho$ reduced, we can stick to the set of shortest projections of various directions, and denote it by $\Phi_\rho$. Define an equivalence relation $R$ on $\Phi$ by $\alpha \equiv \beta$ iff $\hat{\alpha}$ is a positive multiple of $\hat{\beta}$. If $\Phi/R$ denotes the collection of all equivalence classes of this relation, then $\Phi_\rho$ is in one-to-one correspondence with $\Phi / R$ by identifying a root $\hat{\alpha}$ of $\Phi_\rho$ with a class $[\alpha]$ of $\Phi / R$. Similarly, there exists a one-to-one correspondence between $\Tilde{\Phi}_\rho$ and $\{ J_\alpha \mid \alpha \in \Phi \}$ by sending a root $\hat{\alpha}$ of $\Tilde{\Phi}_\rho$ to $J_\alpha$. Clearly $-[\alpha] = [-\alpha]$ and $-J_{\alpha} = J_{-\alpha}.$

\begin{lemma}[{\cite[page 103]{RS}}]
    \normalfont
    If $\Phi$ is irreducible, then an element of $\Phi/R$ is the positive system of roots of a system of one of the following types:
    \begin{enumerate}[(a)]
        \item $A_1^n, \hspace{1mm} n=1, 2$ or $3$.
        \item $A_2$ (this occurs only if $\Phi$ is of type $A_{2n}$).
        \item $C_2$ (this occurs if $\Phi$ is of type $C_{2}$ or $F_4$).
        \item $G_2$ (this occurs only if $\Phi$ is of type $G_2$).
    \end{enumerate}
\end{lemma}

If a class $[\alpha]$ in $\Phi/R$ is the positive system of roots of a system of type $X$ (where $X$ is any of the above root systems) then we write $[\alpha] \sim X$. Similarly, if $\Phi \sim X$ then we write $\Phi_\rho \sim {}^n X$ where $n$ is the order of $\rho$. In the following table we describe some root systems $\Phi_\rho$ and $\Tilde{\Phi}_\rho$ after the twist:

\begin{center}
    \begin{tabular}{|c|c|c|c|c|}
        \hline
        \multirow{2}{*}{\textbf{Type}} & \multirow{2}{*}{\textbf{$\Tilde{\Phi}_\rho$}} & \multirow{2}{*}{\textbf{$\Phi_\rho$}} & \multicolumn{2}{c|}{\textbf{Type of Roots}} \\
        \cline{4-5}
        & & & \textbf{Long} & \textbf{Short}  \\
        \hline 
        ${}^2 A_{2n-1} \ (n \geq 2)$ & $C_n$ & $C_n$ & $A_1$ & $A_1^2$ \\
        \hline 
        ${}^2 A_{2n} \ (n \geq 2)$ & $BC_n$ & $B_n$ & $A_1^2$ & $A_2$ \\
        \hline 
        ${}^2 D_{n} \ (n \geq 4)$ & $B_{n-1}$ & $B_{n-1}$ & $A_1$ & $A_1^2$ \\
        \hline
        ${}^3 D_{4}$ & $G_2$ & $G_2$ & $A_1$ & $A_1^3$ \\
        \hline
        ${}^2 E_{6}$ & $F_4$ & $F_4$ & $A_1$ & $A_1^2$ \\
        \hline
    \end{tabular}
\end{center}

Finally, let us discuss the action of $\rho$ on the weight lattice $\Lambda_{sc}$. Assume that $\Phi$ has one root length. Since $\rho$ permutes simple roots (hence all roots), the action of $\rho$ on root lattice $\Lambda_r$ is clear. The fundamental dominant weights $\lambda_1, \dots, \lambda_l$ forms a $\mathbb{Z}$-basis of the weight lattice $\Lambda_{sc}$. We can define the action of $\rho$ on $\lambda_i$ by $\rho (\lambda_i) = \lambda_j$ if $\rho (\alpha_i) = \alpha_j$. This action can be naturally extended to a $\mathbb{Z}$-linear automorphism $\rho$ of $\Lambda_{sc}$ such that $\rho (\Lambda_{r}) = \Lambda_r$. Thus $\rho$ can be thought as a group automorphism of the fundamental group $\Lambda_{sc}/\Lambda_r$ of $\Phi$. Now let $\Lambda$ be a sublattice of $\Lambda_{sc}$ which contains $\Lambda_r$. Then $\Lambda / \Lambda_r$ is a subgroup of $\Lambda_{sc} / \Lambda_r$ which is cyclic except for the case of $\Phi = D_{2n}$. Therefore $\rho(\Lambda / \Lambda_r) = \Lambda / \Lambda_r$ and hence $\rho (\Lambda) = \Lambda$. For the case of $\Phi = D_{2n}$, the fundamental group $\Lambda_{sc} / \Lambda_r$ is isomorphic to $\mathbb{Z}_2 \times \mathbb{Z}_2$. Hence there are exactly two proper sublattices $\Lambda_1$ and $\Lambda_2$ of $\Lambda_{sc}$ which contains $\Lambda_r$ as proper sublattice with the property that $\rho (\Lambda_i) \not\subset \Lambda_i$ for $i=1,2$. Therefore, if $\Lambda_\pi = \Lambda_1$ or $\Lambda_2$, then the graph automorphism of $G_\pi (\Phi, R)$ and $E_\pi (\Phi, R)$ do not exist even when $1/2 \in R$ (see \cite[page 91]{RS}).

%%%%%%%%%%%%%%%%%%%%%%%%%%%%%%%%%%%%%%%%%%%%%%%%%%

\subsection{Twisted Chevalley Groups}\label{Subsec:TCG}

Assume that $\Phi$ is of type $A_n (n \geq 2), D_n (n \geq 4)$ or $E_6$ and let $G(R) = G_\pi (\Phi, R)$ (resp., $E(R) = E_\pi (\Phi, R)$) be a Chevalley group (resp., an elementary Chevalley group) over a commutative ring $R$. Let $\sigma$ be an automorphism of $G(R)$ which is the product of a graph automorphism $\rho$ and a ring automorphism $\theta$ such that $o(\theta) = o(\rho)$. Denote the corresponding permutation of the roots also by $\rho$. Since $\rho \circ \theta = \theta \circ \rho$, we have $o(\theta) = o(\rho) = o(\sigma)$. Since $E(R)$ is a characteristic subgroup of $G(R)$, $\sigma$ is also an automorphism of $E(R)$. 

Define $G_\sigma (R) = \{ g \in G(R) \mid \sigma(g)=g \}$. Clearly, $G_\sigma (R)$ is a subgroup of $G(R)$. We call $G_\sigma (R)$ the \textbf{twisted Chevalley group} over the ring $R$. The notion of the adjoint twisted Chevalley group and the universal (or simply connected) twisted Chevalley group is clear.

Write $E_\sigma (R) = E(R) \cap G_\sigma (R)$. Consider the subgroups $U, H, B, U^-$ and $N$ of $E(R)$. Then $\sigma$ preserves $U, H, B, U^-$ and $N$. Hence we can make sense of $U_\sigma, H_\sigma, B_\sigma, U^-_\sigma$ and $N_\sigma$ (if $A \subset G(R)$ then we define $A_\sigma = A \cap G_\sigma(R)$). Note that $\sigma$ preserves $N/H \cong W$ (as it preserves $N$ and $H$). The action thus induced on $W$ is concordant with the permutation $\rho$ of the roots. Finally, let us define $E_\sigma' (R) = \langle U_\sigma, U_\sigma^- \rangle$, a subgroup of $E_\sigma (R)$ generated by $U_\sigma$ and $U_\sigma^-$. We call $E'_\sigma (R)$ the \textbf{elementary twisted Chevalley group} over the ring $R$. Write $H'_\sigma = H \cap E'_\sigma(R), N'_\sigma = N \cap E'_\sigma(R)$ and $B'_\sigma = B \cap E'_\sigma (R)$. Then $B'_\sigma = U_\sigma H'_\sigma$. 

Let $T(R) = T_\pi (\Phi, R)$ be the standard maximal torus of $G(R)$. Define $T_\sigma (R) = T(R) \cap G_\sigma (R)$ and called it the \textbf{standard maximal torus} of $G_\sigma (R).$ 
For a character $\chi \in $ Hom$(\Lambda_\pi, R^*)$, we define its conjugation $\bar{\chi}_\sigma$ with respect to $\sigma$ by $\bar{\chi}_\sigma (\lambda) = \theta (\chi (\rho^{-1}(\lambda)))$ for every $\lambda \in \Lambda_\pi$. 
If $h(\chi) \in T(R)$, then $\sigma (h(\chi)) = h(\bar{\chi}_\sigma)$.
A character $\chi \in $ Hom$(\Lambda_\pi, R^*)$ is called \textbf{self-conjugate with respect to $\sigma$} if $\chi = \bar{\chi}_\sigma$, i.e., $\chi (\rho(\lambda)) = \theta (\chi (\lambda)),$ for every $\lambda \in \Lambda_\pi$. We denote the set of such characters by $\text{Hom}_1 (\Lambda_\pi, R^*) = \{ \chi \in \text{Hom} (\Lambda_\pi, R^*) \mid \chi = \bar{\chi}_\sigma \}$.
Thus we have $T_\sigma (R) = \{ h(\chi) \mid \chi \in \text{Hom}_1 (\Lambda_\pi, R^*) \}$. Note that, an element $h(\chi) \in H_\sigma \subset T_\sigma(R)$ if and only if $\chi$ is a self-conjugate character of $\Lambda_\pi$ (with respect to $\sigma$) that can be extended to a self-conjugate character of $\Lambda_{sc}$.

For the sake of completeness, let us also define $G_\sigma^0 (R) = G^0_\pi(\Phi, R) \cap G_\sigma (R)$ and $G'_\sigma (R) = T_\sigma (R) E'_\sigma (R)$. 

If $G(R)$ is of type $X$ and $\sigma$ is of order $n$, we say $G_\sigma(R)$ is of type $^nX.$ We write $G(R) \sim X$ and $G_\sigma(R) \sim {}^nX$. We use a similar notation for $E(R), E_\sigma(R)$ and $E'_\sigma(R)$. 

\begin{rmk}
    Sometimes we use more detailed notations such as $G_{\pi, \sigma} (\Phi, R)$ or $G_{\sigma} (\Phi, R)$ to refer to the group $G_\sigma(R)$. This convention similarly applies to other groups described above.
\end{rmk}

%%%%%%%%%%%%%%%%%%%%%%%%%%%%%%%%%%%%%%%%%%%%%%%%%%

\subsection{Identifying Certain Twisted and Untwisted Chevalley Groups}\label{subsec:Chevalley as Twisted Chevalley}

This subsection will examine isomorphisms between certain Chevalley and twisted Chevalley groups.

Let \(\Phi\) be a root system of type \(A_n \ (n \geq 2)\), \(D_n \ (n \geq 4)\), or \(E_6\), and let \(\Delta\) be the corresponding simple system. Let \(\rho\) be the non-trivial angle-preserving permutation of the simple roots of \(\Phi\). Let \(R\) be a commutative ring with unity. If \(o(\rho) = 2\) (respectively, \(o(\rho) = 3\)), define the ring automorphism \(\theta: R \times R \longrightarrow R \times R\) by \((a,b) \mapsto (b,a)\) (respectively, \(\theta: R \times R \times R \longrightarrow R \times R \times R\) by \((a,b,c) \mapsto (b,c,a)\)).

Next, consider the automorphisms of the group \(G(R \times R)\) (resp., \(G(R \times R \times R)\)) induced by \(\rho\) and \(\theta\), which we shall also denote by the same symbols. Set \(\sigma = \rho \circ \theta\). Now, consider the twisted Chevalley group \(G_\sigma(R \times R)\) (resp., \(G_\sigma(R \times R \times R)\)).

\begin{prop}\label{prop:Chevalley as TChevalley}
    \normalfont
    Retaining the above notations, we establish the following isomorphisms:
    \begin{enumerate}[(a)]
        \item $G_\sigma(R \times R) \cong G(R)$ (respectively, $G_\sigma(R \times R \times R) \cong G(R)$). 
        \item $E'_\sigma(R \times R) \cong E(R)$ (respectively, $E'_\sigma(R \times R \times R) \cong E(R)$).
    \end{enumerate}
\end{prop}

\begin{proof}
    We shall prove the isomorphism \(G_\sigma(R \times R) \cong G(R)\). The remaining assertions can be established in a similar manner and are therefore omitted.

    Let \(x \in G_\sigma(R \times R)\). Since \(G_\sigma(R \times R) \subseteq G(R \times R) \cong G(R) \times G(R)\), there exist elements \(x_1, x_2 \in G(R)\) such that \(x\) corresponds to the pair \((x_1, x_2)\). By definition, \(x\) satisfies \(\sigma(x) = x\), which implies \((\rho \circ \theta)(x) = x\). Therefore, we have \(\theta(x) = \rho^{-1}(x) = \rho(x)\), meaning \( \theta((x_1, x_2)) = \rho((x_1, x_2)) \).

    Note that, $\theta (x) = \theta (x_1, x_2) = (x_2, x_1)$ and the permutation $\rho$ of simple roots induces an automorphism of $G(R)$ such that $\rho ((x_1, x_2)) = (\rho (x_1), \rho(x_2))$.
    Therefore, we have \[ (x_2, x_1) = \theta((x_1, x_2)) = \rho((x_1, x_2)) = (\rho(x_1), \rho(x_2)). \] 
    This implies that $x_2 = \rho(x_1)$, and hence $x = (x_1, \rho(x_1))$. Conversely, for any $x_1 \in G(R)$, the element $x = (x_1, \rho(x_1))$ belongs to $G_\sigma (R \times R)$. Thus, the map $\phi: G(R) \longrightarrow G_\sigma (R \times R)$ defined by $x \mapsto (x, \rho(x))$ establishes the desired isomorphism of groups.
\end{proof}

%%%%%%%%%%%%%%%%%%%%%%%%%%%%%%%%%%%%%%%%%%%%%%%%%%
%Section: Some Known Results
%%%%%%%%%%%%%%%%%%%%%%%%%%%%%%%%%%%%%%%%%%%%%%%%%%

\section{The Group \texorpdfstring{$E'_\sigma (R)$}{E(R)}}\label{sec:E(R)}

Let us establish some notations before proceeding further. Denote $\bar{\alpha}, \bar{\bar{\alpha}}, \bar{t}, \bar{\bar{t}}$ as $\rho (\alpha)$, $\rho^2(\alpha)$, $\theta(t)$, $\theta^2(t)$, respectively. Write $R_\theta= \{ t \in R \mid t=\bar{t} \}$. Recall that for $\sigma$ as defined earlier, $\sigma (x_\alpha (t)) = x_{\bar{\alpha}}(\epsilon_\alpha \bar{t})$ holds for all $\alpha \in \Phi$. We can conveniently select a Chevalley basis that fully specifies the values of $\epsilon_\alpha$ for all $\alpha \in \Phi$.

The angle preserving permutation $\rho$ of $\Phi$, induced an automorphism of $\mathcal{L}$ (also denote by $\rho$) such that $$ \rho (H_\alpha) = H_{\bar{\alpha}}, \hspace{5mm} \rho (X_{\alpha}) = X_{\bar{\alpha}}, \hspace{5mm} \rho (X_{-\alpha}) = X_{-\bar{\alpha}}$$ for all $\alpha \in \Delta.$ Then we have $\rho (X_\alpha) = \epsilon_\alpha X_{\bar{\alpha}}$, where $\epsilon_\alpha = \pm 1$ for any $\alpha \in \Phi$ (these $\epsilon_\alpha$s coincide with those mentioned above, for example see the proof of Theorem 29 of \cite{RS}).

\begin{lemma}[{\cite[Proposition 3.1]{EA1}}]\label{epsilonalpha}
    \normalfont
    We can choose a Chevalley basis of $\mathcal{L}$ which satisfies the following conditions
    \begin{enumerate}[(a)]
        \item $\epsilon_\alpha = \epsilon_{\bar{\alpha}}$; 
        \item $\epsilon_\alpha = -1$, if $[\alpha] \sim A_2$ and $\alpha = \bar{\alpha}$;
        \item $\epsilon_\alpha = 1$, otherwise.
    \end{enumerate}
\end{lemma}

We shall always fix a Chevalley basis of $\mathcal{L}$ with the above properties. Therefore the image of $x_\alpha (t)$ under the map $\sigma$ is precisely given as follows:
\[
    \sigma (x_\alpha (t)) = \begin{cases}
        x_{\bar{\alpha}}(-\bar{t}) & \text{if } [\alpha] \sim A_2 \text{ and } \alpha = \bar{\alpha}, \\
        x_{\bar{\alpha}}(\bar{t}) & \text{otherwise}.
    \end{cases}
\]

%%%%%%%%%%%%%%%%%%%%%%%%%%%%%%%%%%%%%%%%%%%%%%%%%%

\subsection{The Generators of \texorpdfstring{$E'_\sigma(R)$}{E(R)}}

We fix a total order on $\Phi$ defined by $\Delta$. 
For a class $[\beta] \in \Phi_\rho$, we choose a representative $\alpha$ such that $\alpha$ is the smallest element in this class. 
Whenever $\alpha$ satisfies this condition, we denote the class $[\beta]$ by $[\alpha]$.
Now we define some special elements of $E'_\sigma (R)$ as follows: 
\begin{enumerate}
    \item If $[\alpha] \sim A_1$ (that is, $[\alpha]=\{ \alpha \}$), then define $x_{[\alpha]}(t) = x_\alpha (t), t \in R_\theta.$ In this case, $x_{[\alpha]}(t)x_{[\alpha]}(u)=x_{[\alpha]}(t+u)$ for every $t,u \in R_\theta$.

    \item If $[\alpha] \sim A_1^2$ (that is, $[\alpha]=\{ \alpha, \bar{\alpha} \}$), then define $x_{[\alpha]}(t) = x_\alpha (t) \ x_{\bar{\alpha}} (\bar{t}), t \in R$. In this case, $x_{[\alpha]}(t)x_{[\alpha]}(u)=x_{[\alpha]}(t+u)$ for every $t,u \in R$. 
    
    \item If ${[\alpha]} \sim A_1^3$ (that is, $[\alpha]=\{ \alpha, \bar{\alpha}, \bar{\bar{\alpha}} \}$), then define $x_{[\alpha]}(t) = x_\alpha (t) \ x_{\bar{\alpha}} (\bar{t}) \ x_{\bar{\bar{\alpha}}} (\bar{\bar{t}}), \ t \in R.$ 
    In this case, $x_{[\alpha]}(t)x_{[\alpha]}(u)=x_{[\alpha]}(t+u)$ for every $t,u \in R$. 
    
    \item If ${[\alpha]} \sim A_2$ with $\alpha \neq \bar{\alpha}$ (that is, ${[\alpha]}=\{ \alpha, \bar{\alpha}, \alpha + \bar{\alpha} \}$), then define $$x_{[\alpha]}(t,u) = x_\alpha (t) x_{\bar{\alpha}} (\bar{t}) x_{\alpha + \bar{\alpha}}(N_{\bar{\alpha}, \alpha} u), \hspace{1mm} \text{where } t,u \in R \text{ such that } t \bar{t} = u + \bar{u}.$$ 
    In this case, $x_{[\alpha]}(t,u)x_{[\alpha]}(t',u')=x_{[\alpha]}(t+t', u+u' + \bar{t} t')$ for every $t,u,t',u' \in R$ such that $t \bar{t} = u + \bar{u}$ and $t' \bar{t'} = u' + \bar{u'}$.
\end{enumerate}

Define $\mathcal{A}(R) := \{ (t,u) \mid t, u \in R \text{ such that } t \bar{t} = u + \bar{u} \}$. Note that, for $[\alpha] \sim A_2$ we define $x_{[\alpha]} (t,u)$ only in the case of $(t,u) \in \mathcal{A}(R)$. The product of $x_{[\alpha]}(t,u)$ and $x_{[\alpha]}(t',u')$ suggest the operation on $\mathcal{A}(R)$ as follows: let $(t,u), (t',u') \in \mathcal{A}(R)$, then define an operation $\oplus$ on $\mathcal{A}(R)$ by $(t,u) \oplus (t',u') = (t + t', u+u'+\bar{t}t').$ With this operation $\mathcal{A}(R)$ becomes a group with $(0,0)$ as an identity and $(-t, \bar{u})$ as an inverse of $(t,u)$. From this we can say that $(x_{[\alpha]}(t,u))^{-1} = x_{[\alpha]} (-t, \bar{u})$. Further, we can define an action of the monoid $(R, \times)$ on the set $\mathcal{A}(R)$ by $$r \cdot (t,u) = (rt, r \bar{r} u)$$ for any $r \in R$ and $(t,u) \in \mathcal{A}(R)$. (For more information, please refer to \cite{EA1}.)

For $[\alpha] \in \Phi_\rho$, we write 
\begin{align*}
    R_{[\alpha]} = \begin{cases}
        R_\theta & \text{if } [\alpha] \sim A_1, \\
        R & \text{if } [\alpha] \sim A_1^2 \text{ or } A_1^3, \\
        \mathcal{A}(R) & \text{if } [\alpha] \sim A_2.
    \end{cases}
\end{align*}
If $[\alpha] \sim A_2$ then $t \in R_{[\alpha]}$ means that $t$ is a pair $(t_1, t_2)$ such that $(t_1, t_2) \in \mathcal{A}(R)$. Additionally, for $r \in R$ and $t \in R_{[\alpha]}$, the notation $r \cdot t$ means $rt$ if $[\alpha] \sim A_1, A_1^2, A_1^3$, and it means $(rt_1, r \bar{r} t_2)$ if $[\alpha] \sim A_2$.

Define $\mathfrak{X}_{[\alpha]}$ to be the subset of $E_\sigma'(R)$ consists of all $x_{[\alpha]}(t), t \in R_{[\alpha]}$. Clearly, $\mathfrak{X}_{[\alpha]}$ is a subgroup of $E_\sigma'(R)$. For a subset $S$ of $\Phi_\rho$, define $\mathfrak{X}_S$ to be the subgroup generated by $\mathfrak{X}_{[\alpha]}, \ [\alpha] \in S$.

\begin{prop}\label{lemma:X_alpha and R_alpha}
    \normalfont
    The subgroup $\mathfrak{X}_{[\alpha]}$ is isomorphic to the additive group $R_{[\alpha]}$.
\end{prop}

The proof of this proposition is straightforward. Next, we present the following result from Steinberg \cite{RS}.

\begin{prop}[{\cite[Lemma 62]{RS}}]\label{lemma 62 of RS}
    \normalfont
    Let $S$ be a closed subset of $\Phi_\rho$, i.e., if $[\alpha], [\beta] \in S$ then $[\alpha] + [\beta] \in S$. Moreover, assume that if $[\alpha] \in S$ then $-[\alpha] \notin S$. Then $\mathfrak{X}_S = \prod_{[\alpha] \in S} \mathfrak{X}_{[\alpha]}$ with the product taken in any fixed order and there is uniqueness of expression on the right. In particular, $U_\sigma = \prod_{[\alpha] > 0} \mathfrak{X}_{[\alpha]}$.
\end{prop}

\begin{cor}
    \normalfont
    The group $E'_\sigma (R)$ is generated by $x_{[\alpha]}(t)$ for all $[\alpha] \in \Phi_\rho$ and $t \in R_{[\alpha]}$.
\end{cor}

The proof of this corollary follows directly from the definition of $E'_\sigma (R)$ and the preceding propositions. 
Finally, we highlight another useful consequence of Proposition~\ref{lemma 62 of RS}, which is also stated in Steinberg \cite{RS}.

\begin{cor}\label{Che-Comm}
    \normalfont
    Let $[\alpha], [\beta] \in \Phi_\rho$ be such that $[\alpha] \neq \pm [\beta]$. Then $[\mathfrak{X}_{[\alpha]}, \mathfrak{X}_{[\beta]}] \subset \mathfrak{X}_S$, where $S = \{ i[\alpha] + j[\beta] \in \Phi_\rho \mid i,j \in \frac{1}{2} \mathbb{Z}_{> 0} \}$.
\end{cor}

\begin{rmk}
    The precise commutator relations will be given in \ref{subsec:CheComm}.
\end{rmk}

%%%%%%%%%%%%%%%%%%%%%%%%%%%%%%%%%%%%%%%%%%%%%%%%%%

\subsection{Types of Pairs of Roots in \texorpdfstring{$\Phi_\rho$}{Phi rho}}

We aim to categorize pairs of roots $[\alpha]$ and $[\beta] (\neq \pm [\alpha])$ in $\Phi_\rho$ according to their placement within the subsystem generated by them. The positions of $[\alpha]$ and $[\beta]$ can be fully determined by considering their possible sums and differences. Hence we classify them into the following types (see \cite{EA1}):
\begin{enumerate}[(a)]
    \item $[\alpha], [\beta] \in \Phi_\rho$, but $[\alpha]+[\beta],[\alpha]-[\beta] \notin \Phi_\rho$.
        \begin{enumerate}
            \item[(a$_1$)] $[\alpha]+[\beta],[\alpha]-[\beta] \notin \Tilde{\Phi}_\rho$.
            \item[(a$_2$)] $[\alpha]+[\beta],[\alpha]-[\beta] \in \Tilde{\Phi}_\rho$.
                \begin{enumerate}[(i)]
                    \item $[\alpha] \sim A_1, [\beta] \sim A_1, 1/2([\alpha]+[\beta])=[\gamma] \sim A_1^2,$ where $\gamma + \bar{\gamma} = \alpha + \beta \notin \Phi$;
                    \item $[\alpha] \sim A_1^2, [\beta] \sim A_1^2, 1/2([\alpha]+[\beta])=[\gamma] \sim A_2,$ where $\gamma + \bar{\gamma} = \alpha + \beta \text{ or } \alpha + \bar{\beta} \in \Phi$.
                \end{enumerate}
        \end{enumerate}
        
    \item $[\alpha], [\beta] \in \Phi_\rho, [\alpha]+[\beta] \in \Phi_\rho,$ but $[\alpha]-[\beta] \notin \Phi_\rho$.
        \begin{enumerate}[(i)]
            \item $[\alpha] \sim A_1, [\beta] \sim A_1, [\alpha]+[\beta]=[\alpha + \beta] \sim A_1;$
            \item $[\alpha] \sim A_1^2, [\beta] \sim A_1^2, [\alpha]+[\beta] = [\alpha + \beta] \text{ or } [\alpha + \bar{\beta}] \text{ or } [\bar{\alpha} + \beta] \sim A_1^2.$
        \end{enumerate}
        
    \item $[\alpha], [\beta] \in \Phi_\rho, [\alpha]+[\beta] \in \Phi_\rho, [\alpha]-[\beta] \in \Phi_\rho$.
        \begin{enumerate}[(i)]
            \item $[\alpha] \sim A_1^2, [\beta] \sim A_1^2, [\alpha]+[\beta]=[\alpha + \bar{\beta}] \sim A_1;$
            \item $[\alpha] \sim A_2, [\beta] \sim A_2, [\alpha]+[\beta]=[\alpha + \bar{\beta}] \text{ or } [\bar{\alpha} + \beta] \sim A_1^2.$
        \end{enumerate}
        
    \item $[\alpha], [\beta] \in \Phi_\rho, [\alpha]+[\beta] \in \Phi_\rho, [\alpha]+ 2[\beta] \in \Phi_\rho$.
        \begin{enumerate}[(i)]
            \item $[\alpha] \sim A_1, [\beta] \sim A_1^2, [\alpha]+[\beta]=[\alpha + \beta] \sim A_1^2, [\alpha]+2[\beta]=[\alpha + \beta + \bar{\beta}] \sim A_1;$
            \item $[\alpha] \sim A_1^2, [\beta] \sim A_2, [\alpha]+[\beta]= [\alpha + \beta] \text{ or }[\alpha + \bar{\beta}] \text{ or } [\bar{\alpha} + \beta] \sim A_2, [\alpha]+2[\beta]=[\alpha + \beta + \bar{\beta}] \sim A_1^2.$
        \end{enumerate}
        
    \item $[\alpha], [\beta] \in \Phi_\rho, [\alpha]+[\beta] \in \Phi_\rho, [\alpha]+ 2[\beta] \in \Phi_\rho, [\alpha]+ 3[\beta] \in \Phi_\rho, 2[\alpha]+ 3[\beta] \in \Phi_\rho$. 
        \begin{itemize}
            \item $[\alpha] \sim A_1, [\beta] \sim A_1^3, [\alpha]+[\beta]= [\alpha + \beta] \sim A_1^3, [\alpha]+2[\beta]=[\alpha + \beta + \bar{\beta}] \text{ or } [\alpha + \beta + \bar{\bar{\beta}}]\sim A_1^3, [\alpha]+3[\beta]=[\alpha + \beta + \bar{\beta} + \bar{\bar{\beta}}] \sim A_1, 2[\alpha]+3[\beta]=[2\alpha + \beta + \bar{\beta} + \bar{\bar{\beta}}] \sim A_1.$
        \end{itemize}
        
    \item $[\alpha], [\beta] \in \Phi_\rho, [\alpha]+[\beta] \in \Phi_\rho, 2[\alpha]+ [\beta] \in \Phi_\rho,  [\alpha]+ 2[\beta] \in \Phi_\rho, [\alpha]-[\beta] \in \Phi_\rho.$
        \begin{itemize}
            \item $[\alpha] \sim A_1^3, [\beta] \sim A_1^3, [\alpha]+[\beta] = [\alpha + \bar{\beta}] \text{ or } [\alpha + \bar{\bar{\beta}}] \sim A_1^3, 2[\alpha]+[\beta]=[\alpha + \bar{\alpha} + \bar{\bar{\beta}}] \sim A_1, [\alpha]+2[\beta]=[\alpha + \bar{\beta} + \bar{\bar{\beta}}] \sim A_1, [\alpha]-[\beta] = [\alpha - \beta] \sim A_1.$
        \end{itemize}
        
    \item $[\alpha], [\beta] \in \Phi_\rho, [\alpha]+[\beta] \in \Phi_\rho, [\alpha]-[\beta] \in \Phi_\rho, 2[\alpha]- [\beta] \in \Phi_\rho,  [\alpha]- 2[\beta] \in \Phi_\rho.$
        \begin{itemize}
            \item $[\alpha] \sim A_1^3, [\beta] \sim A_1^3, [\alpha]+[\beta] = [\alpha + \bar{\beta}] \text{ or } [\alpha + \bar{\bar{\beta}}] \sim A_1, [\alpha]-[\beta] = [\alpha - \bar{\beta}] \text{ or } [\alpha - \bar{\bar{\beta}}] \sim A_1^3, 2[\alpha]-[\beta]=[\alpha + \bar{\alpha} - \beta] \text{ or } [\alpha + \bar{\alpha} - \bar{\beta}] \sim A_1, [\alpha]-2[\beta]= [\bar{\alpha} - \beta -\bar{\beta}] \text{ or } [\alpha - \beta - \bar{\beta}] \sim A_1.$
        \end{itemize}
\end{enumerate}

\begin{rmk}
    Cases $(e), (f)$ and $(g)$ occur exclusively in the context of ${}^3 D_4$.
\end{rmk}

%%%%%%%%%%%%%%%%%%%%%%%%%%%%%%%%%%%%%%%%%%%%%%%%%%

\subsection{Chevalley Commutator Formulas}\label{subsec:CheComm} 

We now present pivotal formulas that will serve as a cornerstone for this paper. The numbering in the following formulas reflects the types of root pairs mentioned in the preceding subsection. 

\begin{enumerate}
    \item[\textbf{(a$_1$)}] $[x_{[\alpha]}(t),x_{[\beta]}(u)]=1$, where $t \in R_{[\alpha]}$ and $u \in R_{[\beta]}$.
    \item[\textbf{(a$_2-$i)}] $[x_{[\alpha]}(t),x_{[\beta]}(u)]=1$, where $t,u \in R_\theta$.
    \item[\textbf{(a$_2-$ii)}] $[x_{[\alpha]}(t),x_{[\beta]}(u)]= x_{[\gamma]}(0, N_{\alpha,\beta} N_{\bar{\gamma},\gamma}(t \bar{u} - \bar{t} u))$, where $1/2([\alpha]+[\beta]) = [\gamma] =\{ \gamma, \bar{\gamma}, \gamma + \bar{\gamma} \}$ and $t,u \in R$.
    \item[\textbf{(b$-$i)}] $[x_{[\alpha]}(t),x_{[\beta]}(u)]= x_{[\alpha]+[\beta]}(N_{\alpha,\beta}tu)$, where $t,u \in R_\theta$.
    \item[\textbf{(b$-$ii)}] $[x_{[\alpha]}(t),x_{[\beta]}(u)]= x_{[\alpha]+[\beta]}(N_{\alpha,\beta}tu)$ or $x_{[\alpha]+[\beta]}(N_{\alpha,\bar{\beta}} t \bar{u})$ or $x_{[\alpha]+[\beta]}(N_{\bar{\alpha},\beta} \bar{t} u)$, where $t,u \in R$.
    \item[\textbf{(c$-$i)}] $[x_{[\alpha]}(t),x_{[\beta]}(u)]= x_{[\alpha]+[\beta]}(N_{\alpha,\bar{\beta}}(t \bar{u} + \bar{t} u))$, where $t,u \in R$.
    \item[\textbf{(c$-$ii)}] $[x_{[\alpha]}(t_1, t_2),x_{[\beta]}(u_1, u_2)]= x_{[\alpha]+[\beta]}(N_{\alpha,\bar{\beta}} t_1 \bar{u_1})$ or $x_{[\alpha]+[\beta]}(N_{\bar{\alpha},\beta} \bar{t_1} u_1)$, where \\ $(t_1, t_2), (u_1, u_2) \in \mathcal{A}(R)$.
    \item[\textbf{(d$-$i)}] $[x_{[\alpha]}(t),x_{[\beta]}(u)]= x_{[\alpha]+[\beta]}(N_{\alpha,\beta}tu) x_{[\alpha]+2[\beta]}(N_{\alpha,\beta} N_{\beta,\alpha + \bar{\beta}} t u \bar{u})$, where $t \in R_\theta, u \in R$.
    \item[\textbf{(d$-$ii)}] 
        \begin{align*} 
                [x_{[\alpha]}(t),x_{[\beta]}(u_1, u_2)] = \begin{cases}
                    x_{[\alpha]+[\beta]}((N_{\alpha,\beta}t u_1, t \bar{t} u_2))x_{[\alpha]+2[\beta]}(N_{\beta,\bar{\beta}} N_{\beta+ \bar{\beta},\alpha} t u_2),  \text{ or }  \\
                    x_{[\alpha]+[\beta]}((N_{\alpha,\bar{\beta}} t \bar{u_1}, t \bar{t} \bar{u_2}))x_{[\alpha]+2[\beta]}(N_{\beta,\bar{\beta}} N_{\beta+ \bar{\beta},\alpha} t \bar{u_2}), \text{ or } \\
                    x_{[\alpha]+[\beta]}((N_{\bar{\alpha},\beta} \bar{t} u_1, t \bar{t} u_2))x_{[\alpha]+2[\beta]}(N_{\beta,\bar{\beta}} N_{\beta+ \bar{\beta},\alpha} t \bar{u_2});
                \end{cases}
        \end{align*}
        where $t \in R$ and $(u_1, u_2) \in \mathcal{A}(R)$.
    \item [\textbf{(e)}] 
        \begin{align*} 
                [x_{[\alpha]}(t),x_{[\beta]}(u)] = \begin{cases}
                    x_{[\alpha]+[\beta]}(N_{\alpha,\beta}tu) x_{[\alpha]+2[\beta]}(N_{\alpha, \bar{\beta}} N_{\beta,\alpha + \bar{\beta}} t u \bar{u}) \\
                    \hspace{5mm} x_{[\alpha]+3[\beta]}(N_{\alpha, \bar{\bar{\beta}}} N_{\beta,\alpha + \bar{\bar{\beta}}} N_{\beta,\alpha + \bar{\beta} + \bar{\bar{\beta}}} t u \bar{u} \bar{\bar{u}}) \\
                    \hspace{10mm} x_{2[\alpha]+3[\beta]}(N_{\beta,\alpha + \bar{\beta}} N_{\alpha + \beta + \bar{\beta}, \alpha + \bar{\bar{\beta}}} t^2 u \bar{u} \bar{\bar{u}}), \text{ or }  \\
                    x_{[\alpha]+[\beta]}(N_{\alpha,\beta}tu) x_{[\alpha]+2[\beta]}(N_{\alpha, \bar{\bar{\beta}}} N_{\beta,\alpha + \bar{\bar{\beta}}} t u \bar{\bar{u}}) \\
                    \hspace{5mm} x_{[\alpha]+3[\beta]}(N_{\alpha, \bar{\bar{\beta}}} N_{\beta,\alpha + \bar{\bar{\beta}}} N_{\beta,\alpha + \bar{\beta} + \bar{\bar{\beta}}} t u \bar{u} \bar{\bar{u}}) \\
                    \hspace{10mm} x_{2[\alpha]+3[\beta]}(N_{\beta,\alpha + \bar{\beta}} N_{\alpha + \beta + \bar{\beta}, \alpha + \bar{\bar{\beta}}} t^2 u \bar{u} \bar{\bar{u}});
                \end{cases}
        \end{align*}
    where $t \in R_\theta, u \in R$.
    \item [\textbf{(f)}] 
        \begin{align*} 
            [x_{[\alpha]}(t),x_{[\beta]}(u)] = \begin{cases}
                x_{[\alpha]+[\beta]}(N_{\alpha,\bar{\beta}} \ t \bar{u} + N_{\bar{\alpha}, \beta} \ \bar{t} u) \\
                \hspace{5mm} x_{2[\alpha]+[\beta]}(N_{\alpha, \bar{\alpha} + \bar{\bar{\beta}}} N_{\bar{\alpha}, \bar{\bar{\beta}}} (t \bar{t} \bar{\bar{u}} + \bar{t} \bar{\bar{t}} u + t \bar{\bar{t}} \bar{u})) \\
                \hspace{10mm} x_{[\alpha]+2[\beta]}(N_{\bar {\beta}, \alpha + \bar{\bar{\beta}}} N_{\alpha, \bar{\bar{\beta}}} (t \bar{u} \bar{\bar{u}} + \bar{t} u \bar{\bar{u}} + \bar{\bar{t}} u \bar{u})), \text{ or }  \\
                x_{[\alpha]+[\beta]}(N_{\alpha,\bar{\bar{\beta}}} t \bar{\bar{u}} + N_{\bar{\bar{\alpha}}, \beta} \bar{\bar{t}} u) \\
                \hspace{5mm} x_{2[\alpha]+[\beta]}(N_{\alpha, \bar{\alpha} + \bar{\bar{\beta}}} N_{\bar{\alpha}, \bar{\bar{\beta}}} (t \bar{t} \bar{\bar{u}} + \bar{t} \bar{\bar{t}} u + t \bar{\bar{t}} \bar{u})) \\
                \hspace{10mm} x_{[\alpha]+2[\beta]}(N_{\bar {\beta}, \alpha + \bar{\bar{\beta}}} N_{\alpha, \bar{\bar{\beta}}} (t \bar{u} \bar{\bar{u}} + \bar{t} u \bar{\bar{u}} + \bar{\bar{t}} u \bar{u}));
                \end{cases}
        \end{align*}
    where $t,u \in R$.
    \item [\textbf{(g)}] $[x_{[\alpha]}(t),x_{[\beta]}(u)] = x_{[\alpha]+[\beta]}(N_{\alpha,\bar{\beta}} (t \bar{u} + \bar{t} \bar{\bar{u}} + \bar{\bar{t}} u))$ or $x_{[\alpha]+[\beta]}(N_{\alpha,\bar{\bar{\beta}}} (t \bar{\bar{u}} + \bar{t} u + \bar{\bar{t}} \bar{u})),$ where $t, u \in R$.
\end{enumerate}

\begin{rmk}
    For the proof of $(a_1)$ to $(d-ii)$, see \cite{EA1}. We will give a proof of part $(e), (f)$ and $(g)$. 
\end{rmk}

\noindent \textit{Proof of $(e):$} Using commutator relations in $G_{\pi}(\Phi,R)$, we can show that
\begin{align*}
    [x_{[\alpha]}(t), x_{[\beta]}(u)] &= \{ x_{\alpha + \beta}(N_{\alpha, \beta} t u) x_{\alpha + \bar{\beta}}(N_{\alpha, \bar{\beta}} t \bar{u}) x_{\alpha + \bar{\bar{\beta}}}(N_{\alpha, \bar{\bar{\beta}}} t \bar{\bar{u}}) \} \{ x_{\alpha + \beta + \bar{\beta}}(N_{\alpha, \bar{\beta}} N_{\beta, \alpha + \bar{\beta}}  t u \bar{u}) \\
    & \hspace{10mm} x_{\alpha + \beta + \bar{\bar{\beta}}}(N_{\alpha, \bar{\bar{\beta}}} N_{\beta, \alpha + \bar{\bar{\beta}}}  t u \bar{\bar{u}}) x_{\alpha + \bar{\beta} + \bar{\bar{\beta}}}(N_{\alpha, \bar{\bar{\beta}}} N_{\bar{\beta}, \alpha + \bar{\bar{\beta}}}  t \bar{u} \bar{\bar{u}}) \} \\
    & \hspace{15mm} \{ x_{\alpha + \beta + \bar{\beta} + \bar{\bar{\beta}}}(N_{\alpha, \bar{\bar{\beta}}} N_{\beta,\alpha + \bar{\bar{\beta}}} N_{\beta, \alpha + \bar{\beta} + \bar{\bar{\beta}}} t u \bar{u} \bar{\bar{u}}) \} \\
    & \hspace{20mm} \{ x_{2\alpha + \beta + \bar{\beta} + \bar{\bar{\beta}}}(N_{\alpha, \bar{\beta}} N_{\alpha, \bar{\bar{\beta}}} N_{\beta,\alpha + \bar{\beta}} N_{\alpha + \beta + \bar{\beta}, \alpha + \bar{\bar{\beta}}} t^2 u \bar{u} \bar{\bar{u}}) \}.
\end{align*} 
From the choice of Chevalley bases (see Lemma \ref{epsilonalpha}), we have $N_{\alpha, \beta} = N_{\alpha, \bar{\beta}} = N_{\alpha, \bar{\bar{\beta}}}$. 
For $X,Y,Z \in \mathcal{L}$, we have Jacobi identity 
$$[X,[Y,Z]] + [Y,[Z,X]] + [Z,[X,Y]] = 0.$$
By taking $X = X_\alpha, Y= X_\beta$ and $Z= X_{\bar{\beta}}$, we get $N_{\beta, \alpha + \bar{\beta}} = N_{\bar{\beta}, \alpha + \beta}$. But then 
\begin{align*}
    N_{\bar{\beta}, \alpha + \bar{\bar{\beta}}} = N_{\beta, \alpha + \bar{\beta}} = N_{\bar{\beta}, \alpha + \beta} = N_{\bar{\bar{\beta}}, \alpha + \bar{\beta}} = N_{\beta, \alpha + \bar{\bar{\beta}}}.
\end{align*}
Now our assertion follows readily. \qed

\vspace{2mm}

\noindent \textit{Proof of $(f):$} Using commutator relations in $G_{\pi}(\Phi,R)$ and the fact that $\alpha + \bar{\beta} = \bar{\alpha} + \beta$, $\alpha + \bar{\bar{\beta}} = \bar{\bar{\alpha}} + \beta$, $\bar{\alpha} + \bar{\bar{\beta}} = \bar{\bar{\alpha}} + \bar{\beta}$, we can show that
\begin{align*}
    [x_{[\alpha]}(t), x_{[\beta]}(u)] &= \{ x_{\alpha + \bar{\beta}} (N_{\alpha, \bar{\beta}} t \bar{u} + N_{\bar{\alpha}, \beta} \bar{t} u) x_{\bar{\alpha} + \bar{\bar{\beta}}} (N_{\bar{\alpha}, \bar{\bar{\beta}}} \bar{t} \bar{\bar{u}} + N_{\bar{\bar{\alpha}}, \bar{\beta}} \bar{\bar{t}} \bar{u}) x_{\bar{\bar{\alpha}} + \beta} (N_{\bar{\bar{\alpha}}, \beta} \bar{\bar{t}} u + N_{\alpha, \bar{\bar{\beta}}} t \bar{\bar{u}}) \} \\
    & \hspace{10mm} \{ x_{\alpha + \bar{\alpha} + \bar{\bar{\beta}}} (N_{\alpha, \bar{\alpha} + \bar{\bar{\beta}}} N_{\bar{\alpha}, \bar{\bar{\beta}}} \ t \bar{t} \bar{\bar{u}} + N_{\bar{\alpha}, \bar{\bar{\alpha}} + \beta} N_{\bar{\bar{\alpha}}, \beta} \ \bar{t} \bar{\bar{t}} u + N_{\alpha, \bar{\bar{\alpha}} + \bar{\beta}} N_{\bar{\bar{\alpha}}, \bar{\beta}} \ t \bar{\bar{t}} \bar{u}) \} \\
    & \hspace{15mm} \{ x_{\alpha + \bar{\beta} + \bar{\bar{\beta}}} (N_{\bar{\beta}, \alpha + \bar{\bar{\beta}}} N_{\alpha, \bar{\bar{\beta}}} \ t \bar{u} \bar{\bar{u}} + N_{\beta, \bar{\alpha} + \bar{\bar{\beta}}} N_{\bar{\alpha}, \bar{\bar{\beta}}} \ \bar{t} u \bar{\bar{u}} + N_{\beta, \bar{\bar{\alpha}} + \bar{\beta}} N_{\bar{\bar{\alpha}}, \bar{\beta}} \ \bar{\bar{t}} u \bar{u}) \}. 
\end{align*}
From the choice of Chevalley basis (see Lemma \ref{epsilonalpha}), we have $N_{\alpha, \bar{\beta}} = N_{\bar{\alpha}, \bar{\bar{\beta}}} = N_{\bar{\bar{\alpha}}, \beta}$ and $N_{\bar{\alpha}, \beta} = N_{\bar{\bar{\alpha}}, \bar{\beta}} = N_{\alpha, \bar{\bar{\beta}}}$. 
For $X,Y,Z \in \mathcal{L}$, we have Jacobi identity 
$$[X,[Y,Z]] + [Y,[Z,X]] + [Z,[X,Y]] = 0.$$
By taking $X = X_{\bar{\alpha}}, Y= X_{\bar{\bar{\alpha}}}$ and $Z= X_{\beta}$, we get $N_{\bar{\alpha}, \bar{\bar{\alpha}} + \beta} N_{\bar{\bar{\alpha}}, \beta} = N_{\bar{\bar{\alpha}}, \bar{\alpha} + \beta} N_{\bar{\alpha}, \beta}$. But then 
\begin{align*}
    N_{\alpha, \bar{\alpha} + \bar{\bar{\beta}}} N_{\bar{\alpha}, \bar{\bar{\beta}}} = N_{\bar{\alpha}, \bar{\bar{\alpha}} + \beta} N_{\bar{\bar{\alpha}}, \beta} = N_{\bar{\bar{\alpha}}, \bar{\alpha} + \beta} N_{\bar{\alpha}, \beta} =
    N_{\alpha, \bar{\bar{\alpha}} + \bar{\beta}} N_{\bar{\bar{\alpha}}, \bar{\beta}}.
\end{align*}
Similarly, by taking $X = X_{\bar{\beta}}, Y= X_{\bar{\bar{\beta}}}, Z= X_{\alpha}$, we get $N_{\bar{\beta}, \alpha + \bar{\bar{\beta}}} N_{\alpha, \bar{\bar{\beta}}} = N_{\bar{\bar{\beta}}, \alpha + \bar{\beta}} N_{\alpha, \bar{\beta}}$. But then 
\begin{align*}
    N_{\beta, \bar{\bar{\alpha}} + \bar{\beta}} N_{\bar{\bar{\alpha}}, \bar{\beta}} = N_{\bar{\beta}, \alpha + \bar{\bar{\beta}}} N_{\alpha, \bar{\bar{\beta}}} = N_{\bar{\bar{\beta}}, \alpha + \bar{\beta}} N_{\alpha, \bar{\beta}} =
    N_{\beta, \bar{\alpha} + \bar{\bar{\beta}}} N_{\bar{\alpha}, \bar{\bar{\beta}}}.
\end{align*}
Now our assertion follows readily. \qed

\vspace{2mm}

\noindent \textit{Proof of $(g):$} Observe that, either $\alpha + \bar{\beta} \in \Phi$ and $\alpha + \bar{\bar{\beta}} \not\in \Phi$ or vice versa. We consider a case where $\alpha + \bar{\beta} \in \Phi$ (the proof of other case is similar and hence omitted). Now using commutator relations in $G_{\pi}(\Phi,R)$ and the fact that $\alpha + \bar{\beta} = \bar{\alpha} + \bar{\bar{\beta}} = \bar{\bar{\alpha}} + \beta$, we can show that 
$$[x_{[\alpha]}(t), x_{[\beta]}(u)] = x_{\alpha+\bar{\beta}}(N_{\alpha,\bar{\beta}} t \bar{u} + N_{\bar{\alpha}, \bar{\bar{\beta}}} \bar{t} \bar{\bar{u}} + N_{\bar{\bar{\alpha}}, \beta} \bar{\bar{t}} u).$$
From the choice of Chevalley basis (see Lemma \ref{epsilonalpha}), we have $N_{\alpha, \bar{\beta}} = N_{\bar{\alpha}, \bar{\bar{\beta}}} = N_{\bar{\bar{\alpha}}, \beta}$ and hence the result follows. \qed

%%%%%%%%%%%%%%%%%%%%%%%%%%%%%%%%%%%%%%%%%%%%%%%%%%

\subsection{The Subgroups \texorpdfstring{$N'_\sigma$}{N} and \texorpdfstring{$H'_\sigma$}{H}}\label{N and H} 

We now turn our attention to studying the subgroups $N'_{\sigma}$ and $H'_{\sigma}$ of the group $E'_\sigma (R)$. Understanding the structure of $N'_\sigma$ and $H'_\sigma$ are relatively straightforward when $\Phi_\rho \sim {}^2 A_{2n-1} \ (n \geq 2), {}^2 D_n \ (n \geq 4), {}^2 E_6$ or ${}^3 D_4$. However, it becomes more complex when $\Phi_\rho \sim {}^2 A_{2n} \ (n \geq 1)$.

\begin{conv}
    \normalfont
    At this point, we want to establish a convention regarding some notation. The classes $[-\alpha]$ and $-[\alpha]$ denote the same set, but they may differ as ordered sets. If $[\alpha] \sim A_1$, then both notations are identical. If $[\alpha] \sim A_1^2$ or $A_1^3$, then $\alpha'$ represents $[-\alpha]$ where $\alpha' = \min \{ -\alpha, -\bar{\alpha}, -\bar{\bar{\alpha}} \}$. In these cases, as an ordered set, $[-\alpha] = [\alpha'] = \{ \alpha', \bar{\alpha'} \}$ or $\{ \alpha', \bar{\alpha'}, \bar{\bar{\alpha'}} \}$ depending on whether $[\alpha] \sim A_1^2$ or $A_1^3$, respectively. Whence, for $-[\alpha]$, as an ordered set, $-[\alpha] = \{ -\alpha, -\bar{\alpha} \}$ or $\{ -\alpha, -\bar{\alpha}, -\bar{\bar{\alpha}} \}$ depending on whether $[\alpha] \sim A_1^2$ or $A_1^3$, respectively. Finally, if $[\alpha] \sim A_2$, both the notations represent the same ordered set: $\{ - \bar{\alpha}, -\alpha, -\alpha - \bar{\alpha} \}$ if $\alpha < \bar{\alpha}$.
\end{conv}

Write $R^*= \{r \in R \mid \exists s \in R \text{ such that } rs = 1\}$, $R_\theta^* = R_\theta \cap R^*$ and $\mathcal{A}(R)^* := \{ (t,u) \in \mathcal{A}(R) \mid u \in R^* \}.$ For given $[\alpha] \in \Phi_\rho$, we also write 
\[
    R_{[\alpha]}^* = \begin{cases}
        R^*_\theta & \text{if } [\alpha] \sim A_1, \\
        R^* & \text{if } [\alpha] \sim A_1^2 \text{ or } A_1^3, \\
        \mathcal{A}(R)^* & \text{if } [\alpha] \sim A_2;
    \end{cases} \quad \text{and} \quad     R_{[\alpha]}^{\star} = \begin{cases}
        R^*_\theta & \text{if } [\alpha] \sim A_1, \\
        R^* & \text{if } [\alpha] \sim A_1^2, A_1^3 \text{ or } A_2.
    \end{cases}
\]

With these notations established, we proceed to define the following special elements of $N_\sigma$ and $H_\sigma$:
\begin{enumerate}
    \item[$\textbf{(W1)}$] If $[\alpha] \sim A_1$, then define $w_{[\alpha]}(t) := x_{[\alpha]}(t) x_{-[\alpha]}(-t^{-1}) x_{[\alpha]}(t) = w_\alpha (t), t \in R_\theta^*.$
    \item[$\textbf{(W2)}$] If $[\alpha] \sim A_1^2$, then define $w_{[\alpha]}(t) := x_{[\alpha]}(t) x_{-[\alpha]}(-t^{-1}) x_{[\alpha]}(t) = w_{\alpha} (t) w_{\bar{\alpha}} (\bar{t}), t \in R^*.$
    \item[$\textbf{(W3)}$] If $[\alpha] \sim A_1^3$, then define $w_{[\alpha]}(t) := x_{[\alpha]}(t) x_{-[\alpha]}(-t^{-1}) x_{[\alpha]}(t) = w_{\alpha} (t) w_{\bar{\alpha}} (\bar{t}) w_{\bar{\bar{\alpha}}} (\bar{\bar{t}}), t \in R^*.$
    \item[$\textbf{(W4)}$] If $[\alpha] \sim A_2$, then define $w_{[\alpha]}(t,u) := x_{[\alpha]}(t,u) x_{-[\alpha]}( -\bar{u}^{-1} \cdot (t,u)) x_{[\alpha]}( u \bar{u}^{-1} \cdot (t,u)) = x_{[\alpha]}(t,u) x_{-[\alpha]}(- (\bar{u}^{-1}) t, (\bar{u}^{-1})) x_{[\alpha]}(u \bar{u}^{-1} t, u),$ where $(t,u) \in \mathcal{A}(R)^*$. 
    \item[\textbf{(W4$'$)}] If $[\alpha] \sim A_2$ such that $\alpha \neq \bar{\alpha}$, then define $w_{[\alpha]} (t) := w_\alpha (\bar{t}) w_{\bar{\alpha}}(1) w_{\alpha} (t), t \in R^*$. \\
    
    \item[$\textbf{(H1)}$] If $[\alpha] \sim A_1$, then define $h_{[\alpha]}(t) := w_{[\alpha]}(t) w_{[\alpha]}(-1) = h_\alpha (t), \ t \in R^*_{\theta}$.
    \item[$\textbf{(H2)}$] If $[\alpha] \sim A_1^2$, then define $h_{[\alpha]}(t) = w_{[\alpha]}(t) w_{[\alpha]}(-1) = h_\alpha (t) h_{\bar{\alpha}}(\bar{t}), \ t \in R^*$.
    \item [$\textbf{(H3)}$] If $[\alpha] \sim A_1^3$, then define $h_{[\alpha]}(t) = w_{[\alpha]}(t) w_{[\alpha]}(-1) = h_\alpha (t) h_{\bar{\alpha}}(\bar{t}) h_{\bar{\bar{\alpha}}}(\bar{\bar{t}}), \ t \in R^*$.
    \item[$\textbf{(H4)}$] If $[\alpha] \sim A_2$, then define $h_{[\alpha]}((t,u),(t',u')) = w_{[\alpha]}(t,u) w_{[\alpha]}(t',u')$, where $(t,u), (t',u') \in \mathcal{A}(R)^*$. 
    \item[\textbf{(H4$'$)}] If $[\alpha] \sim A_2$ such that $\alpha \neq \bar{\alpha}$, then define $h_{[\alpha]} (t) := h_\alpha (t) h_{\bar{\alpha}}(\bar{t}), t \in R^*$.
\end{enumerate}

\begin{rmk}
    \normalfont
    \begin{enumerate}[(a)]
        \item One can easily verify that the last equality holds in $(W1), (W2), (W3),$ $(H1),$ $(H2)$ and $(H3)$.
        \item Recall that, $\sigma (h_\alpha (t)) = h_{\bar{\alpha}} (\bar{t})$ and $\sigma (w_\alpha (t)) = w_{\bar{\alpha}} (\epsilon_\alpha \bar{t})$, where $\epsilon_\alpha = \pm 1$ (note that $\epsilon_\alpha = -1$ if and only if $[\alpha] \sim A_2$ and $\alpha \neq \bar{\alpha}$ (see Lemma \ref{epsilonalpha})). Hence it is clear that $w_{[\alpha]}(t) \in N'_\sigma \subset N_\sigma$ and $h_{[\alpha]}(t) \in H'_\sigma \subset H_\sigma$, if $[\alpha] \sim A_1, A_1^2, A_1^3.$ Similarly, $w_{[\alpha]}(t,u) \in N'_\sigma \subset N_\sigma$ and $h_{[\alpha]}((t,u),(t',u')) \in H'_\sigma \subset H_\sigma$, if $[\alpha] \sim A_2$  (see Lemma $\ref{lemmaW4'}$ and Lemma $\ref{lemmaRel-of-WandH}$).
        \item If $[\alpha] \sim A_2$ then $w_{[\alpha]}(t) \in N_\sigma$ and $h_{[\alpha]}(t) \in H_\sigma$ (see Lemma $\ref{lemmaW4'}$), but it is not necessary that every $w_{[\alpha]}(t)$ (resp., $h_{[\alpha]} (t)$), $t \in R^*$ contained in $N'_\sigma$ (resp., $H'_\sigma$).
        \item Each $h_{[\alpha]}(t)$ defined in $(H1), (H2), (H3)$ and $(H4')$ is multiplicative as a function of $t$. In particular, $h_{[\alpha]}(t)^{-1} = h_{[\alpha]} (t^{-1})$.
        \item If $[\alpha] \sim A_1, A_1^2$ or $A_1^3$, then $x_{[\alpha]} (t)^{-1} = x_{[\alpha]}(-t)$ and $w_{[\alpha]} (t)^{-1} = w_{[\alpha]}(-t)$.
    \end{enumerate}
\end{rmk}

\begin{lemma}[see {\cite[Proposition 2.1]{EA1}}]\label{w^-1}
    If $[\alpha] \sim A_2$ and $(t,u) \in \mathcal{A}(R)^*$ then $w_{[\alpha]}(t,u)^{-1} = w_{[\alpha]} (-t u \bar{u}^{-1}, \bar{u})$.
\end{lemma}

\begin{proof}
     Note that $w_{[\alpha]}(t,u) = x_{[\alpha]}(t,u) x_{-[\alpha]}(- \bar{u}^{-1} t, \bar{u}^{-1}) x_{[\alpha]}(u \bar{u}^{-1} t, u)$. Hence 
    \begin{align*}
        (w_{[\alpha]}(t,u))^{-1} &= [{x_{[\alpha]}(t,u) x_{-[\alpha]}(- (\bar{u}^{-1}) t, (\bar{u}^{-1})) x_{[\alpha]}(u (\bar{u}^{-1}) t, u)}]^{-1} \\
        &= [x_{[\alpha]}(u (\bar{u}^{-1}) t, u)]^{-1} [x_{-[\alpha]}(- (\bar{u}^{-1}) t, (\bar{u}^{-1}))]^{-1} [x_{[\alpha]}(t,u)]^{-1} \\
        &= x_{[\alpha]}(- u (\bar{u}^{-1}) t, \bar{u}) x_{-[\alpha]}((\bar{u}^{-1}) t, (u^{-1})) x_{[\alpha]}(-t, \bar{u}) \\
        &= w_{[\alpha]} (- u (\bar{u}^{-1}) t, \bar{u}).
    \end{align*}
    This proves our lemma.
\end{proof}

% \begin{lemma}[see {\cite[Proposition 2.3]{EA1}}]
%     If $[\alpha] \sim A_2$ and $(t,u) \in \mathcal{A}(R)^*$ then $w_{[\alpha]}(t,u) = w_{-[\alpha]} (-t u^{-2} \bar{u}, \bar{u}^{-1})$.
% \end{lemma}

\begin{lemma}\label{lemmaW4'}
    \normalfont
    If $[\alpha] \sim A_2$ such that $\alpha \neq \bar{\alpha}$, then
    \begin{enumerate}[(a)]
        \item $w_{[\alpha]}(t) \in N_\sigma$ and $w_{[\alpha]} (t) ^{-1} = w_{[\alpha]} (\bar{t})$, for every $t \in R^*$.
        \item $h_{[\alpha]}(t) \in H_\sigma, h_{[\alpha]}(t) = w_{[\alpha]} (\bar{t}) w_{[\alpha]} (1)$ and $h_{[\alpha]}(t)^{-1} = h_{[\alpha]}(t^{-1})$, for every $t \in R^*$.
    \end{enumerate} 
\end{lemma}

\begin{proof}
    Define $E_3 (R)$ be the subgroup of $SL_3(R)$ generated by 
    \begin{gather*}
        x'_{\alpha}(t) := 1 + t E_{23}, \quad x'_{\bar{\alpha}} (t):= 1 + t E_{12}, \quad x'_{\alpha + \bar{\alpha}} (t):= 1 + t E_{13}, \\
        x'_{-\alpha}(t):= 1 + t E_{32}, \quad x'_{-\bar{\alpha}} (t):= 1 + t E_{21}, \quad x'_{-\alpha - \bar{\alpha}} (t):= 1 + t E_{31}.
    \end{gather*}
    
    Consider a subgroup $K = \langle \mathfrak{X}_{\alpha}, \mathfrak{X}_{-\alpha}, \mathfrak{X}_{\bar{\alpha}}, \mathfrak{X}_{-\bar{\alpha}} \rangle$ of $E_\pi(\Phi, R)$. Then there is a natural surjective homomorphism of groups $$ \psi: E_3 (R) \longrightarrow K$$ given by 
    \[
        x'_{\beta}(t) \longmapsto x_{\beta}(t)
    \]
    for all $\beta \in \{ \pm \alpha, \pm \bar{\alpha}, \pm (\alpha + \bar{\alpha}) \}$ and $t \in R$.
    
    Note that $\sigma = \sigma|_K$ is an automorphism of the subgroup $K$ and there exists a natural automorphism $\sigma'$ of $E_3 (R)$ such that $\sigma \circ \psi = \psi \circ \sigma'$. We have $\psi(E_{3, \sigma'}(R)) = K_{\sigma}(R)$ and $\psi(E'_{3, \sigma'}(R)) = K'_{\sigma} (R)$. Hence it is enough to prove the corresponding results in the group $S(R)$.  
    
    The notation of $w'_{[\alpha]}(t)$ and $h'_{[\alpha]}(t)$ is clear. Note that $w'_{[\alpha]}(t) \in N_{S, \sigma'}(R)$ and $h'_{[\alpha]}(t) \in H_{S, \sigma'}(R),$ hence $w_{[\alpha]}(t) \in N_\sigma$ and $h_{[\alpha]}(t) \in H_\sigma$. Now the proof is immediate from below: $$w'_{[\alpha]} (t) w'_{[\alpha]} (\bar{t}) = \begin{pmatrix}
        0 & 0 & t \\
        0 & -t^{-1}\bar{t} & 0 \\
        \bar{t}^{-1} & 0 & 0
    \end{pmatrix} \begin{pmatrix}
        0 & 0 & \bar{t} \\
        0 & -\bar{t}^{-1}t & 0 \\
        t^{-1} & 0 & 0
    \end{pmatrix} = \begin{pmatrix}
        1 & 0 & 0 \\
        0 & 1 & 0 \\
        0 & 0 & 1
    \end{pmatrix},$$ 
    $$ w'_{[\alpha]} (\bar{t}) w'_{[\alpha]} (1) = \begin{pmatrix}
        0 & 0 & \bar{t} \\
        0 & -\bar{t}^{-1}t & 0 \\
        t^{-1} & 0 & 0
    \end{pmatrix} \begin{pmatrix}
        0 & 0 & 1 \\
        0 & -1 & 0 \\
        1 & 0 & 0
    \end{pmatrix} = \begin{pmatrix}
        \bar{t} & 0 & 0 \\
        0 & t \bar{t}^{-1} & 0 \\
        0 & 0 & t^{-1}
    \end{pmatrix} = h'_{[\alpha]}(t).$$
\end{proof}

\begin{lemma}\label{lemmaRel-of-WandH}
    \normalfont Let $[\alpha] \sim A_2$ such that $\alpha \neq \bar{\alpha}$.
    \begin{enumerate}[(a)]
        \item If $(t, u) \in \mathcal{A}(R)^*$, then $w_{[\alpha]}(t,u) = w_{[\alpha]}(u)$.
        \item If $(t, u), (t', u') \in \mathcal{A}(R)^*$, then $h_{[\alpha]}((t,u), (t',u')) = h_{[\alpha]}(\bar{u} {u'}^{-1}).$
    \end{enumerate}
\end{lemma}

\begin{proof}
    As we pointed out earlier, it is enough to prove the corresponding results in $sl_3(R)$. Note that 
    \begin{align*}
        w'_{[\alpha]}(u) &= w'_{\alpha}(\bar{u}) w'_{\bar{\alpha}}(1) w'_{\alpha} (u) \\
        &= \begin{pmatrix}
        1 & 0 & 0 \\
        0 & 0 & \bar{u} \\
        0 & -\bar{u}^{-1} & 0
    \end{pmatrix} \begin{pmatrix}
        0 & 1 & 0 \\
        -1 & 0 & 0 \\
        0 & 0 & 1
    \end{pmatrix} \begin{pmatrix}
        1 & 0 & 0 \\
        0 & 0 & u \\
        0 & -{u}^{-1} & 0
    \end{pmatrix} \\
    &= \begin{pmatrix}
        0 & 0 & u \\
        0 & -u^{-1}\bar{u} & 0 \\
        \bar{u}^{-1} & 0 & 0
    \end{pmatrix} \\
    &= w'_{[\alpha]}(t,u),
    \end{align*} which proves $(a)$. 
    Again note that 
    \begin{align*}
        h'_{[\alpha]}(\bar{u} (u')^{-1}) &= h'_{\alpha}(\bar{u} (u')^{-1}) h'_{\bar{\alpha}}(u (\bar{u'})^{-1}) \\ 
        &= \begin{pmatrix}
        1 & 0 & 0 \\
        0 & \bar{u} (u')^{-1} & 0 \\
        0 & 0 & u' \bar{u}^{-1}
        \end{pmatrix} \begin{pmatrix}
        u (\bar{u'})^{-1} & 0 & 0 \\
        0 & u^{-1} (\bar{u'}) & 0 \\
        0 & 0 & 1
        \end{pmatrix} \\
        &= \begin{pmatrix}
        u(\bar{u'})^{-1} & 0 & 0 \\
        0 & u^{-1}\bar{u} (u')^{-1}\bar{u'}  & 0 \\
        0 & 0 & u' (\bar{u})^{-1}
        \end{pmatrix} \\
        &= h'_{[\alpha]}((t,u),(t',u')),
    \end{align*} which proves $(b)$.
\end{proof}

\begin{defn}
    \normalfont
    Let $[\alpha] \sim A_2$ such that $\alpha \neq \bar{\alpha}$. We define 
    \begin{enumerate}[(a)]
        \item $\mathcal{R}_1 = \{ u \in R^* \mid \exists \ t \in R \text{ such that } (t,u) \in \mathcal{A}(R)^* \}$;
        \item $\mathcal{R}_k = \{ u \in R^* \mid \exists \ u_1, \dots, u_k \in \mathcal{R}_1 \text{ such that } u = u_1 \dots u_k \}$, for given $k \in \mathbb{N}$;
        \item $\mathcal{R}^{(l)} = \{ u \in R^* \mid \exists \ k \in \mathbb{N} \text{ such that } u \in \mathcal{R}_{kl} \} = \cup_{k \in \mathbb{N}} \mathcal{R}_{kl}$, for given $l \in \mathbb{N}$;
        \item $\mathcal{R} := \mathcal{R}^{(1)}, \mathcal{R}':= \mathcal{R}^{(2)}$ and $\mathcal{R}'':=  \cup_{k \in \mathbb{N}} \mathcal{R}_{2k-1}$.
    \end{enumerate}
\end{defn}

\begin{rmk}
    \normalfont
    The following are immediate consequences of the definition, provided $\mathcal{R}_1 \neq \phi$:
    \begin{enumerate}[(a)]
        \item If $u$ is in $\mathcal{R}_k$, then so are $\bar{u}, u^{-1}$ and $a \bar{a} u \ (a \in R^*)$.
        \item For any $l \in \mathbb{N}$, $\mathcal{R}^{(l)}$ is a subgroup of the multiplicative group $R^*$ generated by $\mathcal{R}_l$. In particular, $\mathcal{R}$ and $\mathcal{R}'$ are subgroups of $R^*$.
        \item If $1 \in \mathcal{R}_k$, then $\mathcal{R}_m \subset \mathcal{R}_{k+m},$ for all $m \in \mathbb{N}$. In particular, since $1 \in \mathcal{R}_2$, we have $\mathcal{R}_1 \subset \mathcal{R}_{3} \subset \mathcal{R}_5 \subset \cdots$ and $\mathcal{R}_2 \subset \mathcal{R}_{4} \subset \mathcal{R}_6 \subset \cdots$.
        \item If $R$ is a field then $R^* = \mathcal{R} = \mathcal{R}' = \mathcal{R}_{2k}$, for all $k \in \mathbb{N}$. To see this, it is enough to see that $\mathcal{R}_2 = R^*$ (by part $(c)$). Let $u \in R^*$. If $u = \bar{u}$ then we choose $u_1 \in R^*$ such that $u_1 = - \bar{u_1}$ (such a $u_1$ exists in field) and $u_2 = u (u_1)^{-1}$. Similarly, if $u \neq \bar{u}$ then we can choose $u_1 = (\bar{u} - u)^{-1}$ and $u_2= u (u_1)^{-1}$. In both the cases, we have $u = u_1 u_2$ such that $u_1, u_2 \in \mathcal{R}_1$ as $(0, u_1), (u - \bar{u}, u_2) \in \mathcal{A}(R)^*$. Therefore $u \in \mathcal{R}_2$, as desired.  
    \end{enumerate}
\end{rmk}

\begin{cor}\label{cor:H' subset H}
    \normalfont
    Let $[\alpha] \sim A_2$. If $u \in \mathcal{R}'$ then $h_{[\alpha]}(u) \in H'_\sigma$. Similarly, if $u \in \mathcal{R}''$ then $w_{[\alpha]}(u) \in N'_\sigma$.
\end{cor}

\begin{proof}
    Immediate from Lemma \ref{lemmaRel-of-WandH} and the fact that $w_{[\alpha]}(t_1,u_1) w_{[\alpha]}(t_2,u_2)^{-1} w_{[\alpha]}(t_3, u_3) = w_{[\alpha]}(u_1 u_2^{-1} u_3) \in N'_\sigma$.
\end{proof}

%%%%%%%%%%%%%%%%%%%%%%%%%%%%%%%%%%%%%%%%%%%%%%%%%%

\subsection{The Steinberg Relations} In this subsection we present some useful relations in the group $E'_\sigma (R)$. 
Recall that $w_\alpha (t) X_\beta w_{\alpha}(t)^{-1} = c t^{-\langle \beta, \alpha \rangle} X_{s_\alpha(\beta)}$, where $c = c(\alpha, \beta) = \pm 1$ is independent of $t, R$ and the representation chosen, and $c(\alpha, \beta)= c(\alpha, -\beta)$ (see Lemma $19(a)$ of \cite{RS}). 
If $\Phi$ is a root system with one root length, then the function $c: \Phi \times \Phi \longrightarrow \{\pm 1 \}$ can be preciously given as follows: 
\[
    c(\alpha, \beta) = c(\alpha, -\beta) = \begin{cases}
        -1 & \text{ if } \alpha = \beta, \\
        1 & \text{ if } \alpha \neq \beta \text{ and } \alpha \pm \beta \not\in \Phi, \\
        N_{\alpha, \beta} & \text{ if } \alpha \neq \beta \text{ and } \alpha + \beta \in \Phi, \alpha - \beta \not\in \Phi.
    \end{cases}
\]

Our objective is to establish relations for twisted Chevalley groups analogous to those in Lemma 20 of \cite{RS} for Chevalley groups. But before we do that, let us consider the function $d: \Phi_\rho \times \Phi \longrightarrow \{ \pm 1 \}$ given by
\[
    d([\alpha], \beta) = \begin{cases}
        c(\alpha, \beta) & \text{ if } [\alpha] \sim A_1, \\
        c(\alpha, \beta) c(\bar{\alpha}, s_\alpha (\beta)) & \text{ if } [\alpha] \sim A_1^2, \\
        c(\alpha, \beta) c(\bar{\alpha}, s_\alpha (\beta)) c(\bar{\bar{\alpha}}, s_{\bar{\alpha}}s_\alpha (\beta)) & \text{ if } [\alpha] \sim A_1^3, \\
        c(\alpha, \beta) c(\bar{\alpha}, s_\alpha (\beta)) c(\alpha, s_{\bar{\alpha}}s_\alpha (\beta)) & \text{ if } [\alpha] \sim A_2.
    \end{cases} 
\]

\begin{lemma}\label{lemma on d}
    \normalfont
    For every $\alpha, \beta \in \Phi$, $d([\alpha], \beta) = d([\alpha], \bar{\beta})$. Moreover if $\beta \in \Phi$ is such that $\beta \neq \bar{\beta}$ and $[\beta] \sim A_2$, then $d([\alpha], \beta + \bar{\beta}) = N_{s_{\alpha}(\bar{\beta}), s_{\alpha}(\beta)} N_{\bar{\beta}, \beta}$. 
\end{lemma}

\begin{proof}
    First, we observe that \( c(\alpha, \beta) = c(\bar{\alpha}, \bar{\beta}) \) since \( N_{\alpha, \beta} = N_{\bar{\alpha}, \bar{\beta}} \).
    Assume \( [\alpha] \sim A_1 \). Then \( c(\alpha, \beta) = c(\alpha, \bar{\beta}) \) and hence \( d([\alpha], \beta) = d([\alpha], \bar{\beta}) \).
    Now, consider \( [\alpha] \sim A_1^2 \). In this case, depending on the possible values of \( \langle \beta, \alpha \rangle \) and \( \langle \bar{\beta}, \alpha \rangle \), we can address several subcases to establish our result. For instance, if \( \alpha \) and \( \beta \) are such that \( \langle \beta, \alpha \rangle = -1 \) and \( \langle \bar{\beta}, \alpha \rangle = 0 \), then 
    \( d([\alpha], \beta) = c(\alpha, \beta) c(\bar{\alpha}, s_\alpha (\beta)) = N_{\alpha, \beta}. \)
    On the other hand, under the same assumption on \( \alpha \) and \( \beta \), we have 
    \( d([\alpha], \bar{\beta}) = c(\alpha, \bar{\beta}) c(\bar{\alpha}, s_\alpha (\bar{\beta})) = c(\alpha, \bar{\beta}) c(\bar{\alpha}, \bar{\beta}) = N_{\bar{\alpha}, \bar{\beta}}. \)
    Therefore, \( d([\alpha], \beta) = d([\alpha], \bar{\beta}) \). One can similarly verify all the other subcases.
    Finally, the cases where \( [\alpha] \sim A_1^3 \) and \([\alpha] \sim A_2 \) can be proved in a similar manner and are thus omitted.

    The second assertion is only valid in the case where \( \Phi_\rho \sim {}^2 A_{2n} \ (n \geq 2) \). First, assume that \( [\alpha] \sim A_2 \). Unless \( [\alpha] = \pm [\beta] \), we have \( \langle \beta, \alpha \rangle = 0 \) and hence \( s_\alpha (\beta) = \beta \) and \( s_\alpha (\bar{\beta}) = \bar{\beta} \). Therefore, 
    \[
        d([\alpha], \beta + \bar{\beta}) = 1 = N_{s_{\alpha}(\bar{\beta}), s_{\alpha}(\beta)} N_{\bar{\beta}, \beta}.
    \]
    Now, if \( [\alpha] = \pm [\beta] \), then since \( d([\alpha], \beta + \bar{\beta}) = d([\alpha], - \beta - \bar{\beta}) \), we can assume that \( \alpha = \beta \) or \( \bar{\beta} \). Without loss of generality, assume \( \alpha = \beta \), then 
    \begin{align*}
        d([\alpha], \beta + \bar{\beta}) &= c(\alpha, \alpha + \bar{\alpha}) c(\bar{\alpha}, s_\alpha (\alpha + \bar{\alpha})) c(\alpha, s_{\bar{\alpha}} s_{\alpha} (\alpha + \bar{\alpha})) \\
        &= c(\alpha, \alpha + \bar{\alpha}) c(\bar{\alpha}, \bar{\alpha}) c(\alpha, -\bar{\alpha}) \\
        &= - c(\alpha, - \alpha - \bar{\alpha}) c(\alpha, \bar{\alpha}) \\
        &= - N_{\alpha, - \alpha - \bar{\alpha}} N_{\alpha, \bar{\alpha}} \\
        &= N_{s_{\alpha}(\bar{\beta}), s_\alpha (\beta)} N_{\bar{\beta}, \beta}.
    \end{align*}
    Now, assume that \( [\alpha] \sim A_1^2 \). In this case, the possible values of \( \langle \beta + \bar{\beta}, \alpha \rangle \) are \(-1, 0\), or \(1\). If \( \langle \beta + \bar{\beta}, \alpha \rangle = 0 \), then \( \langle \beta, \alpha \rangle = 0 = \langle \bar{\beta}, \alpha \rangle \) (since \( [\alpha] \sim A_1^2 \)). Therefore,
    \(
        N_{s_{\alpha}(\bar{\beta}), s_\alpha (\beta)} N_{\bar{\beta}, \beta} = (N_{\bar{\beta}, \beta})^2 = 1.
    \)
    On the other hand,
    \(
        d([\alpha], \beta + \bar{\beta}) = c(\alpha, \beta) c(\bar{\alpha}, s_\alpha (\beta)) = 1.
    \)
    Therefore,
    \(
        d([\alpha], \beta + \bar{\beta}) = N_{s_{\alpha}(\bar{\beta}), s_\alpha (\beta)} N_{\bar{\beta}, \beta}.
    \)
    The cases where \( \langle \beta + \bar{\beta}, \alpha \rangle = -1 \) or \( 1 \) can be proved similarly and are thus omitted.
\end{proof}

A version of the following Proposition is provided in \cite{EA1}. Here, we present more general formulas that cover all cases, unlike those in \cite[4.1 and 4.3]{EA1}.

\begin{prop}\label{prop wxw^{-1}}
    \normalfont 
    For $\alpha, \beta \in \Phi, t \in R^{\star}_{[\alpha]}$ and $u \in R_{[\beta]}$, we have the following relations:
    \[  
        w_{[\alpha]}(t) x_{[\beta]}(u) w_{[\alpha]}(t)^{-1} = 
        \begin{cases}
            x_{s_{[\alpha]}([\beta])}(d([\alpha],\beta') t^{-\langle \beta', \alpha \rangle} \cdot u') & \text{if } [\alpha] \sim A_1, \\
            x_{s_{[\alpha]}([\beta])}(d([\alpha],\beta') t^{-\langle \beta', \alpha \rangle} {(\bar{t})}^{-\langle \beta', \bar{\alpha} \rangle} \cdot u') & \text{if } [\alpha] \sim A_1^2, \\
            x_{s_{[\alpha]}([\beta])}(d([\alpha],\beta') t^{-\langle \beta', \alpha \rangle} {(\bar{t})}^{-\langle \beta', \bar{\alpha} \rangle} {(\bar{\bar{t}})}^{-\langle \beta', \bar{\bar{\alpha}} \rangle} \cdot u') & \text{if } [\alpha] \sim A_1^3, \\
            x_{s_{[\alpha]}([\beta])}(d([\alpha],\beta') t^{-\langle \beta', \alpha \rangle} {(\bar{t})}^{- \langle \beta', \bar{\alpha} \rangle} \cdot u' ) & \text{if } [\alpha] \sim A_2.
        \end{cases} 
    \]

    \noindent Where the values of $\beta'$ and $u'$ depend on the representing element of the class $s_{[\alpha]}([\beta])$. To be precious, if $[\beta] \sim A_1$, then $\beta' = \beta$ and $u' = u$; if $[\beta] \sim A_1^2$, then $\beta' = \beta$ or $\bar{\beta}$ and $u' = u$ or $\bar{u}$, respectively; if $[\beta] \sim A_1^3$, then $\beta' = \beta, \bar{\beta}$ or $\bar{\bar{\beta}}$ and $u' = u, \bar{u}$ or $\bar{\bar{u}}$, respectively; if $[\beta] \sim A_2$, then $\beta' = \beta$ or $\bar{\beta}$ and $u' = u = (u_1, u_2)$ or $(\bar{u_1}, \bar{u_2})$, respectively.
\end{prop}

\begin{proof}
    By simple calculations using Lemma $19(a)$ of \cite{RS}, we have 
    \[
        w_{[\alpha]}(t) x_{\beta}(u) w_{[\alpha]}(t)^{-1} = \begin{cases}
            x_{s_{\alpha}(\beta)}(d([\alpha], \beta) t^{-\langle \beta, \alpha \rangle} u) & \text{ if } [\alpha] \sim A_1, \\
            x_{s_{\alpha}s_{\bar{\alpha}}(\beta)}(d([\alpha], \beta) t^{-\langle \beta, \alpha \rangle} \bar{t}^{-\langle \beta, \bar{\alpha} \rangle} u) & \text{ if } [\alpha] \sim A_1^2, \\
            x_{s_{\alpha}s_{\bar{\alpha}}s_{\bar{\bar{\alpha}}}(\beta)}(d([\alpha], \beta) t^{-\langle \beta, \alpha \rangle} \bar{t}^{-\langle \beta, \bar{\alpha}  \rangle} \bar{\bar{t}}^{-\langle \beta, \bar{\bar{\alpha}}  \rangle} u) & \text{ if } [\alpha] \sim A_1^2, \\
            x_{s_{\alpha + \bar{\alpha}}(\beta)}(d([\alpha], \beta) t^{-\langle \beta, \alpha \rangle} \bar{t}^{ - \langle \beta, \bar{\alpha} \rangle} u) & \text{ if } [\alpha] \sim A_2.
        \end{cases}
    \]
    If $[\beta] \sim A_1$ then our result follows the above equations. Now if $[\beta] \sim A_1^2$, then
    \begin{align*}
        w_{[\alpha]}(t) x_{[\beta]}(u) w_{[\alpha]}(t)^{-1} &= w_{[\alpha]}(t) x_{\beta}(u) x_{\bar{\beta}}(\bar{u}) w_{[\alpha]}(t)^{-1} \\
        &= (w_{[\alpha]}(t) x_{\beta}(u) w_{[\alpha]}(t)^{-1}) (w_{[\alpha]}(t) x_{\bar{\beta}}(\bar{u}) w_{[\alpha]}(t)^{-1}) \\
        &=
        \begin{cases}
            x_{s_{\alpha}(\beta)}(d([\alpha], \beta) t^{-\langle \beta, \alpha \rangle} u) x_{s_{\alpha}(\bar{\beta})}(d([\alpha], \bar{\beta}) t^{-\langle \bar{\beta}, \alpha \rangle} \bar{u}) & \text{ if } [\alpha] \sim A_1, \\
            x_{s_{\alpha}s_{\bar{\alpha}}(\beta)}(d([\alpha], \beta) t^{-\langle \beta, \alpha \rangle} \bar{t}^{-\langle \beta, \bar{\alpha} \rangle} u) \\ 
            \hspace{10mm} x_{s_{\alpha}s_{\bar{\alpha}}(\bar{\beta})}(d([\alpha], \bar{\beta}) t^{-\langle \bar{\beta}, \alpha \rangle} \bar{t}^{-\langle \bar{\beta}, \bar{\alpha} \rangle} \bar{u}) & \text{ if } [\alpha] \sim A_1^2, \\
            x_{s_{\alpha}s_{\bar{\alpha}}s_{\bar{\bar{\alpha}}}(\beta)}(d([\alpha], \beta) t^{-\langle \beta, \alpha \rangle} \bar{t}^{-\langle \beta, \bar{\alpha}  \rangle} \bar{\bar{t}}^{-\langle \beta, \bar{\bar{\alpha}}  \rangle} u) \\ 
            \hspace{10mm} x_{s_{\alpha}s_{\bar{\alpha}}s_{\bar{\bar{\alpha}}}(\bar{\beta})}(d([\alpha], \bar{\beta}) t^{-\langle \bar{\beta}, \alpha \rangle} \bar{t}^{-\langle \bar{\beta}, \bar{\alpha}  \rangle} \bar{\bar{t}}^{-\langle \bar{\beta}, \bar{\bar{\alpha}}  \rangle} \bar{u}) & \text{ if } [\alpha] \sim A_1^2, \\
            x_{s_{\alpha + \bar{\alpha}}(\beta)}(d([\alpha], \beta) t^{-\langle \beta, \alpha \rangle} \bar{t}^{- \langle \beta, \bar{\alpha} \rangle} u) \\
            \hspace{10mm} x_{s_{\alpha + \bar{\alpha}}(\bar{\beta})}(d([\alpha], \bar{\beta}) t^{-\langle \bar{\beta}, \alpha \rangle} \bar{t}^{- \langle \bar{\beta}, \bar{\alpha} \rangle} \bar{u}) & \text{ if } [\alpha] \sim A_2.
        \end{cases} \\
        &= 
        \begin{cases}
            x_{s_{[\alpha]}([\beta])}(d([\alpha],\beta') t^{-\langle \beta', \alpha \rangle} u') & \text{if } [\alpha] \sim A_1, \\
            x_{s_{[\alpha]}([\beta])}(d([\alpha],\beta') t^{-\langle \beta', \alpha \rangle} {(\bar{t})}^{-\langle \beta', \bar{\alpha} \rangle} u') & \text{if } [\alpha] \sim A_1^2, \\
            x_{s_{[\alpha]}([\beta])}(d([\alpha],\beta') t^{-\langle \beta', \alpha \rangle} {(\bar{t})}^{-\langle \beta', \bar{\alpha} \rangle} {(\bar{\bar{t}})}^{-\langle \beta', \bar{\bar{\alpha}} \rangle} u') & \text{if } [\alpha] \sim A_1^3, \\
            x_{s_{[\alpha]}([\beta])}(d([\alpha],\beta') t^{-\langle \beta', \alpha \rangle} {(\bar{t})}^{- \langle \beta', \bar{\alpha} \rangle} u' ) & \text{if } [\alpha] \sim A_2.
        \end{cases}
    \end{align*}
    Where $u'$ is either $u$ or $\bar{u}$, depending on the representative of the class $s_{[\alpha]}([\beta])$. The last equality follows from Lemma \ref{lemma on d} and the fact that $\langle \alpha, \beta \rangle = \langle \bar{\alpha}, \bar{\beta} \rangle$ for every root $\alpha, \beta \in \Phi$. The proof for the case when $[\beta] \sim A_1^3$ is similar and will therefore be omitted.
    Now if $[\beta] \sim A_2$, then
    \begin{align*}
        & \hspace{-3mm} w_{[\alpha]}(t) x_{[\beta]}(u) w_{[\alpha]}(t)^{-1} \\
        &= w_{[\alpha]}(t) x_{[\beta]}(u_1,u_2) w_{[\alpha]}(t)^{-1} \\
        &= w_{[\alpha]}(t) x_{\beta}(u_1) x_{\bar{\beta}}(\bar{u_1}) x_{\beta + \bar{\beta}}(N_{\bar{\beta},\beta}u_2) w_{[\alpha]}(t)^{-1} \\
        &= (w_{[\alpha]}(t) x_{\beta}(u_1) w_{[\alpha]}(t)^{-1}) (w_{[\alpha]}(t) x_{\bar{\beta}}(\bar{u_1}) w_{[\alpha]}(t)^{-1}) (w_{[\alpha]}(t) x_{\beta + \bar{\beta}}(N_{\bar{\beta},\beta}u_2) w_{[\alpha]}(t)^{-1}) \\
        &=
        \begin{cases}
            x_{s_{\alpha}(\beta)}(d([\alpha], \beta) t^{-\langle \beta, \alpha \rangle} u_1) x_{s_{\alpha}(\bar{\beta})}(d([\alpha], \bar{\beta}) t^{-\langle \bar{\beta}, \alpha \rangle} \bar{u_1}) \\ 
            \hspace{10mm} x_{s_{\alpha}(\beta + \bar{\beta})}(d([\alpha], \beta + \bar{\beta}) N_{\bar{\beta},\beta} t^{-\langle \beta + \bar{\beta}, \alpha \rangle} u_2) & \text{ if } [\alpha] \sim A_1, \\
            x_{s_{\alpha}s_{\bar{\alpha}}(\beta)}(d([\alpha], \beta) t^{-\langle \beta, \alpha \rangle} \bar{t}^{-\langle \beta, \bar{\alpha} \rangle} u_1) x_{s_{\alpha}s_{\bar{\alpha}}(\bar{\beta})}(d([\alpha], \bar{\beta}) t^{-\langle \bar{\beta}, \alpha \rangle} \bar{t}^{-\langle \bar{\beta}, \bar{\alpha} \rangle} \bar{u_1}) \\ 
            \hspace{10mm} x_{s_{\alpha}s_{\bar{\alpha}}(\beta+ \bar{\beta})}(d([\alpha], \beta + \bar{\beta}) N_{\bar{\beta},\beta} t^{-\langle \beta + \bar{\beta}, \alpha \rangle} \bar{t}^{-\langle \beta + \bar{\beta}, \bar{\alpha} \rangle} u_2) & \text{ if } [\alpha] \sim A_1^2, \\
            x_{s_{\alpha}s_{\bar{\alpha}}s_{\bar{\bar{\alpha}}}(\beta)}(d([\alpha], \beta) t^{-\langle \beta, \alpha \rangle} \bar{t}^{-\langle \beta, \bar{\alpha}  \rangle} \bar{\bar{t}}^{-\langle \beta, \bar{\bar{\alpha}}  \rangle} u_1) \\ 
            \hspace{5mm} x_{s_{\alpha}s_{\bar{\alpha}}s_{\bar{\bar{\alpha}}}(\bar{\beta})}(d([\alpha], \bar{\beta}) t^{-\langle \bar{\beta}, \alpha \rangle} \bar{t}^{-\langle \bar{\beta}, \bar{\alpha}  \rangle} \bar{\bar{t}}^{-\langle \bar{\beta}, \bar{\bar{\alpha}}  \rangle} \bar{u_1}) \\
            \hspace{10mm} x_{s_{\alpha}s_{\bar{\alpha}}s_{\bar{\bar{\alpha}}}(\beta + \bar{\beta})}(d([\alpha], \beta + \bar{\beta}) N_{\bar{\beta},\beta} t^{-\langle \beta + \bar{\beta}, \alpha \rangle} \bar{t}^{-\langle \beta + \bar{\beta}, \bar{\alpha}  \rangle} \bar{\bar{t}}^{-\langle \beta + \bar{\beta}, \bar{\bar{\alpha}}  \rangle} u_2) & \text{ if } [\alpha] \sim A_1^3, \\
            x_{s_{\alpha + \bar{\alpha}}(\beta)}(d([\alpha], \beta) t^{-\langle \beta, \alpha \rangle} \bar{t}^{ - \langle \beta, \bar{\alpha} \rangle} u_1) x_{s_{\alpha + \bar{\alpha}}(\bar{\beta})}(d([\alpha], \bar{\beta}) t^{-\langle \bar{\beta}, \alpha \rangle} \bar{t}^{-\langle \bar{\beta}, \bar{\alpha} \rangle} \bar{u_1}) \\
            \hspace{10mm} x_{s_{\alpha + \bar{\alpha}}(\beta + \bar{\beta})}(d([\alpha], \beta + \bar{\beta}) N_{\bar{\beta},\beta} t^{-\langle \beta + \bar{\beta}, \alpha \rangle} \bar{t}^{-\langle \beta + \bar{\beta}, \bar{\alpha} \rangle} u_2) & \text{ if } [\alpha] \sim A_2.
        \end{cases} \\
        &= 
        \begin{cases}
            x_{s_{[\alpha]}([\beta])}(d([\alpha],\beta') t^{-\langle \beta', \alpha \rangle} \cdot u') & \text{if } [\alpha] \sim A_1, \\
            x_{s_{[\alpha]}([\beta])}(d([\alpha],\beta') t^{-\langle \beta', \alpha \rangle} {(\bar{t})}^{-\langle \beta', \bar{\alpha} \rangle} \cdot u') & \text{if } [\alpha] \sim A_1^2, \\
            x_{s_{[\alpha]}([\beta])}(d([\alpha],\beta') t^{-\langle \beta', \alpha \rangle} {(\bar{t})}^{-\langle \beta', \bar{\alpha} \rangle} {(\bar{\bar{t}})}^{-\langle \beta', \bar{\bar{\alpha}} \rangle} \cdot u') & \text{if } [\alpha] \sim A_1^3, \\
            x_{s_{[\alpha]}([\beta])}(d([\alpha],\beta') t^{-\langle \beta', \alpha \rangle} {(\bar{t})}^{- \langle \beta', \bar{\alpha} \rangle} \cdot u' ) & \text{if } [\alpha] \sim A_2.
        \end{cases}
    \end{align*}
    Where $u'$ is either $(u_1, u_2)$ or $(\bar{u}_1, \bar{u_2})$, depending on the representative of the class $s_{[\alpha]}([\beta])$. The last equality follows from Lemma \ref{lemma on d} and the fact that $\langle \alpha, \beta \rangle = \langle \bar{\alpha}, \bar{\beta} \rangle$ for every $\alpha, \beta \in \Phi$. 
\end{proof}

\begin{prop}
    \normalfont 
    For $\alpha, \beta \in \Phi, t \in R^{\star}_{[\alpha]}$ and $u \in R^{\star}_{[\beta]}$, we have the following relations:
    \[  
        w_{[\alpha]}(t) w_{[\beta]}(u) w_{[\alpha]}(t)^{-1} = 
        \begin{cases}
            w_{s_{[\alpha]}([\beta])}(d([\alpha],\beta') t^{-\langle \beta', \alpha \rangle} \cdot u') & \text{if } [\alpha] \sim A_1, \\
            w_{s_{[\alpha]}([\beta])}(d([\alpha],\beta') t^{-\langle \beta', \alpha \rangle} {(\bar{t})}^{-\langle \beta', \bar{\alpha} \rangle} \cdot u') & \text{if } [\alpha] \sim A_1^2, \\
            w_{s_{[\alpha]}([\beta])}(d([\alpha],\beta') t^{-\langle \beta', \alpha \rangle} {(\bar{t})}^{-\langle \beta', \bar{\alpha} \rangle} {(\bar{\bar{t}})}^{-\langle \beta', \bar{\bar{\alpha}} \rangle} \cdot u') & \text{if } [\alpha] \sim A_1^3, \\
            w_{s_{[\alpha]}([\beta])}(d([\alpha],\beta') t^{-\langle \beta', \alpha \rangle} {(\bar{t})}^{- \langle \beta', \bar{\alpha} \rangle} \cdot u' ) & \text{if } [\alpha] \sim A_2.
        \end{cases} 
    \]

    \noindent Where $\beta' = \beta, \bar{\beta}$ or $\bar{\bar{\beta}}$ and $u' = u, \bar{u}$ or $\bar{\bar{u}}$, respectively, depending on the representing element of class $s_{[\alpha]}([\beta]).$ 
\end{prop}

\begin{prop}
    \normalfont 
    For $\alpha, \beta \in \Phi, t \in R^{\star}_{[\alpha]}$ and $u \in R^{\star}_{[\beta]}$, we have 
    \[  
        w_{[\alpha]}(t) h_{[\beta]}(u) w_{[\alpha]}(t)^{-1} = h_{s_{[\alpha]}([\beta])} (u).
    \] 
\end{prop}

The proofs of the above two propositions are analogous to the proof of Proposition \ref{prop wxw^{-1}}; therefore, we omit them.
Finally, we conclude this section by stating the following well-known relations.

\begin{prop}[see {\cite[2.4]{KS1}}]\label{prop h(chi)xh(chi)^-1}
    \normalfont
    Let $[\alpha] \in \Phi_\rho$ and $h(\chi) \in T_\sigma (R)$. Then for $u \in R_{[\alpha]}$, we have
    \[
        h(\chi) x_{[\alpha]}(u) h(\chi)^{-1} = x_{[\alpha]} (\chi (\alpha) \cdot u).
    \]
\end{prop}

%%%%%%%%%%%%%%%%%%%%%%%%%%%%%%%%%%%%%%%%%%%%%%%%%%
%Section: Some Remarks on Known Results
%%%%%%%%%%%%%%%%%%%%%%%%%%%%%%%%%%%%%%%%%%%%%%%%%%

\section{Some Remarks on Known Results}\label{sec:SKR}

In \cite{KS2} and \cite{KS3}, K. Suzuki focuses on root system of types $\Phi_\rho \sim {}^2 A_n \ (n \geq 3), {}^2 D_n \ (n \geq 4)$ and ${}^2 E_6$. It is easy to see that the analogous versions of the main theorems in both papers hold for $\Phi_\rho \sim {}^3 D_4$. In this section, we specify the conditions under which the corresponding results are valid, and we also state a consequence of both main results which we utilize in the present paper.

First, consider the following condition on a ring $R$: 
\begin{enumerate}
    \item[\textbf{(A1)}] For any maximal ideal $\m$ of $R$, the natural map $R_{\theta} \longrightarrow (R/(\m \cap \overline{\m}))_{\theta}$ is surjective if $\Phi_\rho$ is not of type ${}^2A_{2n}$, and the natural map $\mathcal{A}(R) \longrightarrow \mathcal{A}(R/(\m \cap \overline{\m}))$ is surjective and $\mathcal{A}(R)^* \neq \phi$ if $\Phi_\rho$ is of type ${}^2A_{2n}$.
    \item[\textbf{(A2)}] For any maximal ideal $\m_\theta$ of $R_\theta$, we have $\m_\theta = R_{\theta} \cap \m_\theta R$.
\end{enumerate}

\begin{thm}[see \cite{KS2}]\label{thm:KS2}
    Let $G_\sigma(R)$ and $E'_\sigma(R)$ be as above. Assume that $\Phi_\rho \sim {}^2A_{n} \ (n \geq 3), {}^2D_{n} \ (n \geq 4)$ or ${}^2E_6$, and $R$ satisfies the condition $(A1)$ and $(A2)$ above. Then $E'_\sigma (R)$ is a normal subgroup of $G_\sigma (R)$.
\end{thm}

Now consider the following condition:
\begin{enumerate}
    \item[\textbf{(A1$'$)}] For any maximal ideal $m$ of $R$, the natural map $R_{\theta} \longrightarrow (R/(\m \cap \overline{\m} \cap \overline{\overline{\m}}))_{\theta}$ is surjective if $\Phi_\rho$ is of type ${}^3 D_{4}$. 
\end{enumerate}

\begin{rmk}
    We can state a similar result to Theorem \ref{thm:KS2} for the case of $\Phi_\rho \sim {}^3 D_4$, assuming that $R$ satisfies conditions $(A1')$ and $(A2)$. The proof follows similar lines as those in \cite{KS2}.
\end{rmk}

\begin{lemma}\label{A1&A2}
    \normalfont
    If $2$ (resp., $3$) is invertible in $R$, then $(A1)$ (resp., $(A1')$) and $(A2)$ are satisfied. 
\end{lemma}

\begin{proof}
    Assume that $2$ (resp., $3$) is invertible in $R$. Let $I = \m \cap \overline{\m}$ (resp., $I = \m \cap \overline{\m} \cap \overline{\overline{\m}}$) then $I = \bar{I}$. Let $x \in R$ such that $x + I = \bar{x} + I \implies x-\bar{x} \in I$ (resp., $x - \bar{x}, x - \bar{\bar{x}} \in I$). Set $y = (x+\bar{x})/2 \in R_\theta$ (resp., $y = (x + \bar{x} + \bar{\bar{x}})/3 \in R_\theta$). Then $x - y = (x - \bar{x})/2 \in I$ (resp., $x-y = ((x-\bar{x}) + (x-\bar{\bar{x}}))/3 \in I$), so $x+I = y+I$. Since $y \in R_\theta$, it serves as a pre-image of $x+I$. Therefore, the map $R_\theta \longrightarrow (R/I)_\theta$ is surjective. 

    We now prove that if $2$ is invertible in $R$, then the natural map $\mathcal{A}(R) \longrightarrow \mathcal{A}(R/I)$ is surjective. Let $(x_1 + I, x_2 + I) \in \mathcal{A}(R/I)$. Then 
    \begin{align*}
        (x_1 + I) (\bar{x}_1 + I) = (x_2 + I) + (\bar{x}_2 + I) 
        \implies & x_1 \bar{x}_1 + I = (x_2 + \bar{x}_2) + I \\
        \implies & x_1 \bar{x}_1 - (x_2 + \bar{x}_2) \in I.
    \end{align*}
    Set $y_1 = x_1$ and $y_2 = x_2 + (x_1 \bar{x}_1)/ 2 - (x_2 + \bar{x}_2)/2$. Then $(y_1, y_2) \in \mathcal{A}(R)$, and clearly $(y_1 + I, y_2 + I) = (x_1 + I, x_2 + I)$, showing that $(y_1, y_2) \in \mathcal{A}(R)$ serves as a pre-image of $(x_1 + I, x_2 + I)$. Hence, the natural map $\mathcal{A}(R) \longrightarrow \mathcal{A}(R/I)$ is a surjective.

    We now return to the assumption that $2$ (resp., $3$) is invertible in $R$, and verify condition $(A2)$. It is clear that $\m_\theta \subset R_\theta \cap \m_\theta R$. Let $x \in R_\theta \cap \m_\theta R.$ Then $x = \bar{x}$ and $x = \sum_{i=1}^k m_i x_i$ where $m_i \in \m_\theta$ and $x_i \in R$. Define $y_i = (x_i + \bar{x}_i)/2$ (resp., $y_i = (x_i + \bar{x}_i + \bar{\bar{x}}_i)/3)$ and set $y = \sum_{i=1}^{k} m_i y_i \in \m_\theta$. Then $x - y = \sum_{i=1}^k m_i (x_i - y_i) = \sum_{i=1}^{k} m_i (x_i - \bar{x_i})/2 = (x - \bar{x})/2 = 0$ (resp., $x-y = \sum_{i=1}^{k} m_i ((x_i - \bar{x_i}) + (x_i - \bar{\bar{x_i}}))/3 = ((x - \bar{x}) + (x- \bar{\bar{x}}))/3 = 0$). Therefore $x = y \in \m_\theta$, as desired.
\end{proof}

\begin{cor}\label{cor:KS2}
    Assume that $1/2 \in R$ if $\Phi_\rho \sim {}^2 A_{n} \ (n \geq 3), {}^2 D_n \ (n \geq 4)$ or ${}^2 E_6$, and that $1/3 \in R$ if $\Phi_\rho \sim {}^3 D_4$. Then $E'_\sigma (R)$ is a normal subgroup of $G_\sigma (R)$.
\end{cor}

Recall that in Section \ref{Subsec:TCG}, we defined the set $\text{Hom}_1 (\Lambda_\pi, R^*) = \{ \chi \in \text{Hom} (\Lambda_\pi, R^*) \mid \chi = \bar{\chi}_\sigma \}$. We also define $\text{Hom}_1(\Lambda_\pi / \Lambda_r, R^*)$ as $\{ \chi \in \text{Hom}_1 (\Lambda_\pi, R^*) \mid \chi|_{\Lambda_r} = 1 \}$. The following theorem is a special case of the main theorem of \cite{KS3}.
\begin{thm}\label{thm:KS3}
    \normalfont
    Let $G=G_\sigma (R)$ be a twisted Chevalley group of type $\Phi_\rho \sim {}^2 A_{n} \ (n \geq 3), {}^2 D_n \ (n \geq 4)$ or ${}^2 E_6$ and let $E = E'_\sigma(R)$ be its elementary subgroup. Assume that $1/2 \in R$. Then $Z(G) = C_G (E) = \text{Hom}_1(\Lambda_\pi / \Lambda_r, R^*)$, where $Z(G)$ is a centre of $G$ and $C_G (E)$ is a centralizer of $E$ in $G$. 
\end{thm}

\begin{rmk}
    As before, we can state a similar result to Theorem \ref{thm:KS3} for the case of $\Phi_\rho \sim {}^3D_4$, assuming $1/3 \in R$. The proof follows similar lines as those in \cite{KS3}.
\end{rmk}

%%%%%%%%%%%%%%%%%%%%%%%%%%%%%%%%%%%%%%%%%%%%%%%%%%
%Section: Certain Subgroups of $G_\sigma(R)$
%%%%%%%%%%%%%%%%%%%%%%%%%%%%%%%%%%%%%%%%%%%%%%%%%%

\section{Certain Subgroups of \texorpdfstring{$G_\sigma(R)$}{G(R)}}\label{sec:E=UHV}

Let $R$ be a commutative ring with unity and let $\theta$ be an automorphism of a ring $R$ of order $2$ or $3$. Let $J$ be a $\theta$-invariant ideal of $R$ (that is, $J$ is an ideal of $R$ such that $\theta (J) \subset J$). For $[\alpha] \in \Phi_{\rho}$, define \[ J_{[\alpha]} = \begin{cases}
    J_{\theta} & \text{if } [\alpha] \sim A_1, \\
    J & \text{if } [\alpha] \sim A_1^2 \text{ or } A_1^3, \\
    \mathcal{A}(J) & \text{if } [\alpha] \sim A_2; \\
\end{cases}\]
where $J_{\theta} = \{ r \in J \mid r = \overline{r} \} = J \cap R_{\theta}$ and $\mathcal{A}(J) = \{ (a,b) \in \mathcal{A}(R) \mid a,b \in J \}.$

The natural projection map $R \longrightarrow R/J$ gives the canonical map $\phi: G_\sigma (R) \longrightarrow G_{\sigma} (R/J)$. Write $G_\sigma(J) = \ker{\phi}$ and $G_{\sigma}(R,J) = \phi^{-1}(Z(G_{\sigma}(R/J))).$ Let $E'_{\sigma}(J)$ denote the subgroup of $E'_{\sigma}(R) \cap G_{\sigma}(J)$ generated by all $x_{[\alpha]}(t)$ where $[\alpha] \in \Phi_{\rho}$ and $t \in J_{[\alpha]}$. Let $E'_{\sigma}(R,J)$ be the normal subgroup of $E'_{\sigma}(R)$ generated by $E'_{\sigma}(J)$. Note that $E'_{\sigma}(R,J)$ is also a subgroup of $E'_{\sigma}(R) \cap G_{\sigma}(J)$ as the later subgroup is normal in $E'_{\sigma}(R)$ and it contains $E'_{\sigma}(J)$.

Let $U_\sigma (J)$ (resp., $U^{-}_\sigma (J)$) be the subgroup of $E'_\sigma (R)$ generated by $x_{[\alpha]}(t)$ (resp., $x_{-[\alpha]}(t)$) where $[\alpha] \in \Phi^{+}_\rho$ and $t \in J_{[\alpha]}$. Define $T_\sigma (J) = G_\sigma(J) \cap T_\sigma (R), T_\sigma (R, J) = G_\sigma (R, J) \cap T_\sigma(R), H_\sigma(J) = E_\sigma (J) \cap T_\sigma (R), H_\sigma (R,J) = E_\sigma (R,J) \cap T_\sigma (R), H'_\sigma(J) = E'_\sigma (J) \cap T_\sigma (R)$ and $H'_\sigma (R,J) = E'_\sigma (R,J) \cap T_\sigma (R).$

\begin{lemma}\label{structure of U}
    \normalfont
    Let $J$ be any $\theta$-invariant ideal of $R$. Then each element of $U_\sigma (J)$ is uniquely expressible in the form
        $$ x_{[\alpha_1]}(t_1) \dots x_{[\alpha_n]}(t_n)$$ where $[\alpha_i] \in \Phi^+_\rho$ and $t_i \in J_{[\alpha_i]}$, the ordering of the roots is arbitrary chosen and fixed once for all. 
\end{lemma}

\begin{proof}
    The proof is an easy consequence of the Chevalley commutator formula and is therefore omitted.
\end{proof}

\begin{rmk}
    \normalfont
    We can state and prove a similar result for $U^{-}_\sigma (J)$.
\end{rmk}

\begin{rmk}
    Note that $U_\sigma (J) = E'_\sigma (R,J) \cap U_\sigma (R)$ and $U^{-}_\sigma (J) = E'_\sigma (R,J) \cap U^{-}_\sigma (R)$. This can be seen as follows: Clearly, $U_\sigma (J) \subset E'_\sigma (R,J) \cap U_\sigma (R)$. For the reverse inclusion, let $x \in E'_\sigma (R,J) \cap U_\sigma (R)$. Since $E'_\sigma (R,J) \subset G_\sigma (J)$ we have $x \equiv 1$ (mod $J$). From the uniqueness in the above lemma, we conclude that $x \in U_\sigma (J)$. 
\end{rmk}

\begin{lemma}\label{inUHV}
    \normalfont
    Let $J$ be a $\theta$-invariant ideal of $R$ contained in $rad (R)$, the Jacobson radical of $R$. Then for any $[\alpha] \in \Phi_\rho, t \in J_{[\alpha]}, s \in R_{[\alpha]},$ we have 
    $$ x_{[\alpha]}(s) x_{-[\alpha]}(t) x_{[\alpha]}(s)^{-1} = x_{[\alpha]}(a) h x_{-[\alpha]}(b) $$ where $a,b \in J_{[\alpha]}$ and $h \in H'_\sigma (R,J)$.
\end{lemma}

\begin{proof}
    If $[\alpha] \sim A_1, A_1^2$ or $A_1^3$ then for given $t \in J_{[\alpha]}$ and $s \in R_{[\alpha]}$, we have $(1-st) \in R_{[\alpha]}^*$. In this case, we can take $a = -t s^2 (1-st)^{-1}, b= t (1-st)^{-1}$ and $h = h_{[\alpha]}((1-st)^{-1})$. Clearly, $a, b \in J_{[\alpha]}$. But then $h \in E'_\sigma (R,J)$ and hence $h \in H'_\sigma (R,J)$.

    If $[\alpha] \sim A_2$ then for given $t = (t_1, t_2) \in \mathcal{A}(J)$ and $s=(s_1, s_2) \in \mathcal{A}(R)$, we have $1-(s_1 \bar{t_1} - \bar{s_2}t_2) \in R^*.$ Therefore, the equation
    \begin{small}
        \[
            \begin{pmatrix}
                1 & \bar{s}_1 & s_2 \\
                0 & 1 & s_1 \\
                0 & 0 & 1
            \end{pmatrix}
            \begin{pmatrix}
                1 & 0 & 0 \\
                t_1 & 1 & 0 \\
                t_2 & \bar{t}_1 & 1
            \end{pmatrix}
            \begin{pmatrix}
                1 & -\bar{s_1} & \bar{s_2} \\
                0 & 1 & -s_1 \\
                0 & 0 & 1
            \end{pmatrix} = 
            \begin{pmatrix}
                1 & \bar{a_1} & a_2 \\
                0 & 1 & a_1 \\
                0 & 0 & 1
            \end{pmatrix}
            \begin{pmatrix}
                \bar{u} & 0 & 0 \\
                0 & u \bar{u}^{-1} & 0 \\
                0 & 0 & u^{-1}
            \end{pmatrix}
            \begin{pmatrix}
                1 & 0 & 0 \\
                b_1 & 1 & 0 \\
                b_2 & \bar{b_1} & 1
            \end{pmatrix} 
        \]
    \end{small}
    has a solution for $u, a_1, a_2, b_1, b_2$ and it is given by $u = (1-(\bar{t_1}s_1 - t_2\bar{s_2}))^{-1}, a_1= (t_1\bar{s_2} - \bar{t_1} {s_1}^2 + t_2 s_1 \bar{s_2}) (1-(\bar{t_1}s_1 - t_2\bar{s_2}))^{-1}, a_2 = (t_1 \bar{s_1} \bar{s_2} - \bar{t_1} s_1 s_2 + t_2 s_2 \bar{s_2}) (1-(\bar{t_1}s_1 - t_2\bar{s_2}))^{-1}, b_1= (t_1 - s_1\bar{t_2}) (1-(t_1 \bar{s_1} - \bar{t_2} {s_2}))^{-1}, b_2= t_2 (1-(\bar{t_1}s_1 - t_2\bar{s_2}))^{-1}$. By simple calculation we can see that $(a_1, a_2), (b_1, b_2) \in \mathcal{A}(J)$. But then $$h_{[\alpha]}(u) = x_{[\alpha]}(a)^{-1} x_{[\alpha]}(s) x_{-[\alpha]}(t) x_{[\alpha]}(s)^{-1} x_{-[\alpha]}(b)^{-1} \in E'_\sigma (R, J).$$ Since, $h = h_{[\alpha]}(u) \in T_\sigma (R)$, we have $h = h_{[\alpha]}(u) \in H'_\sigma (R,J)$.
\end{proof}

Let $[\alpha_i]$ be the simple roots of $\Phi_\rho$. We define the height $ht ([\alpha]):= \sum_{i=1}^l m_i$ of a root $[\alpha]= \sum_{i=1}^l m_i [\alpha_i]$ in $\Phi_\rho$. The order of the roots is \textit{regular} if the height $ht ([\alpha])$ is an increasing function of $[\alpha]$. From now on we fix a regular ordering of the roots in $\Phi_\rho$.

\begin{lemma}\label{inUV}
    \normalfont
    Let $J$ be any $\theta$-invariant ideal of $R$. Then for any $[\alpha], [\beta] (\neq [\alpha]) \in \Phi^{+}_\rho$ and $t \in J_{[\alpha]}, s \in R_{[\beta]},$ we have 
    $$ x_{-[\beta]}(s) x_{[\alpha]}(t) x_{-[\beta]}(s)^{-1} = xy$$
    where $x \in U_\sigma (J)$ and $y$ is a product of $x_{-[\gamma]}(u)$'s ($u \in J_{[\gamma]}$) in $U^{-}_\sigma (J)$ such that $-[\gamma] > - [\beta]$.
\end{lemma}

\begin{proof}
    Immediate from the Chevalley commutator formula for $[x_{[\alpha]}(t)^{-1}, x_{-[\beta]}(s)].$
\end{proof}

\begin{prop}\label{lavidecomposition}
    \normalfont
    Let $J$ be a $\theta$-invariant ideal of $R$ contained in $rad (R)$. Then $E'_\sigma (R,J) = U_\sigma(J) H'_\sigma(R,J) U^{-}_\sigma(J)$.
\end{prop}

\begin{proof}

The proof is similar to that of $2.8$ in \cite{EA2} and $2.1$ in \cite{EA&KS}. However, for the convenience of the reader, we will provide the full proof here.

We write $U = U_\sigma (J), H = H'_\sigma (R, J)$ and $V=U^{-}_\sigma (J)$. Clearly, $UHV \subset E'_\sigma (R, J)$. For converge, it is enough to show that $UHV$ is a normal subgroup of $E'_\sigma (R)$ because then $E'_\sigma (J) \subset UHV$ and hence $E'_\sigma (R, J) \subset UHV$. 

First, we will show that $UHV$ is a subgroup of $E'_\sigma (R)$. In other words, we will show that $g h^{-1} \subset UHV$ for every $g, h \in UHV$. For that, it is enough to show that $g (UHV) \subset UHV$ for any element $g$ of the form $x_{[\beta]}(t) \in U, h_{[\beta]}(t) \in H$ and $x_{-[\beta]}(t) \in U^{-}$. If $g = x_{[\beta]} (t)$, then by Lemma \ref{structure of U}, we have $x_{[\beta]}(t) U \subset U$ and hence $x_{[\beta]}(t) UHV \subset UHV$. Similarly, if $g = h_{[\beta]}(t)$, then from Proposition \ref{prop h(chi)xh(chi)^-1}, we have $h_{[\beta]}(t) U \subset U H$ and hence $h_{[\beta]}(t) UHV \subset UHV$. Finally, if $g= x_{-[\beta]}(t),$ we must show that 
\begin{align}
    x_{-[\beta]}(t) U \subset UHV. \label{eq1}
\end{align}
Because then $x_{-[\beta]}(t) UHV \subset (UHV)HV = UH(VH)V = UH(HV)V = UHV$, the second last equality is follows from Proposition \ref{prop h(chi)xh(chi)^-1}.

To prove (\ref{eq1}), we must show that $x_{-[\beta]}(t)x \in UHV$ for every $x \in U$. Write $$x = x_{[\alpha_1]}(t_1) \dots x_{[\alpha_n]}(t_n) \in U,$$ where each $[\alpha_i] \in \Phi_\rho^+ \ (i=1,\dots,n)$ with $[\alpha_1] > \dots > [\alpha_n]$. Let $m = ht([\beta])$. We will use double induction on the pair $(m, n)$ to prove our result. If $n = 1$ then for any pair $(m,1)$ the result is follows from Lemma \ref{inUHV}, if $[\alpha_1] = [\beta]$; and from Lemma \ref{inUV}, if $[\alpha_1] \neq [\beta]$. Assume that for all $1 \leq k \leq n-1$, the result is true for the pairs $(m,k)$ for all $m \geq 1$. We will show that it is also true for the pair $(m,n)$ for all $m \geq 1$, that is, we will show that $x_{-[\beta]}(t) x_{[\alpha_1]}(t_1) \dots x_{[\alpha_n]}(t_n) \in UHV$. 
If $[\beta] = [\alpha_1]$ then, by Lemma \ref{inUHV}, we have $x_{-[\beta]}(t)x_{[\alpha_1]}(t_1) = x_{[\alpha_1]}(t'_1) h x_{-[\beta]}(t')$, where $h \in H$. 
Thus, by induction hypothesis, 
\begin{align*}
    x_{-[\beta]}(t) x_{[\alpha_1]}(t_1) \dots x_{[\alpha_n]}(t_n) &= x_{[\alpha_1]}(t'_1) h x_{-[\beta]}(t') x_{[\alpha_2]}(t_2) \dots x_{[\alpha_n]}(t_n) \\
    &\in U H (UHV) = U(HU)HV = U(UH)HV = UHV.
\end{align*}
Similarly, if $[\beta] \neq [\alpha_1]$ then, by Lemma \ref{inUV}, we have $x_{-[\beta]}(t)x_{[\alpha_1]}(t_1) = x y x_{-[\beta]}(t)$, where $x \in U$ and $y \in V$ as in Lemma \ref{inUV}. Thus, by induction hypothesis, $$x_{-[\beta]}(t) x_{[\alpha_1]}(t_1) \dots x_{[\alpha_n]}(t_n) = x y x_{-[\beta]}(t) x_{[\alpha_2]}(t_2) \dots x_{[\alpha_n]}(t_n) \in U (UHV) = UHV.$$

Now we will show that $UHV$ is a normal subgroup of $E'_\sigma (R)$. It is suffices to show that $x_{\pm [\alpha]}(t) UHV x_{\pm [\alpha]}(t)^{-1} \subset UHV$ for any root $[\alpha] \in \Phi_\rho^+$ and $t \in R_{[\alpha]}$. Clearly, $x_{[\alpha]}(t) U x_{[\alpha]}(t)^{-1} \subset U$ and $x_{[\alpha]}(t) H x_{[\alpha]}(t)^{-1} \subset UH$ (see Proposition \ref{prop h(chi)xh(chi)^-1}). We will show that $x_{[\alpha]}(t) V x_{[\alpha]}(t)^{-1} \in UHV$. Let $y \in V$ be such that $y = x_{-[\alpha_1]}(t_1) \dots x_{-[\alpha_n]}(t_n)$, where $[\alpha_i] \in \Phi^+_\rho, t_i \in J_{[\alpha_i]} \ (i=1,\dots, n)$. Then 
\begin{align*}
    x_{[\alpha]}(t) y x_{[\alpha]}(t)^{-1} &= x_{[\alpha]}(t) x_{[\alpha_1]}(t_1) \dots x_{[\alpha_n]}(t_n) x_{[\alpha]}(t)^{-1} \\
    &= (x_{[\alpha]}(t) x_{[\alpha_1]}(t_1) x_{[\alpha]}(t)^{-1}) \dots (x_{[\alpha]}(t) x_{[\alpha_n]}(t_n)) x_{[\alpha]}(t)^{-1}) \in UHV.
\end{align*}
The containment is follows because if $[\alpha_i] = [\alpha]$ then, by Lemma \ref{inUHV}, $x_{[\alpha]}(t)x_{[\alpha_i]}(t_i)x_{[\alpha]}(t)^{-1} \in UHV$, and if $[\alpha_i] \neq [\alpha]$ then, by Lemma \ref{inUV}, $x_{[\alpha]}(t)x_{[\alpha_i]}(t_i)x_{[\alpha]}(t)^{-1} \in UV \subset UHV$, and also the fact that $UHV$ is a subgroup of $G$. Finally, we have
\begin{align*}
    x_{[\alpha]}(t) UHV x_{[\alpha]}(t)^{-1} &= (x_{[\alpha]}(t) U x_{[\alpha]}(t)^{-1}) (x_{[\alpha]}(t) H x_{[\alpha]}(t)^{-1}) (x_{[\alpha]}(t) V x_{[\alpha]}(t)^{-1}) \\
    &\subset (U) (UH) (UHV) = U(HU)HV = U(UH)HV = UHV.
\end{align*}
Similarly, one can show that $x_{-[\alpha]}(t) UHV x_{-[\alpha]}(t)^{-1} \subset UHV$. 
\end{proof}

\begin{prop}\label{P(J) and Q(J)}
    Let $J$ be a $\theta$-invariant ideal of $R$ contained in $rad(R)$.
    Set $P_\sigma(J) = U_\sigma(J) T_\sigma(R) U^{-}_\sigma(R)$ and $Q_\sigma(J) = U_\sigma(R) T_\sigma(R) U^{-}_\sigma(J)$. Then $P_\sigma(J)$ and $Q_\sigma(J)$ are subgroups of $G_\sigma(R)$.
\end{prop}

\begin{proof}
    Note that $E'_\sigma(R,J) = U_\sigma(J) H'_\sigma(R,J) U^{-}_\sigma(J)$ is normalized by $E'_\sigma(R)$ and $T_\sigma(R)$. Set $B_\sigma(R) = U_\sigma(R) T_\sigma(R) = T_\sigma(R) U_{\sigma}(R)$ and $B_\sigma^{-}(R) = U_\sigma^{-}(R) T_\sigma(R) = T_\sigma(R) U_\sigma^{-}(R)$. Clearly, both $B_\sigma (R)$ and $B^{-}_\sigma (R)$ are subgroups of $G_\sigma(R).$ By a similar argument as in Proposition~\ref{lavidecomposition}, we have 
    $$P_\sigma(J) = E'_\sigma(R,J)B^{-}_\sigma(R) = B^{-}_\sigma(R) E'_\sigma(R,J) \text{ and } Q_\sigma(J) = B_\sigma(R) E'_\sigma(R,J) = E'_\sigma(R,J)B_\sigma(R).$$ Therefore, $P_\sigma(J)$ and $Q_\sigma(J)$ are subgroups of $G_\sigma(R)$.
\end{proof}

For any $\theta$-invariant ideal $J$ of $R$, we have the canonical map $\phi: G_\sigma (R) \longrightarrow G_\sigma (R/J)$ as mentioned above. We now consider the canonical map $\phi': G (R) \longrightarrow G(R/J)$. Clearly, $\phi' |_{G_\sigma(R)} = \phi$. Let $G(J) = \ker(\phi')$ and $G(R,J) = \phi'^{-1} (Z(G(R/J)))$, where $Z(G(R/J))$ is the center of the group $G(R/J)$. Let \(U(J)\) (respectively, \(U^-(J)\)) be the subgroup of \(G(R)\) generated by \(x_{\alpha}(t)\) for \(t \in J\) and \(\alpha \in \Phi^+\) (respectively, \(\alpha \in \Phi^-\)). Define \(T(J) = G(J) \cap T(R)\) and \( T(R,J) = G(R,J) \cap T(R) \).

\begin{lemma}\label{lemma on T(J)}
    \normalfont
    Let $J$ be any $\theta$-invariant ideal of $R$. Then 
    \begin{enumerate}[(a)]
        \item the subgroup $T_\sigma (J)$ of $T_\sigma (R)$ is generated by all $h(\chi)$ such that $\chi = \bar{\chi}_\sigma$ and $\chi (\mu) \equiv 1 $ (mod $J$) for every $\mu \in \Omega_\pi$, where $\Omega_\pi$ is a set of all weights of the representation $\pi$.
        \item the subgroup $T_\sigma (R, J)$ of $T_\sigma (R)$ is generated by all $h(\chi)$ such that $\chi = \bar{\chi}_\sigma$ and $\chi (\alpha) \equiv 1 $ (mod $J$) for every $\alpha \in \Phi$.
    \end{enumerate}
\end{lemma}

\begin{proof}
    To prove our result it is enough to prove that 
    \begin{enumerate}[(a)]
        \item $T_\sigma (J) = T(J) \cap G_\sigma (R)$ and $T(J)$ is a subgroup of $T(R)$ generated by all $h(\chi)$ such that $\chi (\mu) \equiv 1$ (mod $J$) for every $\mu \in \Omega_\pi$.
        \item $T_\sigma (R, J) = T(R, J) \cap G_\sigma (R)$ and $T(R, J)$ is a subgroup of $T(R)$ generated by all $h(\chi)$ such that $\chi (\alpha) \equiv 1$ (mod $J$) for every $\alpha \in \Phi$.
    \end{enumerate}
    
    Since $G_\sigma (J) = G(J) \cap G_\sigma (R)$ and $T_\sigma (R) = T(R) \cap G_\sigma (R)$, the first assertion of part (a) is clear. The second assertion of part (a) directly follows from the definition of $T(J)$ and the action of $h(\chi)$ on the weight spaces corresponding to the representation $\pi$ (see \ref{Subsec:CG}). 

    Note that the center of $G(R/J)$ is $Z(G(R/J)) = \text{Hom }(\Lambda_\pi / \Lambda_r, (R/J)^*)$ (see \cite{EA&JH}) and the center of $G_\sigma (R/J)$ is $Z(G_\sigma (R/J)) = \text{Hom}_1(\Lambda_\pi / \Lambda_r, (R/J)^*)$ (see Theorem \ref{thm:KS3}). Therefore, $Z(G_\sigma (R/J)) = Z(G(R/J)) \cap G_\sigma (R/J)$. But then $G_\sigma (R, J) = G (R, J) \cap G_\sigma (R)$ and hence the first assertion of part (b) follows. For the second assertion of part (b), let $T' (R, J)$ be the subgroup of $T(R)$ generated by all $h(\chi)$ such that $\chi (\alpha) \equiv 1 $ (mod $J$), for every $\alpha \in \Phi$. We want to show that $T' (R, J) = T (R, J)$. 
    Let $h(\chi) \in T (R, J) = G (R, J) \cap T (R)$. Since $G (R, J)$ is normal subgroup of $G (R)$, we have $[h(\chi), x_{\alpha}(1)] \in G (R, J)$ for all $\alpha \in \Phi$, that is $x_{\alpha}(\chi(\alpha) - 1) \in G (R, J)$ for all $\alpha \in \Phi$. Hence, by the main theorem of \cite{EA&JH}, $\chi(\alpha) \equiv 1$ (mod $J$) for all $\alpha \in \Phi$. Thus, $T (R, J) \subset T'(R, J)$. For the reverse inclusion, let $h(\chi) \in T' (R, J)$. Then $\chi (\alpha) \equiv 1 $ (mod $J$), for every $\alpha \in \Phi$ and hence $\phi' (h(\chi)) \in Z(G (R/J))$ (again by the main theorem of \cite{EA&JH}). That is, $h (\chi) \in G(R,J)$. Thus, we have $T'(R, J) \subset T(R, J)$.
\end{proof}

%%%%%%%%%%%%%%%%%%%%%%%%%%%%%%%%%%%%%%%%%%%%%%%%%%
%%%%%%%%%%%%%%%%%%%%%%%%%%%%%%%%%%%%%%%%%%%%%%%%%%

\begin{prop}\label{G=UTV}
    \normalfont
    Let $J$ be a $\theta$-invariant ideal of $R$ contained in $rad (R)$. Then $$G_\sigma (J) = U_\sigma(J) T_\sigma(J) U^{-}_\sigma(J) \subset G'_\sigma (R).$$
\end{prop}

\begin{proof}
    By $2.3$ of \cite{EA&KS}, we have $G(J) = U(J) T(J) U^-(J)$. Note that $G_\sigma (J) = G(J) \cap G_\sigma (R)$, that is, $G_\sigma(J) = \{ U(J) T(J) U^-(J) \} \cap G_\sigma (R)$. Also, we have $U_\sigma (J) = \{ x \in U(J) \mid \sigma(x) = x \}, U^{-}_\sigma (J) = \{ x \in U^{-}(J) \mid \sigma(x) = x \}$ and $T_\sigma (J) = T_\sigma (R) \cap G_\sigma (J) = T (J) \cap G_\sigma (R)$ (by the proof of Lemma \ref{lemma on T(J)}). Since $U_\sigma (J) \cap U^{-}_\sigma (J) = U_\sigma (J) \cap T_\sigma (J) = U^{-}_\sigma (J)\cap T_\sigma (J) = \{ 1 \}$, we have 
    $$G_\sigma (J) = U_\sigma(J) T_\sigma(J) U^{-}_\sigma(J).$$
    It is also clear that $G_\sigma (J) \subset G'_\sigma(R)$. 
\end{proof}

\begin{prop}\label{G=G'}
    \normalfont
    Let $R$ be a semi-local ring. Then $G_\sigma (R) = G^{0}_\sigma (R) = G'_\sigma (R)$. 
\end{prop}

\begin{proof}
    Let $J = rad(R)$. Since $R$ is semi-local, it has finitely many maximal ideals, say $\mathfrak{m}_1, \dots, \mathfrak{m}_k$. Therefore, by the Chinese remainder theorem, we have $R/J = \prod_{i=1}^k R/\mathfrak{m}_i$. 
    Write $\bar{\mathfrak{m}}_i = \theta (\mathfrak{m}_i)$. Set 
    \[
        J_i = \begin{cases}
            \mathfrak{m}_i & \text{ if } \mathfrak{m}_i = \bar{\mathfrak{m}}_i, \\
            \mathfrak{m}_i \cap \bar{\mathfrak{m}}_i & \text{ if } \mathfrak{m}_i \neq \bar{\mathfrak{m}}_i \text{ and } o(\theta) = 2, \\
            \mathfrak{m}_i \cap \bar{\mathfrak{m}}_i \cap \bar{\bar{\mathfrak{m}}}_i & \text{ if } \mathfrak{m}_i \neq \bar{\mathfrak{m}}_i \text{ and } o(\theta) = 3. 
        \end{cases}
    \]
    By the proof of Proposition $2.2$ of \cite{KS2}, we have $G_\sigma(R/J_i) = G'_\sigma (R/J_i)$ (that proof only addresses the case where $o(\theta) = 2$. However, the proof for the case where $o(\theta) = 3$ follows a similar structure).
    
    Since $R/J = \prod_{i=1}^k R/\mathfrak{m}_i = \prod_{i=1}^l R/J_i$, we have $G_\sigma(R/J) = \prod_{i=1}^l G_\sigma(R/J_i).$ But then $$G_\sigma(R/J) = \prod_{i=1}^l G_\sigma (R/J_i) = \prod_{i=1}^l G'_\sigma (R/J_i) = G'_\sigma (R/J).$$ On the other hand, from Proposition \ref{G=UTV}, $G_\sigma(J) \subset G'_\sigma (R)$. Therefore $G_\sigma(R) \subset G'_\sigma (R)$. Hence $G'_\sigma(R) = G^{0}_\sigma (R) = G_\sigma(R)$, as desired. 
\end{proof}

\begin{cor}\label{G(R,J)=UTV}
    \normalfont
    Let $R$ be a semi-local ring and let $J$ be a $\theta$-invariant ideal of $R$ contained in $rad(R)$. Then 
    $$G_\sigma (R,J) = U_\sigma (J) T_\sigma (R,J) U^{-}_\sigma (J).$$
\end{cor}

\begin{proof}
    Since $G_\sigma (J)$ is normalized by $T_\sigma (R,J)$, we conclude that $G_\sigma (J) T_\sigma (R,J)$ is a subgroup of $G_\sigma (R,J)$. On the other hand, by the above proposition, we have $G_\sigma(R) = G'_\sigma (R) = E'_\sigma (R) T_\sigma(R)$. Let $z \in G_\sigma (R,J) \subset G_\sigma(R)$. Then there exist $x \in E'_\sigma (R)$ and $y \in T_\sigma (R)$ such that $z=xy$. Now consider the canonical map $\phi: G_\sigma(R) \longrightarrow G_\sigma(R/J)$. Then $\phi (z) = \phi (x) \phi(y) \in Z(G_\sigma(R/J))$. Since $\phi(y) \in T_\sigma(R/J)$, we obtain $\phi(x) \in T_\sigma(R/J)$. To be precious, $\phi(x) \in H'_\sigma(R/J)$ as $x \in E'_\sigma (R)$. But then there exist $h \in H'_\sigma (R)$ such that $\phi (h) = \phi (x)$, that is, $xh^{-1} \in G_\sigma (J)$. Hence $hy \in T_\sigma(R,J)$. Write $x'= xh^{-1}$ and $y'= h y$. Then $z= xy = x' y' \in G_\sigma (J) T_\sigma(R,J)$. Therefore, $G_\sigma(R,J) = G_\sigma (J) T_\sigma(R,J) = U_\sigma (J) T_\sigma (R,J) U^{-}_\sigma (J)$, the last equality is due to Proposition \ref{G=UTV}.
\end{proof}

%%%%%%%%%%%%%%%%%%%%%%%%%%%%%%%%%%%%%%%%%%%%%%%%%%
%Section: The Subgroup $E'_\sigma(R,J)$
%%%%%%%%%%%%%%%%%%%%%%%%%%%%%%%%%%%%%%%%%%%%%%%%%%

\section{The Subgroup \texorpdfstring{$E'_\sigma(R,J)$}{E(R,J)}}\label{sec:E(R,J)}

In this section, we will explore several important properties of the subgroup $E'_\sigma(R,J)$. Similar properties have been studied by L. N. Vaserstein in \cite{LV} for the case of Chevalley groups. Using his ideas, we will state and prove analogous properties for twisted Chevalley groups. For the remainder of this paper, we adopt the following conventions.
 
\begin{conv}
    Assume that $\Phi_\rho$ is irreducible and the rank of $\Phi_\rho > 1$. Any ideal $J$ of $R$ is $\theta$-invariant. If $o(\theta) = 2$ then assume that $1/2 \in R$ and if $o(\theta) = 3$ then assume that $1/2 \in R$ and $1/3 \in R$. 
\end{conv}

%%%%%%%%%%%%%%%%%%%%%%%%%%%%%%%%%%%%%%%%%%%%%%%%%%

\begin{prop}\label{normal}
    For any ideal $J$ of $R$, the subgroup $E'_\sigma (R,J)$ of $G_\sigma (R)$ is normal. In other words,
    \[
        [G_\sigma (R), E'_\sigma (R,J)] \subset E'_\sigma (R,J).
    \]
\end{prop}

%%%%%%%%%%%%%%%%%%%%%%%%%%%%%%%%%%%%%%%%%%%%%%%%%%

\begin{proof}
    First, consider the case where $J = R$. Then, by definition, $E'_{\sigma}(R, J) = E'_\sigma (R)$. The result in this case follows from the fact that $E'_\sigma (R)$ is a normal subgroup of $G_{\sigma}(R)$ (see Corollary~\ref{cor:KS2}). 

Now suppose $J \subsetneq R$. Let $h \in G_\sigma (R)$ and $g \in E'_\sigma (R, J)$. 
We want to prove that $hgh^{-1} \in E'_\sigma (R,J)$. 
We consider the ring $R':= \{(r,s) \in R \times R \mid r-s \in J\}$ and its ideal $J':= \{ (r,0) \in R \times R \mid r \in J \}$.
The automorphism $\theta$ of the ring $R$ can be naturally induced to an automorphism of $R'$, and we denote it by the same letter $\theta$.
Therefore, the group $G_\sigma(R')$ makes sense.
Consider an element $h'':= (h,h)$ of the group $G_\sigma(R') \subset G_\sigma (R) \times G_{\sigma} (R).$ 
Observe that $E'_\sigma (R, J)$ is embedded into the group $E'_\sigma (R', J')$ by $x \mapsto x':=(x,1).$ (This can be seen as follows: There is a natural embedding from $E'_\sigma (J)$ into $E'_\sigma (J')$ given by $x \mapsto (x,1)$. Now any $y \in E'_\sigma(R,J)$ is can be written as a product of the form $\prod_{i=1}^n g_i x_i g_i^{-1}$, where $x_i \in E'_{\sigma}(J)$ and $g_i \in E'_\sigma(R)$. But then $(y,1) = (\prod_{i=1}^n g_i x_i g_i^{-1},1) = \prod_{i=1}^n (g_i,g_i) (x_i,1) (g_i,g_i)^{-1} \in E'_\sigma(R',J')$, as desired.)

Now we claim that $E'_\sigma (R', J') = E'_\sigma (R') \cap G_\sigma (J')$. Clearly, by definition, $E'_\sigma (R', J') \subset E'_\sigma (R') \cap G_\sigma (J')$. For converse, let $$ x = \prod_{i=1}^{n} x_{[\alpha_i]} (t_i) \in E'_\sigma(R') \cap G_\sigma (J')$$ where $t_i \in R'_{[\alpha_i]}$. 
For each $t_i \in R'_{[\alpha_i]}$, choose elements $s_i \in R'_{[\alpha_i]}$ and $u_i \in J'_{[\alpha_i]}$ as follows:
\begin{enumerate}
    \item If $[\alpha_i] \sim A_1$, $A_1^2$, or $A_1^3$ and $t_i = (\alpha_i, \beta_i)$, then set $s_i = (\beta_i, \beta_i)$ and $u_i = (\alpha_i - \beta_i, 0)$. It is clear that $t_i = s_i + u_i$.

    \item If $[\alpha_i] \sim A_2,$ $t_i = (\alpha_i, \beta_i) \in \mathcal{A}(R')$, $\alpha_i = (a_1, a_2) \in R'$ and $\beta_i = (b_1, b_2) \in R'$, then set $s_i = (\gamma_i, \delta_i) \in \mathcal{A}(R')$ and $u_i = (\mu_i, \nu_i)  \in \mathcal{A}(R')$, where $\gamma_i = (a_2,a_2) \in R', \delta_i = (b_2,b_2) \in R', \mu_i = (a_1 - a_2, 0) \in R'$ and $\nu_i = (b_1 - b_2 - \overline{a_2} (a_1 - a_2), 0) \in R'$. Clearly, $t_i = s_i \oplus u_i$.
\end{enumerate}
Set $$y_k = \prod_{i=1}^{k} x_{[\alpha_i]} (s_i) \in E'_\sigma(R')$$ for $0 \leq k \leq n.$ Clearly, $y_0 = 1$ (by the definition). 
We claim that $y_n = 1$. Since $x \in G_\sigma(J')$, we have $x \equiv 1$ (mod $J'$). But then $y_n \equiv 1$ (mod $J'$), that is, $\prod_{i=1}^n x_{[\alpha_i]}(s_i + J') = 1$ in $E'_\sigma (R'/J')$ (the notion of $s_i + J'$ is clear even when $[\alpha] \sim A_2$). Note that there is a natural embedding from $R'/J'$ to $R/J \times R$ which induces an embedding from the group $E'_\sigma (R'/J')$ to the group $E'_{\sigma} (R/J \times R) \cong E'_{\sigma}(R/J) \times E'_{\sigma}(R)$. 
Under this embedding, $$\prod_{i=1}^n (x_{[\alpha_i]}(\beta_i+J),x_{[\alpha_i]}(\beta_i)) = (\prod_{i=1}^n x_{[\alpha_i]}(\beta_i+J), \prod_{i=1}^n x_{[\alpha_i]}(\beta_i)) = (1,1)$$ in $E'_{\sigma}(R/J) \times E'_{\sigma}(R)$. In particular, $\prod_{i=1}^n x_{[\alpha_i]}(\beta_i) = 1$ in $E'_{\sigma}(R)$. Thus
\begin{align*}
    y_n = \prod_{i=1}^n x_{[\alpha_i]}(s_i) = \prod_{i=1}^n x_{[\alpha_i]}(\beta_i, \beta_i) = (\prod_{i=1}^n x_{[\alpha_i]}(\beta_i), \prod_{i=1}^n x_{[\alpha_i]}(\beta_i)) = (1,1) = 1.
\end{align*}
This proves our claim. Finally,
$$x = \prod_{i=1}^{n} x_{[\alpha_i]} (s_i) x_{[\alpha_i]} (u_i) = \prod_{i=1}^{n} y_{i-1}^{-1}y_i x_{[\alpha_i]} (u_i) = y_0^{-1} \Bigg( \prod_{i=1}^{n} y_i x_{[\alpha_i]} (u_i) y_{i}^{-1} \Bigg) y_n \in E'_{\sigma} (R', J'),$$ as desired.

Again by Corollary \ref{cor:KS2}, $E'_\sigma (R')$ is a normal subgroup of $G_\sigma (R')$. So $h'' g' (h'')^{-1} \in E'_\sigma (R'),$ where $g' = (g,1) \in E'_\sigma (R')$. On the other hand, $h'' g' (h'')^{-1} = (hgh^{-1}, 1) \in G_\sigma (J')$. Therefore $h''g'(h'')^{-1} \in E'_\sigma (R') \cap G_\sigma (J') = E'_\sigma (R', J'),$ hence $hgh^{-1} \in E'_\sigma (R, J)$. Thus $E'_\sigma (R, J)$ is normal in $G_\sigma(R)$. 
\end{proof}

We derive the following corollary from the proof of the above Proposition.

\begin{cor}\label{mixcom}
    $[E'_\sigma (R), G_\sigma (J)] \subset E'_\sigma (R,J)$.
\end{cor}

\begin{proof} 
    Take any $h \in E'_{\sigma} (R)$ and $g \in G_\sigma (J)$. Define, as in proof of Proposition~\ref{normal}, $h'' = (h,h) \in E'_{\sigma} (R')$ and $g' = (g,1) \in G_{\sigma} (J')$. Then $[h'',g'] \in E'_\sigma (R') \cap G_\sigma (J') = E'_\sigma (R', J')$ (as $E'_\sigma (R')$ and $G_\sigma (J')$ are normal subgroups of $G_\sigma (R')$). Since $[h'',g'] = ([h,g],1)$, we have $[h,g] \in E'_\sigma(R,J).$ Thus, $[E'_\sigma (R), G_\sigma (J)] \subset E'_\sigma (R,J)$.
\end{proof}

%%%%%%%%%%%%%%%%%%%%%%%%%%%%%%%%%%%%%%%%%%%%%%%%%%

\begin{prop}\label{genofE(R,I)}
    For any ideal $J$ of $R$, the subgroup $E'_\sigma (R,J)$ is generated by elements of the form $x_{[\alpha]}(r)x_{-[\alpha]}(u)x_{[\alpha]}(r)^{-1}$ with $[\alpha] \in \Phi_\rho, r \in R_{[\alpha]}$ and $u \in J_{[\alpha]}$.
\end{prop}

%%%%%%%%%%%%%%%%%%%%%%%%%%%%%%%%%%%%%%%%%%%%%%%%%%

\begin{proof}
    Let $H$ be the subgroup of $E'_\sigma (R,J)$ generated by all $x_{[\alpha]}(r) x_{-[\alpha]}(u) x_{[\alpha]}(r)^{-1}$, where $[\alpha] \in \Phi_\rho, r \in R_{[\alpha]},$ and $u \in J_{[\alpha]}$. We aim to prove that $H = E'_\sigma (R,J)$. Since $E'_\sigma (J) \subset H$, it suffices to show that $H$ is a normal subgroup of $E'_\sigma (R)$. To demonstrate this, we need to verify that 
    $$ g = x_{[\beta]}(s) x_{[\alpha]}(r) x_{-[\alpha]}(u) x_{[\alpha]}(r)^{-1} x_{[\beta]}(s)^{-1} \in H $$ for all $[\alpha], [\beta] \in \Phi_\rho, r \in R_{[\alpha]}, s \in R_{[\beta]}$, and $u \in J_{[\alpha]}$.

\vspace{2mm}

\noindent \textbf{Case A. $[\alpha] \neq \pm [\beta]$.} For $[\gamma], [\delta] (\neq -[\gamma]) \in \Phi_\rho$, we have $$[x_{[\gamma]}(a), x_{[\delta]}(b)] = \prod x_{i [\gamma] + j [\delta]} (f_{i,j}(a,b)),$$ where $f_{i,j}$ is function of $a$ and $b$ with the property that $f_{i,j}(a,b) \in J_{i[\gamma]+j[\delta]}$ if $a \in J_{[\gamma]}$ or $b \in J_{[\delta]}$. Since no convex combination of the roots $-[\alpha], [\beta]$ and $i [\alpha] + j [\beta] \ (i, j \neq 0)$ is $0$, we have 
\begin{align*}
    g &= x_{[\beta]}(s) x_{[\alpha]}(r) x_{-[\alpha]}(u) x_{[\alpha]}(r)^{-1} x_{[\beta]}(s)^{-1} \\
    &= x_{[\alpha]}(r) x_{-[\alpha]}(u) x_{[\alpha]}(r)^{-1} \Big( \prod x_{i[\alpha]+ j[\beta]}(h_{i,j}(s,t,u)) \Big) \in H,
\end{align*}
where $h_{i,j}$ is function of $s, t$ and $u$ such that $h_{i,j}(s,t,u) \in J_{i[\alpha]+j[\beta]}$.

\vspace{2mm}

\noindent \textbf{Case B. $[\alpha] = \pm [\beta]$.} Note that, if $[\alpha] = [\beta]$ then there is nothing to prove. Now assume that $[\alpha] = -[\beta]$. We then have 
$$ g = x_{-[\alpha]}(s) x_{[\alpha]}(r) x_{-[\alpha]}(u) x_{[\alpha]}(r)^{-1} x_{-[\alpha]}(s)^{-1}.$$ 

Since the rank of $\Phi_\rho > 1$, there exists $[\gamma] \in \Phi_\rho$ such that the subroot system $\Phi'$ generated by $[\alpha]$ and $[\gamma]$ is connected of rank $2$. WLOG, we can assume that $[\alpha], [\gamma]$ is base of $\Phi'$. Set $\Phi'_+$ be the set of positive roots of $\Phi'$ with respect to this base, $\Phi'_- = - \Phi'_+, \Phi''_+ = \{ i [\alpha] + j [\gamma] \in \Phi'_+ \mid j>0 \},$ and $\Phi''_- = - \Phi''_+$. Write $U''_+ (J)$ (resp., $U''_- (J)$) for the subgroup of $E'_\sigma (R)$ generated by $x_{[\delta]}(t)$ with $[\delta] \in \Phi''_+$ (resp., $[\delta] \in \Phi''_-$) and $t \in J_{[\delta]}$. Then $U''_+(J)$ and $U''_-(J)$ are subgroups of $H$. 

Now, by Lemma~\ref{lemma:u_1,u_2 exists} (below), for given $u \in J_{[\alpha]}$ we can find $u_1 \in J_{[\alpha] + [\gamma]}$ and $u_2 \in R_{[\gamma]}$ such that $$x_{-[\alpha]}(u) = [x_{-([\alpha] + [\gamma])}(u_1), x_{[\gamma]}(u_2)] h'$$ with $h' \in U''_{-}(J)$. Set \begin{align*}
    g_1 &:= x_{-[\alpha]}(s) x_{[\alpha]}(r) x_{-([\alpha]+[\gamma])}(u_1) x_{[\alpha]}(r)^{-1} x_{-[\alpha]}(s)^{-1} \in U''_-(J), \\
    g_2 &:= x_{-[\alpha]}(s) x_{[\alpha]}(r) x_{[\gamma]}(u_2) x_{[\alpha]}(r)^{-1} x_{-[\alpha]}(s)^{-1} \in U''_+(R), \\
    g_3 &:= x_{-[\alpha]}(s) x_{[\alpha]}(r) h' x_{[\alpha]}(r)^{-1} x_{-[\alpha]}(s)^{-1} \in U''_-(J).
\end{align*}
Then $g = [g_1,g_2] g_3$, which contained in $H$ by Lemma \ref{lemma:H} (below). 
\end{proof}

\begin{lemma}\label{lemma:u_1,u_2 exists}
    For given $u \in J_{[\alpha]}$ we can find $u_1 \in J_{[\alpha] + [\gamma]}$ and $u_2 \in R_{[\gamma]}$ such that $$x_{-[\alpha]}(u) = [x_{-([\alpha] + [\gamma])}(u_1), x_{[\gamma]}(u_2)] h'$$ with $h' \in U''_-(J)$.
\end{lemma}

\begin{proof}
    The Chevalley commutator formula for $[x_{-([\alpha]+[\gamma])}(u_1), x_{[\gamma]}(u_2)]$ suggests that depending on the types of the pair of roots $(-[\alpha] - [\gamma], [\gamma])$ we can choose $u_1$ and $u_2$ as follows:
    \begin{center}
        \begin{tabular}{ccc}
            Type of pair $(-[\alpha] - [\gamma], [\gamma])$ & $u_1$ & $u_2$  \\
            \hline
            $(b-i)$ & $u$ & $\pm1$ \\
            $(b-ii)$ & $u$ or $\bar{u}$ & $\pm 1$ \\
            $(c-i)$ & $u$ & $\pm 1/2$ \\
            $(c-ii)$ & $(u, u \bar{u}/2)$ or $(\bar{u}, u \bar{u}/2)$ & $(\pm 1, 1/2)$ \\
            $(d-i)$ & $u$ & $\pm 1$ \\
            $(d-ii)$ & $(a, b):=u$ or $(\bar{a},b)$ or $(a, \bar{b})$ & $\pm 1$ \\
            $(e)$ & $u$ & $\pm 1$ \\
            $(g)$ & $u$ & $\pm 1/3$
        \end{tabular}
    \end{center}
    Note that each $u_1 \in J_{[\alpha]+[\gamma]}$ and $u_2 \in R_{[\gamma]}$. An immediate observation from Chevalley commutator formula for $[x_{-([\alpha]+[\gamma])}(u_1), x_{[\gamma]}(u_2)]$ is that $h' \in U''_-(J)$.
\end{proof}

\begin{lemma}\label{lemma:H}
    $[U''_-(J),U''_+(R)] \subset H$.
\end{lemma}

\begin{proof}
    Let $h \in U''_- (J)$ and $g \in U''_+ (R)$. Write $$ h = x_{-[\alpha_1]}(u_1) \dots x_{-[\alpha_n]}(u_n),$$ where $[\alpha_i] \in \Phi''_+$ and $0 \neq u_i \in J_{[\alpha_i]}$. We want to show that $[h,g] \in H.$ For that we use induction of $n$. If $n=1$ then $[h,g] = x_{-[\alpha_1]}(u_1) g x_{-[\alpha_1]}(u_1)^{-1} g^{-1}.$ Write $g = x_{[\beta_1]}(v_1) \dots x_{[\beta_m]}(v_m)$ where $[\beta_i] \in \Phi''_+$ and $0 \neq v_i \in R_{[\beta_i]}$. If $[\beta_i] \neq -[\alpha_1]$ for every $i = 1, \dots, m$, then we are done by Chevalley commutator relations. If there is some $i \in \{1, \dots, m\}$ such that $[\beta_i] = [\alpha_1]$ then also we are done by the definition of $H$ and Chevalley commutator relations. 

    Now for general $n$,
    \begin{align*}
        [h,g] &= h g h^{-1} g^{-1} \\
        &= x_{-[\alpha_1]}(u_1) [x_{-[\alpha_2]}(u_2) \dots x_{-[\alpha_n]}(u_n), g] [g, x_{-[\alpha_1]}(u_1)^{-1}] x_{-[\alpha_1]}(u_1)^{-1} \\
        &\in H.
    \end{align*}
    Which proves the lemma.
\end{proof}

%%%%%%%%%%%%%%%%%%%%%%%%%%%%%%%%%%%%%%%%%%%%%%%%%%

% \begin{thm}\label{thm:normalized}
%     For any ideal $J$ of $R$, we have 
%     $$E'_\sigma(R, J) = [E'_\sigma (R), E'_\sigma (J)] = [E'_\sigma (R), G_\sigma (R,J)] = [G_\sigma (R), E'_\sigma (R,J)] .$$
% \end{thm}

%%%%%%%%%%%%%%%%%%%%%%%%%%%%%%%%%%%%%%%%%%%%%%%%%%

% \begin{proof} 

\medskip

\noindent \emph{Proof of Theorem~\ref{mainthm1}:}
Note that $[E'_\sigma(R), E'_\sigma (J)] \subset [E'_\sigma(R), G_\sigma (R, J)]$ and $[E'_\sigma(R), E'_\sigma (J)] \subset [G_\sigma(R), E'_\sigma (R,J)]$. 
Also, by Proposition~\ref{normal}, we have $[G_\sigma (R), E'_\sigma (R,J)] \subset E'_\sigma (R,J)$. 
Therefore, to prove our proposition, it is enough to prove the following:
\begin{enumerate}[(i)]
    \item $E'_\sigma(R, J) \subset [E'_\sigma(R), E'_\sigma (J)].$
    \item $[E'_\sigma(R), G_\sigma (R, J)] \subset E'_\sigma(R, J).$
\end{enumerate}

Since $H:= [E'_\sigma(R), E'_\sigma (J)]$ is normal in $E'_\sigma(R)$, to prove (i) it is enough to prove that $x_{[\alpha]}(u) \in H$ for every $[\alpha] \in \Phi_\rho$ and $u \in J_{[\alpha]}$. As in the proof of Proposition \ref{genofE(R,I)}, since the rank of $\Phi_\rho > 1$, there exists $[\beta] \in \Phi_\rho$ such that the subsystem $\Phi'$ generated by $[\alpha]$ and $[\beta]$ is connected of rank $2$. WLOG, we can assume that $[\alpha], [\beta]$ is base of $\Phi'$.

\vspace{2mm}

\noindent \textbf{Case A. $\Phi' \sim A_2$.} In this case, the pair of roots $[\alpha] + [\beta]$ and $-[\beta]$ are of type $(b)$. 

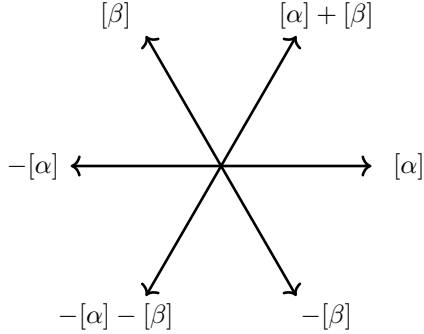
\begin{center}
    \begin{tikzpicture}
        \draw[<->, line width=1pt] (2,0)--(-2,0);
        \draw[<->, line width=1pt] (1,1.73)--(-1,-1.73);
        \draw[<->, line width=1pt] (-1,1.73)--(1,-1.73);
        \node at (2.5,0) {$[\alpha]$};
        \node at (-1.4,2.0) {$[\beta]$};
        \node at (1.4,2.0) {$[\alpha] + [\beta]$};
        \node at (-2.5,0) {$-[\alpha]$};
        \node at (1.4, -2.0) {$-[\beta]$};
        \node at (-1.4, -2.0) {$-[\alpha] - [\beta]$};
    \end{tikzpicture}
\end{center}

\begin{enumerate}
    \item[$(b-i)$] If $[\alpha] + [\beta] \sim A_1$ and $-[\beta] \sim A_1$, then so is $[\alpha] = ([\alpha] + [\beta]) + (-[\beta])$. In this case, for given $u \in J_{[\alpha]} = J_{\theta}$ we have $$ x_{[\alpha]}(u) = [x_{[\alpha]+[\beta]}(\pm 1), x_{-[\beta]}(u)] \in H.$$
    \item[$(b-ii)$] If $[\alpha] + [\beta] \sim A_1^2$ and $-[\beta] \sim A_1^2$, then so is $[\alpha] = ([\alpha] + [\beta]) + (-[\beta])$. In this case, for given $u \in J_{[\alpha]} = J$ we write $u'= u$ or $\bar{u} \in J$. Then we have $$ x_{[\alpha]}(u) = [x_{[\alpha]+[\beta]}(\pm 1), x_{-[\beta]}(u')] \in H.$$
\end{enumerate}

\vspace{2mm}

\noindent \textbf{Case B. $\Phi' \sim B_2$ and $[\alpha]$ is a long root.} In this case, the pair of roots $[\alpha] + [\beta]$ and $-[\beta]$ are of type $(c)$. 

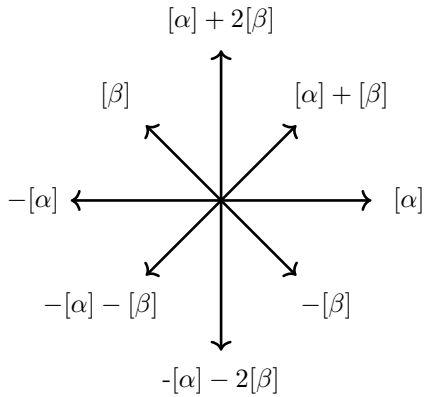
\begin{center}
    \begin{tikzpicture}
        \draw[<->, line width=1pt] (1,1)--(-1,-1);
        \draw[<->, line width=1pt] (-1,1)--(1,-1);
        \draw[<->, line width=1pt] (2,0)--(-2,0);
        \draw[<->, line width=1pt] (0,2)--(0,-2);
        \node at (2.5,0) {$[\alpha]$};
        \node at (-1.4,1.4) {$[\beta]$};
        \node at (1.6,1.4) {$[\alpha] + [\beta]$};
        \node at (0,2.4) {$[\alpha] + 2 [\beta]$};
        \node at (-2.5,0) {$-[\alpha]$};
        \node at (1.4,-1.4) {$-[\beta]$};
        \node at (-1.6,-1.4) {$-[\alpha] - [\beta]$};
        \node at (0,-2.4) {-$[\alpha] - 2 [\beta]$};
    \end{tikzpicture}
\end{center}

\begin{enumerate}
    \item[$(c-i)$] If $[\alpha] + [\beta] \sim A_1^2$ and $-[\beta] \sim A_1^2$, then $[\alpha] = ([\alpha] + [\beta]) + (-[\beta]) \sim A_1$. In this case, for given $u \in J_{[\alpha]} = J_{\theta}$ we have $$ x_{[\alpha]}(u) = [x_{[\alpha]+[\beta]}(\pm 1/2), x_{-[\beta]}(u)] \in H.$$
    \item[$(c-ii)$] If $[\alpha] + [\beta] \sim A_2$ and $-[\beta] \sim A_2$, then $[\alpha] = ([\alpha] + [\beta]) + (-[\beta]) \sim A_1^2$. In this case, for given $u \in J_{[\alpha]} = J$ we write $u'= u$ or $\bar{u} \in J$. Then we have $$ x_{[\alpha]}(u) = [x_{[\alpha]+[\beta]}(\pm 1, 1/2), x_{-[\beta]}(u', u \bar{u} / 2)] \in H.$$
\end{enumerate}

\vspace{2mm}

\noindent \textbf{Case C. $\Phi' \sim B_2$ and $[\alpha]$ is a short root.} In this case, the pair of roots $[\alpha] + [\beta]$ and $-[\beta]$ are of type $(d)$ with $[\alpha] + [\beta]$ being the short root.

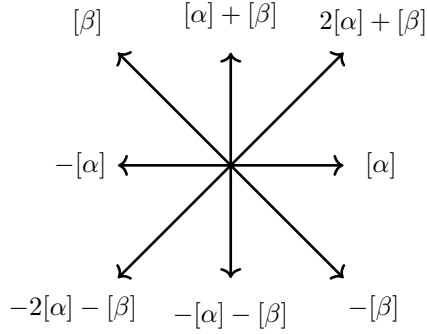
\begin{center}
    \begin{tikzpicture}
        \draw[<->, line width=1pt] (1.5,1.5)--(-1.5,-1.5);
        \draw[<->, line width=1pt] (-1.5,1.5)--(1.5,-1.5);
        \draw[<->, line width=1pt] (1.5,0)--(-1.5,0);
        \draw[<->, line width=1pt] (0,1.5)--(0,-1.5);
        \node at (2,0) {$[\alpha]$};
        \node at (-1.9,1.9) {$[\beta]$};
        \node at (1.9,1.9) {$2[\alpha] + [\beta]$};
        \node at (0,2) {$[\alpha] + [\beta]$};
        \node at (-2,0) {$-[\alpha]$};
        \node at (1.9,-1.9) {$-[\beta]$};
        \node at (-2.1,-1.9) {$-2[\alpha] - [\beta]$};
        \node at (0,-2) {$-[\alpha] - [\beta]$};
    \end{tikzpicture}
\end{center}

\begin{enumerate}
    \item[$(d-i)$] If $[\alpha] + [\beta] \sim A_1^2$ and $-[\beta] \sim A_1$, then $[\alpha] = ([\alpha] + [\beta]) + (-[\beta]) \sim A_1^2$ and $2[\alpha] + [\beta] = 2([\alpha] + [\beta]) + (-[\beta]) \sim A_1$. In this case, for given $u \in J_{[\alpha]} = J$ we have 
    \begin{align*}
        x_{[\alpha]}(u) x_{2[\alpha] + [\beta]}(u) &= [x_{-[\beta]}(\pm u), x_{[\alpha]+[\beta]}(\pm 1)] \\
        &= [x_{[\alpha]+[\beta]}(\pm 1), x_{-[\beta]}(\pm u)]^{-1} \in H.
    \end{align*}
    Now observe that $[\alpha] \sim A_1^2, [\alpha] + [\beta] \sim A_1^2$ and $2[\alpha] + [\beta] \sim A_1$. Then by similar argument as in $(c-i)$ above, we can conclude that $x_{2[\alpha]+[\beta]}(u) \in H$. Hence $$ x_{[\alpha]}(u) = (x_{[\alpha]}(u) x_{2[\alpha] + [\beta]}(u)) (x_{2[\alpha] + [\beta]}(u))^{-1} \in H.$$
    \item[$(d-ii)$] If $[\alpha] + [\beta] \sim A_2$ and $-[\beta] \sim A_1^2$, then $[\alpha] = ([\alpha] + [\beta]) + (-[\beta]) \sim A_2$ and $2[\alpha] + [\beta] = 2([\alpha] + [\beta]) + (-[\beta]) \sim A_1^2$. In this case, for given $u = (u_1, u_2) \in J_{[\alpha]} = \mathcal{A}(J) = \mathcal{J}$ we have 
    \begin{align*}
        x_{[\alpha]}(u_1,u_2) x_{2[\alpha] + [\beta]}(\pm u_2) &= [x_{-[\beta]}(\pm 1), x_{[\alpha]+[\beta]}(u_1, u_2)] \\
        &= [x_{[\alpha]+[\beta]}(u_1, u_2), x_{-[\beta]}(\pm 1)]^{-1} \in H.
    \end{align*}
    Now observe that $[\alpha] \sim A_2, [\alpha] + [\beta] \sim A_2$ and $2[\alpha] + [\beta] \sim A_1^2$. Then by similar argument as in $(c-ii)$ above, we can conclude that $x_{2[\alpha]+[\beta]}(\pm u_2) \in H$. Hence $$ x_{[\alpha]}(u_1,u_2) = (x_{[\alpha]}(u_1, u_2) x_{2[\alpha] + [\beta]}(\pm u_2)) (x_{2[\alpha] + [\beta]}(\pm u_2))^{-1} \in H.$$
\end{enumerate}

\vspace{2mm}

\noindent \textbf{Case D. $\Phi' \sim G_2$ and $[\alpha]$ is a long root.} In this case, we consider a subroot system $\Phi''$ of $\Phi'$ generated by roots $[\alpha]$ and $[\alpha] + 3 [\beta]$. Note that $\Phi'' \sim A_2$ and hence, by case 1 (replace $[\beta]$ by $[\alpha] + 3 [\beta]$), we can conclude that $x_{[\alpha]}(u) \in H$.

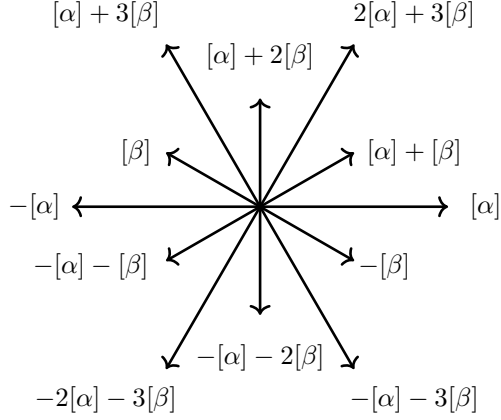
\begin{center}
    \begin{tikzpicture}
        \draw[<->, line width=1pt] (2.5,0)--(-2.5,0);
        \draw[<->, line width=1pt] (1.25,0.72)--(-1.25,-0.72);
        \draw[<->, line width=1pt] (1.25,2.16)--(-1.25,-2.16);
        \draw[<->, line width=1pt] (0,1.44)--(0,-1.44);
        \draw[<->, line width=1pt] (-1.25,2.16)--(1.25,-2.16);
        \draw[<->, line width=1pt] (-1.25,0.72)--(1.25,-0.72);
        \node at (3,0) {$[\alpha]$};
        \node at (-1.65,0.72) {$[\beta]$};
        \node at (2.05,0.72) {$[\alpha] + [\beta]$};
        \node at (2.05,2.56) {$2[\alpha] + 3 [\beta]$};
        \node at (0,2) {$[\alpha] + 2 [\beta]$};
        \node at (-2.05,2.56) {$[\alpha] + 3 [\beta]$};
        \node at (-3,0) {$-[\alpha]$};
        \node at (1.65,-0.8) {$-[\beta]$};
        \node at (-2.25,-0.8) {$-[\alpha] - [\beta]$};
        \node at (-2.05,-2.56) {$-2[\alpha] - 3 [\beta]$};
        \node at (0,-2) {$-[\alpha] - 2 [\beta]$};
        \node at (2.05,-2.56) {$-[\alpha] - 3 [\beta]$};
    \end{tikzpicture}
\end{center}

\vspace{2mm}

\noindent \textbf{Case E. $\Phi' \sim G_2$ and $[\alpha]$ is a short root.} In this case, the pair of roots $2[\alpha] + [\beta]$ and $-[\alpha] - [\beta]$ are of type $(f)$.

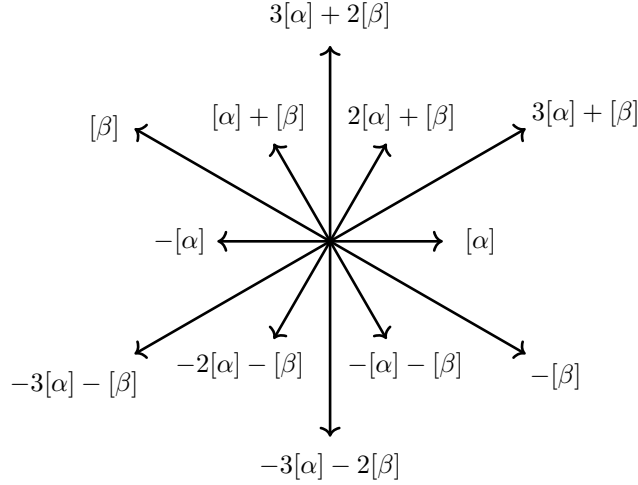
\begin{center}
    \begin{tikzpicture}
        \draw[<->, line width=1pt] (1.5,0)--(-1.5,0);
        \draw[<->, line width=1pt] (2.6,1.5)--(-2.6,-1.5);
        \draw[<->, line width=1pt] (0.75,1.3)--(-0.75,-1.3);
        \draw[<->, line width=1pt] (0,2.6)--(0,-2.6);
        \draw[<->, line width=1pt] (-0.75,1.3)--(0.75,-1.3);
        \draw[<->, line width=1pt] (-2.6,1.5)--(2.6,-1.5);
        \node at (2,0) {$[\alpha]$};
        \node at (-3,1.5) {$[\beta]$};
        \node at (-0.95,1.65) {$[\alpha] + [\beta]$};
        \node at (0.95,1.65) {$2[\alpha] + [\beta]$};
        \node at (3.4,1.7) {$3[\alpha] + [\beta]$};
        \node at (0,3) {$3[\alpha] + 2 [\beta]$};
        \node at (-2,0) {$-[\alpha]$};
        \node at (3,-1.8) {$-[\beta]$};
        \node at (1,-1.65) {$-[\alpha] - [\beta]$};
        \node at (-1.2,-1.65) {$-2[\alpha] - [\beta]$};
        \node at (-3.4,-1.9) {$-3[\alpha] - [\beta]$};
        \node at (0,-3) {$-3[\alpha] - 2 [\beta]$};
    \end{tikzpicture}
\end{center}

Observe that $[\alpha], 2[\alpha] + [\beta], - [\alpha] - [\beta] \sim A_1^3$ and $-[\beta], 3[\alpha] + [\beta] \sim A_1$. For given $u \in J_{[\alpha]} = J$ we write $(u',u'') = (\bar{u}, \bar{\bar{u}})$ or $(\bar{\bar{u}}, \bar{u})$. Then we have 
\begin{align*}
    & \hspace{-5mm} x_{[\alpha]}(u) x_{3[\alpha] + [\beta]}(\pm (u^2 + (u')^2 + (u'')^2 - 2 u u' - 2 u' u'' - 2 uu'')/4) x_{-[\beta]}(\pm (u + u' + u'')/2) \\
    &= [x_{2[\alpha]+[\beta]}((u + u' - u'')/2), x_{-[\alpha]-[\beta]}(\pm 1)] \in H.
\end{align*}

Now observe that $- [\beta] \sim A_1$ and $3[\alpha] + [\beta] \sim A_1$. Then by similar argument as in Case A above, we can conclude that $x_{3[\alpha]+[\beta]}(\pm (u^2 + (u')^2 + (u'')^2 - 2 u u' - 2 u' u'' - 2 uu'')/4) \in H$ and $x_{-[\beta]}(\pm (u + u' + u'')^2/4) \in H$. Hence 
\begin{align*}
    x_{[\alpha]}(u) &= (x_{[\alpha]}(u) x_{3[\alpha] + [\beta]}(\pm (u^2 + (u')^2 + (u'')^2 - 2 u u' - 2 u' u'' - 2 uu'')/4) \\ 
    & \hspace{7mm} x_{-[\beta]}(\pm (u + u' + u'')/2)) (x_{3[\alpha] + [\beta]}(\pm (u^2 + (u')^2 + (u'')^2 - 2 u u' - 2 u' u'' - 2 uu'')/4) \\
    & \hspace{14mm} x_{-[\beta]}(\pm (u + u' + u'')/2))^{-1} \in H.
\end{align*}

This proves part (i). Now for part (ii), we consider the groups $M:= E'_\sigma (R)$ and $N:= (E'_\sigma (R) \cap G_\sigma(J))/E'_{\sigma}(R, J)$. Observe that the group $M$ is perfect (put $J=R$ in part (i)) and the group $N$ is commutative (by Corollary~\ref{mixcom}). For a fixed $g \in G_\sigma(R,J)$, define a map $\psi_g: M \longrightarrow N$ given by $h \longmapsto [h,g]E'_\sigma(R,J)$. Then $\psi_g$ is a well-defined homomorphism from the perfect group $M$ to a commutative group $N$. Hence $\psi_g$ must be trivial, i.e., $[h,g] \in E'_\sigma (R,J)$ for all $h \in E'_\sigma(R).$ Thus, $[E'_\sigma(R), G_\sigma (R, J)] \subset E'_\sigma(R, J),$ as desired. \qed
% \end{proof}

%%%%%%%%%%%%%%%%%%%%%%%%%%%%%%%%%%%%%%%%%%%%%%%%%%

\begin{cor}\label{cor:normalized}
    The group $E'_\sigma (R)$ is perfect, that is, $[E'_\sigma (R),E'_\sigma (R)]=E'_\sigma (R)$.
\end{cor}

\begin{proof}
    Immediate by putting $J=R$ in the above proposition.
\end{proof}

%%%%%%%%%%%%%%%%%%%%%%%%%%%%%%%%%%%%%%%%%%%%%%%%%%

\begin{cor}\label{cor:converge}
    Every subgroup of $G_\sigma (R, J)$ containing $E'_\sigma (R, J)$ is normalized by $E'_\sigma (R)$.
\end{cor}

\begin{proof}
    Let $H$ be a subgroup of $G_\sigma (R, J)$ containing $E'_\sigma (R, J)$. Then $$[E'_\sigma(R), E'_\sigma(J)] \subset [E'_\sigma(R), H] \subset [E'_\sigma(R), G_\sigma (R,J)].$$ By Theorem~\ref{mainthm1}, we have $[E'_\sigma (R), H] = E'_\sigma (R,J) \subset H$. Therefore $H$ is normalized by $E'_\sigma (R)$.
\end{proof}

%%%%%%%%%%%%%%%%%%%%%%%%%%%%%%%%%%%%%%%%%%%%%%%%%%

\begin{cor}\label{C=G}
    Let $C_\sigma (R,J) = \{ x \in G_\sigma (R) \mid [x, E'_\sigma (R)] \subset E'_\sigma (R,J) \}$. Then $G_\sigma (R,J) = C_\sigma (R,J)$.
\end{cor}

\begin{proof}
    Clearly, by Theorem~\ref{mainthm1}, $G_\sigma (R,J) \subset C_\sigma (R,J)$. By definition, we have $G_\sigma (R,J) = \{ x \in G_\sigma (R) \mid [x, G_\sigma (R)] \subset G_\sigma (J) \}$. Let $Z(G)$ be the centre of $G = G_\sigma (R/J)$ and $C_G(E)$ the centralizer of $E = E'_\sigma (R/J)$ in $G$. Then, by Theorem \ref{thm:KS3}, we have $Z(G) = C_G(E)$. But then $$G_\sigma (R,J) = \{x \in G_\sigma (R) \mid [x, E'_\sigma (R)] \subset G_\sigma (J) \}.$$ 
    Since $E'_\sigma (R,J) \subset G_\sigma (J)$, we have $C_\sigma (R,J) \subset G_\sigma (R,J)$.
\end{proof}

%%%%%%%%%%%%%%%%%%%%%%%%%%%%%%%%%%%%%%%%%%%%%%%%%%
%Section: Proof of Theorem~\ref{mainthm}
%%%%%%%%%%%%%%%%%%%%%%%%%%%%%%%%%%%%%%%%%%%%%%%%%%

\section{Proof of Theorem~\ref{mainthm}}\label{sec:Pf of main thm}

Let $H$ be a subgroup of $G_\sigma (R)$ normalized by $E'_\sigma (R)$. For $[\alpha] \in \Phi_\rho$, we write 
\[
    J_{[\alpha]} (H) = \begin{cases}
        \{ t \in R \mid x_{[\alpha]} (t) \in H \} & \text{if } [\alpha] \sim A_1, A_1^2, A_1^3; \\
        \{ t \in R \mid \exists \ u \in R \text{ with } x_{[\alpha]} (t, u) \in H \text{ or } x_{[\alpha]} (u,t) \in H \} & \text{if } [\alpha] \sim A_2. 
    \end{cases}
\]
Define $J = \displaystyle\bigcup_{[\alpha] \in \Phi_\rho} J_{[\alpha]}(H)$. To demonstrate the main theorem, we will first consider the following two propositions.
%%%%%%%%%%%%%%%%%%%%%%%%%%%%%%%%%%%%%%%%%%%%%%%%%%

\begin{prop}\label{prop:E(R,J) subset of H}
    \normalfont
    Let $J$ be as above. Then
    \begin{enumerate}[(a)]
        \item $J$ is a $\theta$-invariant ideal of $R$.
        \item $E'_\sigma (R, J) \subset H$. 
    \end{enumerate}
\end{prop}

%%%%%%%%%%%%%%%%%%%%%%%%%%%%%%%%%%%%%%%%%%%%%%%%%%

\begin{prop}\label{prop:U(R) cap H subset U(J)}
    \normalfont
    Let $J$ be as above. Then
    \begin{enumerate}[(a)]
        \item $U_\sigma (R) \cap H \subset U_\sigma (J)$. 
        \item $U_\sigma (rad (R)) T_\sigma (R) U^{-}_\sigma (R) \cap H \subset U_\sigma (J) T_\sigma (R, J) U^{-}_\sigma (J)$, where $rad(R)$ is the Jacobson radical of $R$.
    \end{enumerate}
\end{prop}

%%%%%%%%%%%%%%%%%%%%%%%%%%%%%%%%%%%%%%%%%%%%%%%%%%

The proofs of Propositions \ref{prop:E(R,J) subset of H} and \ref{prop:U(R) cap H subset U(J)} can be found in Sections \ref{sec:Pf of prop 1} and \ref{sec:Pf of prop 2}, respectively.
Moving forward, let $\mathfrak{m}$ be a maximal ideal of $R$ and define $\bar{\mathfrak{m}} = \theta(\mathfrak{m})$. Set
\[
    S_{\mathfrak{m}} = \begin{cases} 
        R \setminus \mathfrak{m} & \text{if } \mathfrak{m} = \bar{\mathfrak{m}}, \\
        R \setminus (\mathfrak{m} \cup \bar{\mathfrak{m}}) & \text{if } \mathfrak{m} \neq \bar{\mathfrak{m}} \text{ and } o(\theta) = 2, \\
        R \setminus (\mathfrak{m} \cup \bar{\mathfrak{m}} \cup \bar{\bar{\mathfrak{m}}}) & \text{if } \mathfrak{m} \neq \bar{\mathfrak{m}} \text{ and } o(\theta) = 3.
    \end{cases}
\]
Then $S_\mathfrak{m}$ is a multiplicatively closed subset of $R$ such that $\theta (S_\mathfrak{m}) \subset S_\mathfrak{m}$. Therefore there is a natural automorphism of the ring $S_\mathfrak{m}^{-1}R$ induced by an automorphism $\theta$ of $R$. We denote this automorphism of $S_\mathfrak{m}^{-1}R$ also by $\theta$. Let $\psi_\mathfrak{m}: G_\sigma (R) \longrightarrow G_\sigma (S_\mathfrak{m}^{-1} R)$ be the homomorphism of groups induced by the canonical homomorphism of rings from $R$ to $S_\mathfrak{m}^{-1}R$. We write $R_S$ and $J_S$ for $S_\mathfrak{m}^{-1} R$ and $S_\mathfrak{m}^{-1} J$, respectively. By $S_\theta$, we mean $S_\mathfrak{m} \cap R_\theta$.

%%%%%%%%%%%%%%%%%%%%%%%%%%%%%%%%%%%%%%%%%%%%%%%%%%

\begin{prop}\label{local:H subset G(R,J)}
    $\psi_\mathfrak{m} (H) \subset G_\sigma (R_S, J_S)$. 
\end{prop}

%%%%%%%%%%%%%%%%%%%%%%%%%%%%%%%%%%%%%%%%%%%%%%%%%%

The proof of Proposition~\ref{local:H subset G(R,J)} has been provided in Section \ref{sec:Pf of prop 3}.

%%%%%%%%%%%%%%%%%%%%%%%%%%%%%%%%%%%%%%%%%%%%%%%%%%

\begin{prop}\label{local:[x,g] in E(R,J)}
    For any element $g \in G_\sigma (R_S, J_S),$ there exists an elements $s \in S_\theta$ such that $$[\psi_\mathfrak{m}(x_{[\alpha]}(s \cdot t)), g] \in \psi_\mathfrak{m} (E'_\sigma(R, J))$$ for all $[\alpha] \in \Phi_\rho$ and $t \in  R_{[\alpha]}.$
\end{prop}

\begin{proof}
    Note that $R_S$ is a semi-local ring, and $J_S$ is either contained in $rad(R_S)$ or equal to $R_S$. 
    If $J_S \subset \operatorname{rad}(R_S)$, then by Proposition~\ref{lavidecomposition} and Corollary~\ref{G(R,J)=UTV}, we have $G_\sigma(R_S, J_S) = E'_\sigma(R_S, J_S)\, T_\sigma(R_S, J_S)$.
    If $J_S = R_S$, then by Proposition~\ref{G=G'}, it follows that $G_\sigma(R_S, J_S) = E'_\sigma(R_S, J_S)\, T_\sigma(R_S, J_S)$.
    In both cases, by Proposition \ref{genofE(R,I)}, we conclude that $G_\sigma(R_S, J_S)$ is generated by all $h(\chi) \in T_\sigma (R_S,J_S)$ and all the elements of the form 
    \[
    z_{[\beta]}(u,v) := x_{[\beta]}(u)\, x_{-[\beta]}(v)\, x_{[\beta]}(u)^{-1},
    \]
    where $[\beta] \in \Phi_\rho$, $u \in (R_S)_{[\beta]}$, and $v \in (J_S)_{[\beta]}$. We first prove the lemma for each of these generators.
    
    Let $g = h(\chi) \in T_\sigma (R_S, J_S)$ where $\chi \in \text{Hom}_1(\Lambda_\pi, R_S^{*})$. By Lemma \ref{lemma on T(J)}, we have $\chi(\alpha) \equiv 1$ (mod $J_S$) for every root $\alpha \in \Phi$, that is, $1 - \chi(\alpha) \in J_S$ for every root $\alpha \in \Phi$. Hence, there exists $u_\alpha \in J$ and $s_{\alpha} \in S$ such that $1 - \chi(\alpha) = \displaystyle\frac{u_{\alpha}}{s_{\alpha}}$. Take $s = \displaystyle\prod_{\alpha \in \Phi} s_{\alpha} \bar{s}_\alpha$ or $\displaystyle\prod_{\alpha \in \Phi} s_{\alpha} \bar{s}_\alpha \bar{\bar{s}}_\alpha$ depending on whether $o(\theta) = 2$ or $3$. Clearly, $s \in S_\theta$. 
    If $[\alpha] \sim A_2$ and $t = (t_1, t_2) \in R_{[\alpha]}$, then $t/1$ denotes $(t_1/1, t_2/1)$. Now 
    $$ [\psi_\mathfrak{m}(x_{[\alpha]}(s \cdot t)), g] = [x_{[\alpha]}(s/1 \cdot t/1), h(\chi)] = x_{[\alpha]} (s/1 \cdot t/1) x_{[\alpha]}(\chi(\alpha) s/1 \cdot t/1)^{-1}. $$
    If $[\alpha] \sim A_1, A_1^2$ or $A_1^3$ then 
    $$[\psi_\mathfrak{m}(x_{[\alpha]}(s \cdot t)), g] = x_{[\alpha]}\Big(\frac{st}{1}(1 - \chi(\alpha))\Big) = x_{[\alpha]}\Big(\frac{s t u_\alpha}{s_\alpha}\Big) = \psi_\mathfrak{m}(x_{[\alpha]}(a_{\alpha, t})),$$
    where $a_{\alpha, t} = \Bigg( \displaystyle\prod_{\substack{\beta \in \Phi \\ \beta \neq \alpha}} s_\beta \bar{s}_\beta \Bigg) \bar{s}_\alpha t u_\alpha $ or $\Bigg( \displaystyle\prod_{\substack{\beta \in \Phi \\ \beta \neq \alpha}} s_\beta \bar{s}_\beta \bar{\bar{s}}_\beta \Bigg) \bar{s}_\alpha \bar{\bar{s}}_\alpha t u_\alpha$ depending on whether $o(\theta) = 2$ or $3$. 
    Since $a_{\alpha, t} \in J_{[\alpha]}$, we obtain the required result in this case. Now if $[\alpha] \sim A_2$ then
    \begin{align*}
        [\psi_\mathfrak{m}(x_{[\alpha]}(s \cdot t)), g] &= x_{[\alpha]}\Bigg(\frac{st_1}{1}, \frac{s^2 t_2}{1}\Bigg) x_{[\alpha]} \Bigg(\chi(\alpha) \frac{st_1}{1}, \chi(\alpha) \overline{\chi(\alpha)} \frac{s^2 t_2}{1}\Bigg)^{-1} \\
        &= x_{[\alpha]}\Bigg((1 - \chi(\alpha))\frac{st_1}{1}, (1 - \chi(\alpha)) \frac{s^2t_2}{1} - \chi(\alpha) (1-\overline{\chi(\alpha)}) \frac{s^2 \bar{t}_2}{1} \Bigg) \\
        &= x_{[\alpha]}\Bigg(\frac{s t_1 u_\alpha}{s_\alpha}, \frac{s^2 t_2 u_\alpha}{s_\alpha} - \frac{s^2 \bar{t}_2 (s_\alpha - u_\alpha) \bar{u}_\alpha}{s_\alpha \bar{s}_\alpha} \Bigg) \\
        &= \psi_\mathfrak{m}(x_{[\alpha]}(a_{\alpha, t}, b_{\alpha, t})),
    \end{align*}
    where $a_{\alpha, t} = \Bigg( \displaystyle\prod_{\substack{\beta \in \Phi \\ \beta \neq \alpha}} s_\beta \bar{s}_\beta \Bigg) \bar{s}_\alpha t_1 u_\alpha$ and $b_{\alpha, t} = \Bigg( \displaystyle\prod_{\substack{\beta \in \Phi \\ \beta \neq \alpha}} s_\beta \bar{s}_\beta \Bigg) \bar{s}_\alpha s t_2 u_\alpha - \Bigg( \displaystyle\prod_{\substack{\beta \in \Phi \\ \beta \neq \alpha}} s_\beta \bar{s}_\beta \Bigg) s \bar{t}_2 (s_\alpha - u_\alpha) \bar{u}_\alpha$. Since $(a_{\alpha, t}, b_{\alpha, t}) \in \mathcal{A}(J)$, we obtain the required result in this case.

    \vspace{2mm}
    
    Now let $g = z_{[\beta]}(u, v)$, where $[\beta] \in \Phi_\rho, u \in (R_S)_{[\beta]}$ and $v \in (J_S)_{[\beta]}$. We want to find $s \in S_\theta$ such that 
    $$ [x_{[\alpha]}(s/1 \cdot t/1), z_{[\beta]}(u, v)] \in \psi_\mathfrak{m}(E'_\sigma(R, J)),$$ for all $[\alpha] \in \Phi_\rho$ and $t \in R_{[\alpha]}$.
    If $[\beta] \sim A_1, A_1^2$ or $A_1^3$, then write $u = a/b$ and $v = c/d$, where $a \in R, c \in J$ and $b, d \in S_\mathfrak{m}$. If $[\beta] \sim A_2$, then write $u = (u_1, u_2) = (a_1/b_1, a_2/b_2)$ and $v = (v_1, v_2) = (c_1/d_1, c_2/d_2)$, where $a_1, a_2 \in R,$ $c_1, c_2 \in J$ and $b_1, b_2, d_1, d_2 \in S_\mathfrak{m}$.
    Depending on the type of root $[\beta]$, we can choose the value of $s$ as below:
    \[
        s = \begin{cases}
            (b d)^m & \text{if } [\beta] \sim A_1, \\
            (b \bar{b} d \bar{d})^m & \text{if } [\beta] \sim A_1^2, \\
            (b \bar{b} \bar{\bar{b}} d \bar{d} \bar{\bar{d}})^m & \text{if } [\beta] \sim A_1^3, \\
            (b_1 \bar{b}_1 b_2 \bar{b}_2 d_1 \bar{d}_1 d_2 \bar{d}_2)^m & \text{if } [\beta] \sim A_2, 
        \end{cases}
    \]
    for sufficiently large positive integer $m$. We proceed using a similar method as in Proposition~\ref{genofE(R,I)}. 

    \vspace{2mm}

    \noindent \textbf{Case A. $[\alpha] \neq \pm [\beta]$.} Since no convex combination of the roots $[\alpha], -[\beta]$ and $i [\alpha] + j [\beta] \ (i, j \neq 0)$ is $0$, by using Chevalley commutator formula, we obtain 
    \begin{align*}
        & \hspace{-5mm} [x_{[\alpha]}(s/1 \cdot t/1), z_{[\beta]}(u,v)] \\
        &= x_{[\alpha]}(s/1 \cdot t/1) x_{[\beta]}(u) x_{-[\beta]}(v) x_{[\beta]}(u)^{-1} x_{[\alpha]}(s/1 \cdot t/1)^{-1} x_{[\beta]}(u) x_{-[\beta]}(v) x_{[\beta]}(u)^{-1} \\
        &= \Big( \prod x_{i[\alpha]+ j[\beta]}(c_{i,j}(s, t, u, v)) \Big),
    \end{align*}
    where $c_{i,j} = c_{i,j} (s, t, u, v)$ are functions of $s, t, u$ and $v$ such that $s$ and $v$ (or $\bar{v}, v_1, \bar{v}_1, v_2, \bar{v}_2$; the last four candidates appear only when $v=(v_1, v_2) \in \mathcal{A}(R)$) appear in each term of $c_{i,j}$. But then, for sufficiently large $m$ (see the definition of $s$), the values of $c_{i,j}(s,t,u,v)$ are in $J_{i[\alpha]+j[\beta]}$, as desired.

    \vspace{2mm}

    \noindent \textbf{Case B. $[\alpha] = \pm [\beta]$.} Since the rank of $\Phi_\rho > 1$, there exists $[\gamma] \in \Phi_\rho$ such that $[\beta]$ and $[\gamma]$ is base of connected subsystem $\Phi'$ of $\Phi_\rho$ is of rank $2$. 
    By Lemma \ref{lemma:u_1,u_2 exists}, there exists $a \in (J_S)_{[\beta] + [\gamma]}$ and $b \in (R_S)_{[\gamma]}$ such that
    $$ x_{-[\beta]}(v) = [x_{-[\beta]-[\gamma]}(a), x_{[\gamma]}(b)] h'$$ with $h' = \prod_{i \leq 0, j < 0} x_{i[\beta] + j[\gamma]} (a_{ij}) $. 
    But then 
    \begin{align*}
        z_{[\beta]}(u,v) &= x_{[\beta]}(u) x_{-[\beta]}(v) x_{[\beta]}(u)^{-1} \\ 
        &= x_{[\beta]}(u) \{ [x_{-[\beta]-[\gamma]}(a), x_{[\gamma]}(b)] h' \} x_{[\beta]}(u)^{-1} \\
        &= [x_{[\beta]}(u), x_{-[\beta]-[\gamma]}(a)] x_{-[\beta]-[\gamma]}(a) [x_{[\beta]}(u), x_{[\gamma]}(b)] x_{[\gamma]}(b) [x_{[\beta]}(u), x_{-[\beta]-[\gamma]}(a)^{-1}] \\
        & \hspace{10mm}  x_{-[\beta]-[\gamma]}(a)^{-1} [x_{[\beta]}(u), x_{[\gamma]}(b)^{-1}] x_{[\gamma]}(b)^{-1} [x_{[\beta]}(u), h'] h'.
    \end{align*}
    Now we consider
    \begin{align*}
        [x_{[\alpha]}(s/1 \cdot t/1), z_{[\beta]}(u,v)] &= x_{\pm [\beta]}(s/1 \cdot t/1) z_{[\beta]}(u,v) x_{\pm [\beta]}(s/1 \cdot t/1)^{-1} z_{[\beta]}(u,v)^{-1}.
    \end{align*}
    Using the above expression of $z_{[\beta]}(u,v)$ and Chevalley commutator formulas, we can conclude that $[x_{[\alpha]}(s/1 \cdot t/1), z_{[\beta]}(u,v)]$ can be expressed as a product of elements of the form $x_{[\delta]}(c_{[\delta]})$, where $[\delta] \in \Phi'$ and $c_{[\delta]} \in (R_S)_{[\delta]}$ such that $s$ and $v$ (or $\bar{v}, v_1, \bar{v}_1, v_2, \bar{v}_2$) appear in each term of $c_{[\delta]}$. Consequently, for sufficiently large $m$, the values of $c_{[\delta]}$ are in $J_{[\delta]}$, as desired.

    \vspace{2mm}

    Finally, let $g$ be an arbitrary element of $G_\sigma(R_S, J_S)$. We aim to show that for any given $[\alpha] \in \Phi_\rho$ and $t \in R_{[\alpha]}$, there exists $s \in S_\theta$ such that
    \[
        [\psi_{\mathfrak{m}}(x_{[\alpha]}(s \cdot t)), g] \in \psi_{\mathfrak{m}}(E'_\sigma(R, J)).
    \]
    As noted earlier, $g$ can be written as a product
    \[
        g = x_1 \cdots x_n,
    \]
    where each factor $x_i$ is either of the form $h(\chi)$ for some $\chi \in \mathrm{Hom}_1(\Lambda_\pi, (R_S)^*)$, or $z_{[\beta]}(u, v)$ with $[\beta] \in \Phi_\rho$, $u \in (R_S)_{[\beta]}$ and $v \in (J_S)_{[\beta]}$. 
    We have already established that for each such $x_i$, there exists $s_i \in S_\theta$ such that
    \[
        [\psi_{\mathfrak{m}}(x_{[\alpha]}(s_i \cdot t)), x_i] \in \psi_{\mathfrak{m}}(E'_\sigma(R, J)).
    \]
    Let $s = s_1 \cdots s_n$ and define $w = \psi_{\mathfrak{m}}(x_{[\alpha]}(s \cdot t))$. A straightforward computation yields
    \[
        [w, g] = [w, x_1 \cdots x_n] = \{ [w, x_1] \} \{ {}^{x_1}[w, x_2] \} \{ {}^{x_1 x_2}[w, x_3] \} \cdots \{ {}^{x_1 \cdots x_{n-1}}[w, x_n] \},
    \]
    where ${}^a b$ denotes the conjugation $a b a^{-1}$.
    Using a similar (though slightly adapted) argument previously applied to each generator $x_i$, we conclude that
    \[
        \{ {}^{x_1 \cdots x_{i-1}}[w, x_i] \} \in \psi_{\mathfrak{m}}(E'_\sigma(R, J)) \quad \text{for each } i.
    \]
    Therefore, the entire commutator $[w, g]$ lies in $\psi_{\mathfrak{m}}(E'_\sigma(R, J))$, as required.
\end{proof}

%%%%%%%%%%%%%%%%%%%%%%%%%%%%%%%%%%%%%%%%%%%%%%%%%%

Before proceeding further, we state a lemma from G. Taddei \cite{GT}. Let $G_\pi(\Phi, R)$ be a Chevalley group over a commutative ring $R$. Consider $R[X]$, the polynomial ring in one variable $X$ with coefficients in $R$. For a maximal ideal $\mathfrak{m}$, let $\psi'_\mathfrak{m}: G_\pi (\Phi, R) \longrightarrow G_\pi (\Phi, R_S)$ denote the natural map induced by the canonical ring homomorphism $R \longrightarrow R_S$.

\begin{lemma}[{\cite[Lemma 3.14]{GT}}]\label{lemma:GT}
    Let $\epsilon (X)$ be an element of $G_\pi(\Phi, R[X])$. Suppose $\psi'_\mathfrak{m} (\epsilon (X)) = 1$ and $\epsilon  (0) = 1$. Then there exists an element $s$ of $S$ such that $\epsilon (sX) = 1$.
\end{lemma}

\begin{rmk}
    \normalfont
    Consider a twisted Chevalley group $G_\sigma (R)$. Note that an automorphism $\theta: R \longrightarrow R$ of order $n$ (where $n=2$ or $3$) can be naturally extended to an automorphism of $R[X]$ of the same order, denoted also by $\theta$. Therefore, we can make sense of the group $G_\sigma (R[X])$.
    Since $\psi'_\mathfrak{m} \mid_{G_\sigma (R)} = \psi_\mathfrak{m}$, we can apply the conclusion of Lemma~\ref{lemma:GT} to the case of twisted Chevalley groups as well. Moreover, in this context, we can also choose $s \in S_\theta$ (see \cite[Lemma $4.7$]{KS2}).
\end{rmk}

%%%%%%%%%%%%%%%%%%%%%%%%%%%%%%%%%%%%%%%%%%%%%%%%%%

\begin{prop}\label{general:[x,g] in E(R,J)}
    For any maximal ideal $\mathfrak{m}$ of $R$ and $g \in H$, there exists $s \in S_\theta$ such that $$[x_{[\alpha]}(s \cdot t), g] \in E'_\sigma (R, J),$$ for all $[\alpha] \in \Phi_\rho$ and $t \in R_{[\alpha]}$.
\end{prop}

\begin{proof}
    The canonical map $R \longrightarrow R_S$ naturally induces the following maps: 
    $$\psi_\mathfrak{m}: G_\sigma (R) \longrightarrow G_\sigma (R_S) \text{ and } \psi''_\mathfrak{m}: G_\sigma (R[X]) \longrightarrow G_\sigma (R_S[X]).$$
    It is clear that $\psi''_\mathfrak{m} \mid_{G_\sigma(R)} = \psi_\mathfrak{m}$.
    Let $g \in H$. Then, by Proposition~\ref{local:H subset G(R,J)}, $\psi_\mathfrak{m} (g) \in G_\sigma (R_S,J_S) \subset G_\sigma (R_S[X], J_S[X])$. By Proposition \ref{local:[x,g] in E(R,J)}, for the case of $\psi''_\mathfrak{m}$, there exists $s' \in S_\theta$ such that
    $$\psi''_\mathfrak{m} ([x_{[\alpha]}(s' X \cdot t), g]) \in \psi''_\mathfrak{m} (E'_\sigma (R[X], J[X])),$$
    for all $[\alpha] \in \Phi_\rho$ and $t \in R_{[\alpha]}$. Moreover, from the proof of Proposition \ref{local:[x,g] in E(R,J)}, we can preciously write 
    $$\psi''_\mathfrak{m} ([x_{[\alpha]}(s' X \cdot t), g]) = \psi''_\mathfrak{m} \Big( \prod_{i=1}^{m} x_{[\alpha_i]} (c_i (X)) \Big),$$
    where $c_i (X) \in (J[X])_{[\alpha_i]}$ such that $c_i (0) = 0$.
    Put $$ \epsilon_{[\alpha]} (X) = [x_{[\alpha]}(s' X \cdot t), g] \Big( \prod_{i=1}^{m} x_{[\alpha_i]} (c_i(X)) \Big)^{-1}.$$
    Then we see that $\epsilon_{[\alpha]} (X)$ satisfies the hypotheses of Lemma~\ref{lemma:GT}, and hence there exists $s'_{[\alpha]} \in S_\theta$ such that $\epsilon_{[\alpha]}(s'_{[\alpha]}X) = 1$. 
    Thus we obtain $$ [x_{[\alpha]}(s' (s'_{[\alpha]}X) \cdot t), g] = \prod_{i=1}^{m} x_{[\alpha_i]} (c_i(s'_{[\alpha]}X)).$$
    Now by taking $X = 1$ and $s_{[\alpha]} = s' s'_{[\alpha]}$, we derive
    $$ [x_{[\alpha]}(s_{[\alpha]} \cdot t), g] \in E'_\sigma (R, J).$$
    Finally, if we set $s = \prod_{[\alpha] \in \Phi_\rho} s_{[\alpha]} \in S_\theta$ then $s$ is the required element.
\end{proof}

%%%%%%%%%%%%%%%%%%%%%%%%%%%%%%%%%%%%%%%%%%%%%%%%%%

Now we are in a position to prove the main theorem. 

%%%%%%%%%%%%%%%%%%%%%%%%%%%%%%%%%%%%%%%%%%%%%%%%%%

\vspace{2mm}

\noindent \textit{Proof of Theorem \ref{mainthm}.} Let $H$ be a subgroup of $G_\sigma (R)$ normalized by $E'_\sigma (R)$ and let $J$ be as earlier. Since $E'_\sigma (R, J) \subset H$ (see Proposition \ref{prop:E(R,J) subset of H}), it only remains to prove that $H \subset G_\sigma (R,J)$. To demonstrate this, it suffices to show, by Corollary \ref{C=G}, that if $g \in H$ then $[x_{[\alpha]}(t), g] \in E'_\sigma (R,J)$ for every $[\alpha] \in \Phi_\rho$ and $t \in R_{[\alpha]}$.

\vspace{2mm}

\noindent \textbf{Case A. $[\alpha] \sim A_1$:} Define $I_{g, [\alpha]} = \{ s \in R_\theta \mid [x_{[\alpha]} (st), g] \in E'_\sigma (R,J) \text{ for all } t \in R_{\theta} \}$. Then $I_{g,[\alpha]}$ is an ideal of $R_\theta$. To see this, let $s_1, s_2 \in I_{g,[\alpha]}$. Then for all $t \in R_\theta$,
$$ [x_{[\alpha]}((s_1 + s_2)t), g] = (x_{[\alpha]}(s_1 t) [x_{[\alpha]}(s_2 t), g] x_{[\alpha]}(s_1 t)^{-1}) [x_{[\alpha]}(s_1 t), g] \in E'_\sigma(R,J).$$
Therefore, $s_1 + s_2 \in I_{g, [\alpha]}$. By the definition of $I_{g, [\alpha]}$, it is clear that for all $r \in R_\theta$, we have $r s_1 \in I_{g, [\alpha]}$. 
For any maximal ideal $\mathfrak{m}_\theta$ of $R_\theta$, by Proposition \ref{general:[x,g] in E(R,J)} and by Lemma \ref{A1&A2}, there exists $s \in I_{g, [\alpha]}$ such that $s \not \in \mathfrak{m}_\theta$. Thus, we can conclude that $I_{g, [\alpha]} = R_\theta$. But then $1 \in I_{g, [\alpha]}$ and hence $[x_{[\alpha]}(t), g] \in E'_\sigma (R,J)$ for all $t \in R_{\theta}$, as desired.

\vspace{2mm}

\noindent \textbf{Case B. $[\alpha] \sim A_1^2$ or $A_1^3$:} Define $I_{g, [\alpha]} = \{ s \in R \mid [x_{[\alpha]} (st), g] \in E'_\sigma (R,J) \text{ for all } t \in R \}$. Then, by a similar argument as in Case A, we can see that $I_{g,[\alpha]}$ is an ideal of $R$. For any maximal ideal $\mathfrak{m}$ of $R$, by Proposition~\ref{general:[x,g] in E(R,J)}, there exists $s \in I_{g, [\alpha]}$ such that $s \not \in \mathfrak{m}$. Thus, we can conclude that $I_{g, [\alpha]} = R$. But then $[x_{[\alpha]}(t), g] \in E'_\sigma (R,J)$ for all $t \in R$, as desired.

\vspace{2mm}

\noindent \textbf{Case C. $[\alpha] \sim A_2$:} Define $A_{g, [\alpha]} = \{ s \in R \mid [x_{[\alpha]} (s t, s \bar{s} t \bar{t}/2), g] \in E'_\sigma (R,J) \text{ for all } t \in R \}$. Let $I_{g, [\alpha]}^{(1)}$ be the ideal of $R$ generated by $A_{g, [\alpha]}$ and $I_{g, [\alpha]}^{(2)} = \{ s \in R_\theta \mid [x_{[\alpha]}(0, st), g] \in E'_\sigma (R, J) \text{ for all } t \in R \text{ such that } t = - \bar{t} \}$. 
We claim that $I_{g, [\alpha]}^{(2)}$ is an ideal of $R_\theta$. To see this, let $s_1, s_2 \in I^{(2)}_{g,[\alpha]}$. Then for all $t \in R$ such that $t = - \bar{t}$, we have 
\begin{align*}
    [x_{[\alpha]}(0, (s_1 + s_2)t), g] &= [x_{[\alpha]}(0, s_1 t) x_{[\alpha]}(0, s_2 t), g] \\
    &= (x_{[\alpha]}(0, s_1 t) [x_{[\alpha]}(0, s_2 t), g] x_{[\alpha]}(0, s_1 t)^{-1}) ([x_{[\alpha]}(0, s_1 t) ,g]) \\
    & \in E'_\sigma (R,J).
\end{align*}
Therefore, $s_1 + s_2 \in I^{(2)}_{g, [\alpha]}$. By the definition of $I^{(2)}_{g, [\alpha]}$, it is clear that for all $r \in R_\theta$, we have $r s \in I^{(2)}_{g, [\alpha]}$. 

For any maximal ideal $\mathfrak{m}$ of $R$, Proposition~\ref{general:[x,g] in E(R,J)} ensures the existence of an element $s \in A_{g, [\alpha]}$ such that $s \notin \mathfrak{m}$. Therefore, we conclude that $I_{g, [\alpha]}^{(1)} = R$.
Next, observe the following:
\begin{quote}
    \textit{For any maximal ideal $\mathfrak{m}$ of $R$ and any $g \in H$, there exists $s \in S_\theta$ such that
    \[
        [x_{[\alpha]}(0, st), g] \in E'_\sigma(R, J)
    \]
    for all $[\alpha] \in \Phi_\rho$ and all $t \in R$ satisfying $t = -\bar{t}$.}
\end{quote}
The proof of the above statement follows the same argument as that of Proposition~\ref{general:[x,g] in E(R,J)} and is therefore omitted. 
Hence, we conclude that $I^{(2)}_{g, [\alpha]} = R_\theta$. 

Let $(t, u) \in \mathcal{A}(R)$. Since $I^{(1)}_{g, [\alpha]} = R$, there exists $s_1, \dots, s_k \in A_{g, [\alpha]}$ and $r_1, \dots, r_k \in R$ such that $r_1 s_1 + \cdots + r_k s_k = 1$. Therefore, $t = r_1 s_1 t + \cdots + r_k s_k t$. Observe that
\begin{align*}
    (t, u) &= (r_1 s_1 t + \cdots + r_k s_k t, u) \\
    &= \left( r_1 s_1 t, \frac{r_1 \bar{r}_1 s_1 \bar{s}_1 t \bar{t}}{2} \right) \oplus \cdots \oplus \left( r_k s_k t, \frac{r_k \bar{r}_k s_k \bar{s}_k t \bar{t}}{2} \right) \oplus (0, c),
\end{align*}
for some $c \in R$ with $c = -\bar{c}$. Set 
\[
    x_i = x_{[\alpha]} \left( r_i s_i t,\; \frac{r_i \bar{r}_i s_i \bar{s}_i t \bar{t}}{2} \right) \quad \text{for each } i = 1, \dots, k, \quad \text{and} \quad x_{k+1} = x_{[\alpha]}(0, c).
\]
Then $x_{[\alpha]}(t, u) = x_1 \dots x_{k+1}$.
Since $s_i \in A_{g, [\alpha]}$, it follows that $x_i \in E'_\sigma (R, J)$ for all $i = 1, \dots, k$. Moreover, as $1 \in I^{(2)}_{g, [\alpha]}$, we have $x_{k+1} \in E'_\sigma (R, J)$.
Therefore,
\begin{align*}
    [x_{[\alpha]}(t,u), g] &= [x_1 \dots x_{k+1}, g] \\
    &= \{ {}^{(x_1 \cdots x_k)}[x_{k+1}, g] \} \{ {}^{(x_1 \cdots x_{k-1})} [x_k, g] \} \cdots \{ {}^{x_1} [x_2, g] \} \{ [x_1, g] \} \in E'_\sigma (R,J),
\end{align*}
as desired. \qed

%%%%%%%%%%%%%%%%%%%%%%%%%%%%%%%%%%%%%%%%%%%%%%%%%%

\begin{cor}
    Let $R$ and $\Phi$ be as described in Theorem \ref{mainthm}. Then a subgroup $H$ of $E'_{\sigma}(\Phi, R)$ is normal if and only if there exists a unique $\theta$-invariant ideal $J$ of $R$ such that 
    $$E'_{\sigma} (R, J) \subset H \subset G_{\sigma} (R, J) \cap E'_{\sigma}(\Phi, R).$$
\end{cor}

\begin{proof}
    The result follows directly from Theorem~\ref{mainthm} and Corollary~\ref{cor:converge}.
\end{proof}

%%%%%%%%%%%%%%%%%%%%%%%%%%%%%%%%%%%%%%%%%%%%%%%%%%
%Section: Proof of Proposition \ref{prop:E(R,J) subset of H}
%%%%%%%%%%%%%%%%%%%%%%%%%%%%%%%%%%%%%%%%%%%%%%%%%%

\section{Proof of Proposition \ref{prop:E(R,J) subset of H}}\label{sec:Pf of prop 1}

%%%%%%%%%%%%%%%%%%%%%%%%%%%%%%%%%%%%%%%%%%%%%%%%%%

\begin{lemma}\label{A(R)}
    \normalfont
    Let $R$ be a ring with unity, and let $J$ be a $\theta$-invariant ideal of $R$. Assume that $1/2 \in R$. Consider the group $G = (\mathcal{A}(J), \oplus)$. Define the subgroups $H = \{ (0,u) \mid u \in J, \bar{u} = - u \}$ and $K = \langle (r,r \bar{r}/2) \mid r \in J \rangle$ of $G$. Then $H$ is a normal subgroup of $G$; in fact, it is contained in the centre of $G$, and $G = HK$. Moreover, if $J = R$ then $H \subset K$, that is, $G = K$.
\end{lemma}

\begin{proof}
    The first assertion is clear as $(g_1,g_2) (0,u) (g_1,g_2)^{-1} = (0,u)$ for appropriate $g_1, g_2, u \in J$. For the second assertion, note that for given $(g_1, g_2) \in \mathcal{A}(J)$, we have
    $$(g_1, g_2) = (g_1, g_1 \bar{g_1}/2) (0, (g_2 - \bar{g_2})/2).$$ Hence $G = HK$. Now if $J=R$, then we want to show that $H \subset K$.
    Our work is done if we show that $(0,(g_2 - \bar{g_2})/2)$ is generated by elements of the form $(r, r \bar{r}/2), \ r \in R$, which follows from below: $$ (0,(g_2 - \bar{g_2})/2) = (1+g_2, (1+g_2) (1 + \bar{g_2})/2)^{-1} (1, 1/2) (g_2, g_2 \bar{g_2}/2).$$
\end{proof}

%%%%%%%%%%%%%%%%%%%%%%%%%%%%%%%%%%%%%%%%%%%%%%%%%%%%%%%%%%%%%%%

\begin{lemma}\label{theta Rz}
    \normalfont
    Assume that $1/2 \in R$ if $o(\theta) = 2$ and $1/3 \in R$ if $o(\theta) = 3$. Let $z \in R$. 
    \begin{enumerate}[(a)]
        \item If $z = \bar{z}$, then $(Rz)_\theta = R_\theta z$.
        \item If $o(\theta) = 2$ and $rz + r' \bar{z} \in (Rz + R\bar{z})_\theta$, then there exists $t \in R$ such that $rz + r' \bar{z} = tz + \bar{t} \bar{z}.$
        \item If $o(\theta) = 3$ and $rz + r' \bar{z} + r'' \bar{\bar{z}} \in (Rz + R{\bar{z}} + R{\bar{\bar{z}}})_\theta$, then there exists $t \in R$ such that $rz + r' \bar{z} + r'' \bar{\bar{z}} = tz + \overline{tz} + \overline{\overline{tz}}.$ 
    \end{enumerate}
\end{lemma}

\begin{proof}
    Note that $(a)$ follows from $(b)$ (or $(c)$) and our assumption on $R$. To prove $(b)$, observe that $rz + r' \bar{z} \in (Rz + R\bar{z})_\theta \implies \overline{(rz + r' \bar{z})} = rz + r' \bar{z} \implies (r - \bar{r'}) z = (\bar{r} - r') \bar{z}.$ Set $a:= r-\bar{r'},$ then $\bar{a} = \bar{r} - r'$ and $az = \overline{az}$. Now let $t = r - a/2 = r/2 + \bar{r'}/2$, then $\bar{t} = \bar{r} - \bar{a}/2 = r' + \bar{a}/2$. Then $tz + \overline{tz} = (r - a/2)z + (r' + \bar{a}/2) \bar{z} = rz + r' \bar{z} + (\overline{az} - az)/2 = rz + r' \bar{z},$ as desired. Similarly, to prove $(c)$, observe that $rz + r' \bar{z} + r'' \bar{\bar{z}} \in (Rz + R\bar{z} + R\bar{\bar{z}})_\theta \implies \overline{rz + r' \bar{z} + r'' \bar{\bar{z}}} = rz + r' \bar{z} + r'' \bar{\bar{z}} \implies \bar{r} \bar{z} + \bar{r'} \bar{\bar{z}} + \bar{r''} z = rz + r' \bar{z} + r'' \bar{\bar{z}} \implies (r - \bar{r''}) z + (r' - \bar{r}) \bar{z} + (r'' - \bar{r'}) \bar{\bar{z}} = 0.$ Set $a:= r - \bar{r''}, b:= r' - \bar{r}$ and $c:= r'' - \bar{r'}$, then $\bar{\bar{a}} + \bar{b} + c = 0$ and $az + b \bar{z} + c \bar{\bar{z}} = 0$. Now let $t = r - a/3 + \bar{\bar{b}}/3 = (r + \bar{r''} + \bar{\bar{r'}})/3$, then $\bar{t} = r' - \bar{a}/3 - 2b/3$ and $\bar{\bar{t}} = r'' + 2 \bar{\bar{a}}/3 + \bar{b}/3$. Then 
    \begin{align*}
        tz + \overline{tz} + \overline{\overline{tz}} &= (r - a/3 + \bar{\bar{b}}/3)z + (r' - \bar{a}/3 - 2b/3) \bar{z} + (r'' + 2 \bar{\bar{a}}/3 + \bar{b}/3) \bar{\bar{z}} \\
        &= (rz + \overline{rz} + \overline{\overline{rz}}) - 2(az + b \bar{z} + (-\bar{\bar{a}} - \bar{b})\bar{\bar{z}})/3 - (\bar{a}\bar{z} + \bar{b} \bar{\bar{z}} + (-a - \bar{\bar{b}})z)/3 \\
        &= (rz + \overline{rz} + \overline{\overline{rz}}) - 2(az + b \bar{z} + c \bar{\bar{z}})/3 - (\bar{a}\bar{z} + \bar{b} \bar{\bar{z}} + \bar{c} z)/3 \\
        &= (rz + \overline{rz} + \overline{\overline{rz}}),
    \end{align*}
    as desired.
\end{proof}

%%%%%%%%%%%%%%%%%%%%%%%%%%%%%%%%%%%%%%%%%%%%%%%%%%%%%%%%%%%%%%%

\begin{lemma}\label{A(I+J) = A(I) A(J)}
    \normalfont
    Let $I$ and $J$ be $\theta$-invariant ideals of $R$. Then $\mathcal{A}(I + J) = \mathcal{A}(I) \mathcal{A} (J) = \mathcal{A}(J) \mathcal{A} (I)$.
\end{lemma}

\begin{proof}
    Clearly, $\mathcal{A}(I) \mathcal{A} (J) \subset \mathcal{A}(I + J)$. For converge, let $(r_1, r_2) \in \mathcal{A}(I+J)$. Then $r_1 \bar{r}_1 = r_2 + \bar{r}_2$ and there exists $a_1, a_2 \in I$ and $b_1, b_2 \in J$ such that $r_1 = a_1 + b_1$ and $r_2 = a_2 + b_2$. Therefore,
    \begin{align*}
        & r_1 \bar{r}_1 = r_2 + \bar{r}_2 \\
        \implies & (a_1 + b_1) (\bar{a}_1 + \bar{b}_1) = (a_2 + b_2) + (\bar{a}_2 + \bar{b}_2) \\
        \implies & a_1 \bar{a}_1 + a_1 \bar{b}_1 + b_1 \bar{a}_1 + b_1 \bar{b}_1 = (a_2 + b_2) + (\bar{a}_2 + \bar{b}_2). 
    \end{align*}
    By using this, we can write 
    \begin{align*}
        (r_1, r_2) &= (a_1 + b_1, a_2 + b_2) \\
        &= \Big(a_1, \frac{a_1 \bar{a}_1 + (a_2 - \bar{a}_2)}{2}\Big) \oplus \Big(b_1, \frac{b_1 \bar{b}_1 + (b_2 - \bar{b}_2)}{2}\Big) \oplus \Big(0, \frac{a_1 \bar{b}_1 - \bar{a}_1 b_1}{2}\Big).
    \end{align*}
    Hence $(r_1, r_2) \in \mathcal{A}(I) \mathcal{A}(J)$. Therefore $\mathcal{A}(I+J) = \mathcal{A}(I) \mathcal{A}(J) = \mathcal{A}(J) \mathcal{A}(I)$, later equality is possible because $\mathcal{A}(I) \mathcal{A}(J)$ is a group.
\end{proof}

%%%%%%%%%%%%%%%%%%%%%%%%%%%%%%%%%%%%%%%%%%%%%%%%%%

\begin{prop}\label{z to Rz}
    \normalfont
    Fix a root $[\alpha] \in \Phi_\rho$ and an element $z \in R_{[\alpha]}$. Let $H$ be the normal subgroup of $E'_\sigma (R)$ generated by $x_{[\alpha]} (z)$. Then $H = E'_{\sigma}(R, J)$ where 
    \[
        J = \begin{cases}
            Rz & \text{if } [\alpha] \sim A_1, \\
            Rz + R{\bar{z}} & \text{if } [\alpha] \sim A_1^2, \\
            Rz + R{\bar{z}} + R{\bar{\bar{z}}} & \text{if } [\alpha] \sim A_1^3, \\
            Rz_1 + R{\bar{z_1}} + R (z_2 - \bar{z_2}) & \text{if } [\alpha] \sim A_2.
        \end{cases}
    \]
\end{prop}

\begin{proof}
    Since $x_{[\alpha]}(z) \in E'_\sigma(R, J)$, by Proposition \ref{normal}, we have $H \subset E'_\sigma(R, J)$. For the reverse inclusion, we need to prove that $x_{[\beta]}(t) \in H$ for every $t \in J_{[\beta]}$ and $[\beta] \in \Phi_\rho$. Observe that there exist a sequence of roots $[\alpha_i] \ (i = 1, \dots, n)$ such that $[\alpha_1] = [\alpha], [\alpha_n] = [\beta]$, and for every $i \in \{1, \dots, n-1 \}$, the pair of roots $[\alpha_i], [\alpha_{i+1}]$ contained in some connected subsystem of $\Phi_\rho$ of rank $2$. Now, by applying Lemma \ref{z to Rz in phi'} (below) recursively to the pairs $([\alpha_i], [\alpha_{i+1}])$, we obtain the desired result.
\end{proof}

%%%%%%%%%%%%%%%%%%%%%%%%%%%%%%%%%%%%%%%%%%%%%%%%%%%%%%%%%%%%%%%

\begin{lemma}\label{z to Rz in phi'}
    \normalfont
    Let the notation be as established in Proposition \ref{z to Rz}. Suppose $\Phi'$ is a connected subsystem of $\Phi_\rho$ with rank 2. If $[\gamma] \in \Phi'$ and $s \in R$ such that $x_{[\gamma]}(s) \in H$, then $x_{[\beta]}(t) \in H$ for every $[\beta] \in \Phi'$ and $t \in I_{[\beta]}$, where $I$ is an $\theta$-invariant ideal defined similarly to the ideal $J$ in Proposition \ref{z to Rz} by replacing $z$ with $s$.
\end{lemma}

\begin{proof}
    Let $\mu$ denote the angle between $[\beta]$ and $[\gamma]$. For $r \in R$, let $r'$ denote either $r, \bar{r}$ or $\bar{\bar{r}}$. By $r_1$ and $r_2$, we mean the first and second coordinate of $r = (r_1, r_2) \in \mathcal{A}(R)$.
    Consider the following table outlining the possible types of choices for $\Phi'$:
    \begin{center}
        \begin{tabular}{|c|c|c|c|c|}
            \hline
            \multirow{2}{*}{\textbf{Type}} & \multirow{2}{*}{\textbf{$\Phi_\rho$}} & \multicolumn{2}{c|}{\textbf{Type of Roots}} & \multirow{2}{*}{\textbf{Possible Choices of $\Phi'$}} \\
            \cline{3-4}
             & & \textbf{Long} & \textbf{Short} &  \\
            \hline 
            ${}^2 A_{3}$ & $C_2 (= B_2$) & $A_1$ & $A_1^2$ & $B_2$ (containing long and short roots) \\
            \hline 
            ${}^2 A_{2n-1} \ (n \geq 3)$ & $C_n$ & $A_1$ & $A_1^2$ & $A_2$ (only short), $B_2$ (long and short)\\
            \hline 
            ${}^2 A_{4}$ & $B_2$ & $A_1^2$ & $A_2$ & $B_2$ (containing long and short roots) \\
            \hline 
            ${}^2 A_{2n} \ (n \geq 3)$ & $B_n$ & $A_1^2$ & $A_2$ & $A_2$ (only long), $B_2$ (long and short)\\
            \hline 
            ${}^2 D_{n} \ (n \geq 4)$ & $B_{n-1}$ & $A_1$ & $A_1^2$ & $A_2$ (only long), $B_2$ (long and short) \\
            \hline
            ${}^3 D_{4}$ & $G_2$ & $A_1$ & $A_1^3$ & $A_2$ (only long), $G_2$ (long and short) \\
            \hline
            ${}^2 E_{6}$ & $F_4$ & $A_1$ & $A_1^2$ & $A_2$ (with long), $A_2$ (with short), $B_2$ \\
            \hline
        \end{tabular}
    \end{center}

    \vspace{2mm}

    \noindent \textbf{Case A. $\Phi' \sim A_2$ which contains roots of the type $A_1$.} This case arises only when $\Phi_\rho \sim {}^2 D_{n} \ (n \geq 4), {}^3 D_4$ or ${}^2 E_6$. Let us consider the following subcases:
    \begin{enumerate}
        \item[\textbf{(A1)}] \textbf{$[\gamma] \sim A_1$ and $\mu = \pi/3$.} In this case, the pair $[\gamma]$ and $[\beta] - [\gamma]$ is of type $(b-i)$. Here $[\gamma], [\beta]$ and $[\beta] - [\gamma]$ are of type $A_1$. For $r \in R_\theta$, we have
        $$ [x_{[\gamma]}(s), x_{[\beta] - [\gamma]}(\pm r)] = x_{[\beta]} (rs) \in H.$$
        
        \item[\textbf{(A2)}] \textbf{$[\gamma] \sim A_1$ and $\mu$ is arbitrary.} Observe that, we can find a sequence $[\gamma_1], \dots, [\gamma_m]$ of roots in $\Phi'$ such that $2 \leq m \leq 6, [\gamma_1] = [\gamma], [\gamma_m] = [\beta]$ and the angle between $[\gamma_i]$ and $[\gamma_{i+1}]$ is $\pi / 3$ for every $i = 1, \dots, m-1$. Then, by applying Case (A1) to the pair of roots $([\gamma_{i-1}], [\gamma_{i}])$, we have $x_{[\gamma_i]}(rs) \in H$ for every $r \in R_\theta$ and $i = 2, \dots, m$. In particular, by Lemma \ref{theta Rz}, $x_{[\beta]}(t) \in H$ for every $t \in I_{[\beta]}$, as desired.
    \end{enumerate}
    
    \vspace{2mm}

    \noindent \textbf{Case B. $\Phi' \sim A_2$ which contains roots of the type $A_1^2$.} This case arises only when $\Phi_\rho \sim {}^2 A_{2n-1} \ (n \geq 3) {}^2 A_{2n} \ (n \geq 3)$ or ${}^2 E_6$. Let us consider the following subcases:
    \begin{enumerate}
        \item[\textbf{(B1)}] \textbf{$[\gamma] \sim A_1^2$ and $\mu = \pi/3$.} In this case, the pair $[\gamma]$ and $[\beta] - [\gamma]$ is of the type $(b-ii)$. Here $[\gamma], [\beta]$ and $[\beta] - [\gamma]$ are of type $A_1^2$. For $r \in R$, we have
        $$ [x_{[\gamma]}(s), x_{[\beta] - [\gamma]}(\pm r')] = x_{[\beta]} (rs') \in H.$$

        \item[\textbf{(B2)}] \textbf{$[\gamma] \sim A_1^2$ and $\mu$ is arbitrary.} Observe that, we can find a sequence $[\gamma_1], \dots, [\gamma_m]$ of roots in $\Phi'$ such that $2 \leq m \leq 6, [\gamma_1] = [\gamma], [\gamma_m] = [\beta]$ and the angle between $[\gamma_i]$ and $[\gamma_{i+1}]$ is $\pi / 3$ for every $i = 1, \dots, m-1$. Then, by applying Case (B1) to the pair of roots $([\gamma_{i-1}],[\gamma_i])$, we have $x_{[\gamma_i]}(rs') \in H$ for every $r \in R$ and $i = 2, \dots, m$. In particular, $x_{[\beta]}(rs') \in H$ for every $r \in R$.
        
        Note that, we can find a subsystem $\Phi''$ of $\Phi_\rho$ of type $B_2$ such that $[\beta] \in \Phi''$. Since $x_{[\beta]}(rs') \in H$, by Case C (if $\Phi_\rho \sim {}^2 A_{2n-1}$ or ${}^2 E_6$) and by Case D (if $\Phi_\rho \sim {}^2 A_{2n}$) (see below), we have $x_{[\beta]}(\bar{r}\bar{s'}) \in H$ for every $r \in R$. In particular, $x_{[\beta]}(rs)$ and $x_{[\beta]}(r \bar{s}) \in H$ for every $r \in R$. Therefore, $x_{[\beta]}(t) \in H$ for every $t \in I_{[\beta]}$, as desired.
    \end{enumerate}

    \vspace{2mm}
    
    \noindent \textbf{Case C. $\Phi' \sim B_2$ which contains roots of the type $A_1$ and $A_1^2$.} This case arises only when $\Phi_\rho \sim {}^2 A_{2n-1} \ (n \geq 2), {}^2 D_{n} \ (n \geq 4)$ or ${}^2 E_6$. Let us consider the following subcases: 
    \begin{enumerate}
        \item[\textbf{(C1)}] \textbf{$[\gamma] \sim A_1$ and $\mu = \pi/4$.} In this case, the pair $[\gamma]$ and $[\beta] - [\gamma]$ is of the type $(d-i)$. Here $[\gamma], 2[\beta] - [\gamma] \sim A_1$ and $[\beta], [\beta] - [\gamma] \sim A_1^2$. For $r \in R$, we have
        $$ [x_{[\gamma]}(s), x_{[\beta] - [\gamma]}(\pm r/2)] = x_{[\beta]} (rs/2) x_{2[\beta] - [\gamma]} (\pm r \bar{r} s/4) \in H.$$
        Now put $-r$ instead of $r$, we get $x_{[\beta]} (-rs/2) x_{2[\beta] - [\gamma]} (\pm r \bar{r} s/4) \in H$. But then
        $$ x_{[\beta]}(rs) = \{x_{[\beta]} (rs/2) x_{2[\beta] - [\gamma]} (\pm r \bar{r} s/4) \} \{ x_{[\beta]} (-rs/2) x_{2[\beta] - [\gamma]} (\pm r \bar{r} s/4) \}^{-1} \in H.$$
        
        \item[\textbf{(C2)}] \textbf{$[\gamma] \sim A_1^2$ and $\mu = \pi/4$.} In this case, the pair $[\gamma]$ and $[\beta] - [\gamma]$ is of the type $(c-i)$. Here $[\gamma], [\beta] - [\gamma] \sim A_1^2$ and $[\beta] \sim A_1$. For $r \in R$, we have
        $$ [x_{[\gamma]}(s), x_{[\beta] - [\gamma]}(\pm r)] = x_{[\beta]} (r \bar{s} + \bar{r} s) \in H.$$

        \item[\textbf{(C3)}] \textbf{$[\gamma] \sim A_1^2$ and $\mu = \pi/2$.} In this case, the pair $[\gamma]$ and $[\beta] - [\gamma]$ is of the type $(d-i)$. Here $[\gamma], [\beta] \sim A_1^2$ and $[\beta] + [\gamma], [\beta] - [\gamma] \sim A_1$. For $r \in R_\theta$, we have
        $$[x_{[\beta] - [\gamma]} (\pm r), x_{[\gamma]}(s)] = x_{[\beta]} (rs) x_{[\gamma] + [\beta]} (\pm rs\bar{s}) \in H.$$
        Now, by Case (C2), we know that $x_{[\gamma] + [\beta]} (\pm rs\bar{s}) \in H$. Hence 
        $$ x_{[\beta]}(rs) = \{x_{[\beta]} (rs) x_{[\gamma] + [\beta]} (\pm rs\bar{s})\} \{x_{[\gamma] + [\beta]} (\pm rs\bar{s})\}^{-1} \in H.$$
        In particular, $x_{[\beta]}(s) \in H$. By Case (C1) and Case (C2), we have $x_{[\beta]} (s + \bar{s}) \in H$. But then $x_{[\beta]}(\bar{s}) \in H$. Now, by switching roles of ${[\beta]}$ and $[\gamma]$, we get $x_{[\gamma]}(\bar{s}) \in H$. If we replace $s$ by $\bar{s}$ in above argument, then we get $x_{[\beta]}(r\bar{s}) \in H$, for every $r \in R_\theta$. 
        
        \item[\textbf{(C4)}] \textbf{$[\gamma]$ and $\mu$ are arbitrary.} Observe that, we can find a sequence $[\gamma_1], \dots, [\gamma_m]$ of roots in $\Phi'$ such that $2 \leq m \leq 8, [\gamma_1] = [\gamma], [\gamma_m] = [\beta]$ and the angle between $[\gamma_i]$ and $[\gamma_{i+1}]$ is $\pi / 4$ for every $i = 1, \dots, m-1$. 
        \begin{enumerate}[1.]
            \item If $[\gamma] \sim A_1$ and $[\beta] \sim A_1$ then $m$ must be odd. By applying Case (C1) to pairs $([\gamma_{i-1}], [\gamma_{i}])$, Case (C2) to pairs $([\gamma_i], [\gamma_{i+1}])$ for $i = 2, 4, \dots, m-1$, inductively, we get $x_{[\beta]}((r + \bar{r}) s) \in H$ for every $r \in R$. Now by Lemma \ref{theta Rz}, we have $x_{[\beta]}(t) \in H$ for every $t \in I_{[\beta]}$.
            \item If $[\gamma] \sim A_1$ and $[\beta] \sim A_1^2$ then $m$ must be even. Note that $[\gamma_{m-1}]$ is of type $A_1$. Hence, by the above case, we have $x_{[\gamma_{m-1}]}(s) \in H$. Now by applying Case (C1) to the pair $([\gamma_{m-1}], [\gamma_m])$, we conclude that $x_{[\beta]}(t) \in H$ for every $t \in I_{[\beta]}$.
            \item If $[\gamma] \sim A_1^2$ and $[\beta] \sim A_1$ then $m$ must be even. By applying Case (C3) to pairs $([\gamma_i], [\gamma_{i+2}])$ for $i = 1, 3, \dots, m-3$, inductively, we get $x_{[\gamma_{m-1}]}(s) \in H$. Now by applying Case (C2) to the pair $([\gamma_{m-1}],[\gamma_m])$, we get $x_{[\beta]}(r \bar{s} + \bar{r} s) \in H$ for every $r \in R$. Finally by Lemma \ref{theta Rz}, we have $x_{[\beta]}(t) \in H$ for all $t \in I_{[\beta]}$.
            \item If $[\gamma] \sim A_1^2$ and $[\beta] \sim A_1^2$ then $m$ must be odd. By applying Case (C2) to pairs $([\gamma_{i-1}], [\gamma_{i}])$, Case (C1) to pairs $([\gamma_i], [\gamma_{i+1}])$ for $i = 2, 4, \dots, m-1$, inductively, we get $x_{[\beta]}(r_1(r_2 \bar{s} + \bar{r_2} s)) \in H$ for every $r_1, r_2 \in R$. Now by applying Case (C3) to pairs $([\gamma_{i}, \gamma_{i+2}])$ for $i=1,3,\dots, m-2,$ inductively, we get $x_{[\beta]}(r_3 s)$ and $x_{[\beta]}(r_4 \bar{s}) \in H$ for every $r_3, r_4 \in R_\theta$. 
            But then for every $r_5 \in R$, we have $$x_{[\beta]}(r_5^2 s) = x_{[\beta]}(r_5(r_5s + \bar{r_5}\bar{s})) \{x_{[\beta]}(r_5 \bar{r_5} \bar{s})\}^{-1} \in H.$$ Similarly, we have $x_{[\beta]}(r_6^2 \bar{s}) \in H$ for every $r_6 \in R$. 
            Finally, for given $r \in R$ we have 
            $$ x_{[\beta]}(rs) = x_{[\beta]}\Big(\Big(\frac{1+r}{2}\Big)^2s\Big) \Big\{x_{[\beta]}\Big(\Big(\frac{1-r}{2}\Big)^2s\Big) \Big\}^{-1} \in H.$$
            Similarly, we have $x_{[\beta]}(r \bar{s}) \in H$ for every $r \in R$. By Lemma \ref{theta Rz}, we conclude that $x_{[\beta]}(t) \in H$ for all $t \in I_{[\beta]}.$
        \end{enumerate}
    \end{enumerate}

    \vspace{2mm}

    \noindent \textbf{Case D. $\Phi' \sim B_2$ which contains roots of the type $A_1^2$ and $A_2$.} This case arises only when $\Phi_\rho \sim {}^2 A_{2n} \ (n \geq 2)$. Let us consider the following subcases: 
    \begin{enumerate}
        \item[\textbf{(D1)}] \textbf{$[\gamma] \sim A_1^2$ and $\mu = \pi /4$.} First observe that the pair $[\gamma]$ and $[\beta] - [\gamma]$ is of the type $(d-ii)$. Here $[\gamma], 2[\beta] - [\gamma] \sim A_1^2$ and $[\beta], [\beta] - [\gamma] \sim A_2$. In this case, for $r = (r_1, r_2) \in \mathcal{A}(R)$ we have 
        \begin{equation*}\label{eqn_a}
            [x_{[\gamma]}(s), x_{[\beta] - [\gamma]}(\pm r_1'/2, r_2'/4)] = x_{[\beta]} (r_1 s'/2 , r_2 s \bar{s}/4) x_{2[\beta] - [\gamma]} (\pm r_2' s/4) \in H.
        \end{equation*}
        By putting $-r_1$ instead of $r_1$, we get $x_{[\beta]} (-r_1 s'/2, r_2 s \bar{s}/4) x_{2[\beta] - [\gamma]} (\pm r_2' s/4) \in H$. But then
        \begin{equation*}\label{eqn_b}
            \begin{split}
                x_{[\beta]}(r_1 s', \frac{r_1 \bar{r_1}}{2}s \bar{s}) &= x_{[\beta]}(r_1 s', \frac{r_2 + \bar{r_2}}{2}s \bar{s}) \\
                &= \{x_{[\beta]} (r_1 s'/2 , r_2 s \bar{s}/4) x_{2[\beta] - [\gamma]} (\pm r_2 s/4)\} \\ 
                & \hspace{10mm} \{x_{[\beta]} (-r_1 s'/2, r_2 s \bar{s}/4) x_{2[\beta] - [\gamma]} (\pm r_2 s/4)\}^{-1} \in H.
            \end{split}
        \end{equation*}
        Now observe that the pair $[\gamma]$ and $2[\beta] - [\gamma]$ is of type $(a_{2}-ii)$. Here $[\gamma], [\delta]:=2[\beta] - [\gamma] \sim A_1^2$ and $[\beta] = 1/2([\gamma] + [\delta]) \sim A_2$. In this case, for $r \in R$ we have 
        \begin{equation*}\label{eqn_c}
            [x_{[\gamma]}(s), x_{[\delta]}(\pm r)] = x_{[\beta]}(0, \bar{r}s - r \bar{s}) \in H.
        \end{equation*}
        We now claim that for every $(r_3 s', r_4 s') \in \mathcal{A}(Rs')$, the element $x_{[\beta]}(r_3 s', r_4 s') \in H$. As in proof of Lemma \ref{A(R)}, we consider the following decomposition of $(r_3 s', r_4 s'):$ 
        $$(r_3 s', r_4 s') = (r_3 s', r_3 \bar{r_3} s \bar{s}/2) \oplus (0, (r_4 s' - \overline{r_4 s'})/2).$$
        Therefore, we have 
        \begin{equation*}\label{eqn_d}
            x_{[\beta]}(r_3 s', r_4 s') = x_{[\beta]}(r_3 s', r_3 \bar{r_3} s \bar{s}/2) x_{[\beta]}(0, (r_4 s' - \overline{r_4 s'})/2) \in H.
        \end{equation*}
        Finally, by Proposition \ref{prop wxw^{-1}} (or by Proposition $4.1$ of \cite{EA1}), we have 
        \begin{equation}\label{eq_wx=xw}
            w_{2[\beta] - [\gamma]}(1) x_{[\gamma]} (s) {w_{2[\beta] - [\gamma]}(1)}^{-1} = x_{[\gamma]}(-\bar{s}).
        \end{equation}
        Therefore, $x_{[\gamma]}(-\bar{s}) \in H$ and hence $x_{[\gamma]}(\bar{s}) \in H.$ Replacing $s$ by $\bar{s}$ in above, we get $x_{[\beta]}(r_3 \bar{s'}, r_4 \bar{s'}) \in H$ for every $(r_3 \bar{s'}, r_4 \bar{s'}) \in \mathcal{A}(R\bar{s'})$. In particular, we have $x_{[\beta]}(r_3 s, r_4 s) \in H$ (resp., $x_{[\beta]}(r_3 \bar{s}, r_4 \bar{s}) \in H$) for every $(r_3 s, r_4 s) \in \mathcal{A}(Rs)$ (resp., $(r_3 \bar{s}, r_4 \bar{s}) \in \mathcal{A}(R\bar{s})$).
        
        \item[\textbf{(D2)}] \textbf{$[\gamma] \sim A_2$ and $\mu = \pi /4$.} Note that, the pair $[\gamma]$ and $[\beta] - [\gamma]$ is of the type $(c-ii)$. Here $[\gamma], [\beta]-[\gamma] \sim A_2$ and $[\beta] \sim A_1^2$. For $r \in R$, we have
        \begin{equation}\label{eq_d2_1}
            [x_{[\gamma]}(s_1, s_2), x_{[\beta] - [\gamma]}(\pm r', r \bar{r}/2)] = x_{[\beta]} (r s_1') \in H.
        \end{equation}
        By (\ref{eq_wx=xw}), we also have $x_{[\beta]} (\bar{r} \bar{s_1'}) \in H$. In particular, for every $r \in R$, we have $x_{[\beta]}(rs_1) \in H$ and $x_{[\beta]}(r \bar{s_1}) \in H$.

        Observe that, the pair $[\gamma]$ and $[\beta] - 2[\gamma]$ is of the type $(d-ii)$. Here $[\beta] - [\gamma], [\gamma] \sim A_2$ and $[\beta], [\beta] - 2 [\gamma] \sim A_1^2$. For $r \in R$ we have 
        \begin{equation}\label{eq_d2_2}
            [x_{[\beta] - 2[\gamma]}(\pm r'), x_{[\gamma]}(s_1, s_2)] = x_{[\beta] - [\gamma]} (rs_1', r \bar{r} s_2') x_{[\beta]} (\pm r's_2'') \in H. 
        \end{equation}
        Note that, $(x_{[\gamma]}(s_1, s_2))^{-1} = x_{[\gamma]}(-s_1, \bar{s_2}) \in H$. By above, we have $x_{[\beta]}(s_1) \in H$. By applying case (D1) to the pair $([\beta], [\gamma])$, we have $x_{[\gamma]}(2s_1, 2s_1\bar{s_1}) \in H$. But then 
        $$x_{[\gamma]}(s_1, \bar{s_2}) = x_{[\gamma]}(-s_1, \bar{s_2}) x_{[\gamma]}(2s_1, 2s_1\bar{s_1}) \in H.$$

        Put $(s_1, \bar{s_2})$ instead of $(s_1, s_2)$ and $-r$ instead of $r$ in (\ref{eq_d2_2}), we get $$x_{[\beta] - [\gamma]} (- r s_1', r \bar{r} \bar{s_2'}) x_{[\beta]} (\pm (-r') \bar{s_2''}) \in H.$$ Since the elements $x_{[\beta] - [\gamma]} (\cdot)$ and $x_{[\beta]} (\cdot)$ are commutes with each other, we have 
        \begin{align*}
            x_{[\beta]}(\pm r'(s_2'' - \bar{s_2''})) &=  \{x_{[\beta] - [\gamma]} (r s_1', r \bar{r} s_2') x_{[\beta]} (\pm r' s_2'')\} \\
            & \hspace{15mm} \{ x_{[\beta] - [\gamma]} (- r s_1', r \bar{r} \bar{s_2'}) x_{[\beta]} (\pm (-r') \bar{s_2''}) \} \in H.
        \end{align*}
        In particular, we have $x_{[\beta]} (\frac{r}{2} (s_2 - \bar{s_2})) \in H$ for every $r \in R$. Again by (\ref{eq_d2_1}), we have $x_{[\beta]}(\frac{r}{2} (s_1 \bar{s_1})) = x_{[\beta]} (\frac{r}{2} (s_2 + \bar{s_2})) \in H$ for every $r \in R$. But then $$ x_{[\beta]}(r s_2) = x_{[\beta]} (\frac{r}{2} (s_2 + \bar{s_2})) x_{[\beta]}(\frac{r}{2} (s_2 - \bar{s_2})) \in H,$$
        for every $r \in R$. Similarly, we have $x_{[\beta]}(r \bar{s_2}) \in H$ for every $r \in R$.
        
        \item[\textbf{(D3)}] \textbf{$[\gamma]$ and $\mu$ are arbitrary.} Observe that, we can find a sequence $[\gamma_1], \dots, [\gamma_m]$ of roots in $\Phi'$ such that $2 \leq m \leq 8, [\gamma_1] = [\gamma], [\gamma_m] = [\beta]$ and the angle between $[\gamma_i]$ and $[\gamma_{i+1}]$ is $\pi / 4$ for every $i = 1, \dots, m-1$.
        \begin{enumerate}[1.]
            \item If $[\gamma] \sim A_1^2$ and $[\beta] \sim A_1^2$ then $m$ must be odd. By applying Case (D1) to pairs $([\alpha_i], [\alpha_{i+1}])$ and Case (D2) to pairs $([\alpha_{i+1}], [\alpha_{i+2}])$ for every $i = 1, 3, \dots, m-2$, recursively, we get $x_{[\beta]}(rs) \in H$ and $x_{[\beta]}(r\bar{s}) \in H$ for every $r \in R$. Therefore, we have $x_{[\beta]}(t) \in H$ for every $t \in I_{[\beta]}$.

            \item If $[\gamma] \sim A_1^2$ and $[\beta] \sim A_2$ then $m$ must be even. Note that $[\gamma_{m-1}] \sim A_1^2$, by the above case, $x_{[\gamma_{m-1}]}(s) \in H$ and $x_{[\gamma_{m-1}]}(\bar{s}) \in H$. By applying case (D1) to the pair $([\alpha_{m-1}], [\alpha_m])$, we get $x_{[\beta]}(r_1s, r_2s) \in H$ and $x_{[\beta]}(r_1 \bar{s}, r_2 \bar{s}) \in H$ for every $(r_1s,r_2s) \in \mathcal{A}(Rs)$. Finally, by the proof of Lemma \ref{A(I+J) = A(I) A(J)}, we have $x_{[\beta]}(t) \in H$ for every $t \in I_{[\beta]}$.

            \item If $[\gamma] \sim A_2$ and $[\beta] \sim A_1^2$ then $m$ must be even. Note that $[\gamma_{2}] \sim A_1^2$, by applying case (D2) to the pair $([\gamma_1], [\gamma_2])$, we get $x_{[\gamma_2]}(r s_1), x_{[\gamma_2]}(r \bar{s_1}), x_{[\gamma_2]}(r s_2), x_{[\gamma_2]}(r \bar{s_2}) \in H$ for every $r \in R$. Now by Case 1 above, we have $x_{[\beta]}(r s_1),$ $x_{[\beta]}(r s_2),$ $x_{[\beta]}(r \bar{s_1}),$ $x_{[\beta]}(r \bar{s_2}) \in H$ for every $r \in R$.
            Therefore, we have $x_{[\beta]}(t) \in H$ for every $t \in I_{[\beta]}$.

            \item If $[\gamma] \sim A_2$ and $[\beta] \sim A_2$ then $m$ must be odd. Note that $[\gamma_{m-1}] \sim A_1^2$, by the case 3 above, we have $x_{[\gamma_{m-1}]}(s_1), x_{[\gamma_{m-1}]}(s_2), x_{[\gamma_{m-1}]}(\bar{s_1}), x_{[\gamma_{m-1}]}(\bar{s_2}) \in H$. By applying case (D1) to the pair $([\gamma_{m-1}], [\gamma_{m}])$, we have $x_{[\beta]}(r_1 s_1, r_2 s_1), x_{[\beta]}(r_1 \bar{s_1}, r_2 \bar{s_1}) \in H$ for every $(r_1 s_1, r_2 s_1) \in \mathcal{A}(Rs_1)$ and $x_{[\beta]}(r_3 s_2, r_4 s_2), x_{[\beta]}(r_3 \bar{s_2}, r_4 \bar{s_2}) \in H$ for every $(r_3 s_2, r_4 s_2) \in \mathcal{A}(Rs_2)$. Therefore, by the proof of Lemma \ref{A(I+J) = A(I) A(J)}, we have $x_{[\beta]}(t) \in H$ for every $t \in I_{[\beta]}$.
        \end{enumerate}
    \end{enumerate}

    \vspace{2mm}

    \noindent \textbf{Case E. $\Phi' \sim G_2$ which contains roots of the type $A_1$ and $A_1^3$.} This case arises only when $\Phi_\rho \sim {}^3 D_{4}$. Let us consider the following subcases:
    \begin{enumerate}
        \item[\textbf{(E1)}] \textbf{$[\gamma] \sim A_1$ and $\mu = \pi/6$.} In this case, the pair $[\gamma]$ and $[\beta] - [\gamma]$ is of the type $(e)$. Here $[\gamma], 3[\beta] - [\gamma], 3[\beta] - 2[\gamma] \sim A_1$ and $[\beta], [\beta] - [\gamma], 2[\beta] - [\gamma] \sim A_1^3$. For $r \in R$, we have
        \begin{gather*}
            [x_{[\gamma]}(s), x_{[\beta] - [\gamma]}(\pm r/2)] = x_{[\beta]} (rs/2) x_{2[\beta] - [\gamma]} (\pm r r' s /4) \\
            x_{3[\beta] - 2[\gamma]} (\pm r \bar{r} \bar{\bar{r}} s/8) x_{3[\beta] - [\gamma]} (\pm r \bar{r} \bar{\bar{r}} s^2/8) \in H.
        \end{gather*}
        Note that, long roots in $G_2$ form a subsystem of type $A_2$. Hence, by Case A, we have $x_{3[\beta] - 2[\gamma]} (\pm r \bar{r} \bar{\bar{r}} s/8) \in H$ and $x_{3[\beta] - [\gamma]} (\pm r \bar{r} \bar{\bar{r}} s^2/8) \in H$. But then 
        \begin{gather*}
             \{ x_{[\beta]} (rs/2) x_{2[\beta] - [\gamma]} (\pm r r' s/4) x_{3[\beta] - 2[\gamma]} (\pm r \bar{r} \bar{\bar{r}} s/8) x_{3[\beta] - [\gamma]} (\pm r \bar{r} \bar{\bar{r}} s^2/8) \} \\
             \{x_{3[\beta] - 2[\gamma]} (\pm r \bar{r} \bar{\bar{r}} s/8) x_{3[\beta] - [\gamma]} (\pm r \bar{r} \bar{\bar{r}} s^2/8)\}^{-1} \\
             = x_{[\beta]} (rs/2) x_{2[\beta] - [\gamma]} (\pm r r' s/4) \in H.
        \end{gather*}
        Now put $-r$ instead of $r$, we get $x_{[\beta]} (-rs/2) x_{2[\beta] - [\gamma]} (\pm r r' s/4) \in H$. Finally, 
        $$ x_{[\beta]} (rs) = \{ x_{[\beta]} (rs/2) x_{2[\beta] - [\gamma]} (\pm r r' s/4) \} \{ x_{[\beta]} (-rs/2) x_{2[\beta] - [\gamma]} (\pm r r' s/4) \}^{-1} \in H.$$

        \item[\textbf{(E2)}] \textbf{$[\gamma] \sim A_1^3$ and $\mu = \pi/6$.} In this case, the pair $[\gamma]$ and $[\beta] - [\gamma]$ is of type $(g)$. Here $[\gamma] \sim A_1^3, [\beta] - [\gamma] \sim A_1^3$ and $[\beta] \sim A_1$. For $r \in R$, we have
        $$ [x_{[\gamma]}(s), x_{[\beta] - [\gamma]}(\pm r')] = x_{[\beta]} (rs + \bar{r} \bar{s} + \bar{\bar{r}} \bar{\bar{s}}) \in H.$$

        \item[\textbf{(E3)}] \textbf{$[\gamma] \sim A_1^3$ and $\mu = \pi/3$.} In this case, the pair $[\gamma]$ and $[\beta] - [\gamma]$ is of type $(f)$. Here $[\gamma] \sim A_1^3, [\beta] - [\gamma] \sim A_1^3, [\beta] \sim A_1^3, 2[\beta] - [\gamma] \sim A_1$ and $[\gamma] + [\beta] \sim A_1$. For $r \in R$, we have
        \begin{gather*}
            [x_{[\gamma]}(s), x_{[\beta] - [\gamma]}(\pm r/2)] = x_{[\beta]} ((r' s \pm r s')/2) x_{[\gamma] + [\beta]} (\pm(r \bar{s} \bar{\bar{s}} + \bar{r} s \bar{\bar{s}} + \bar{\bar{r}} s \bar{s})/2) \\
            x_{2[\beta] - [\gamma]} (\pm(r \bar{r} \bar{\bar{s}} + \bar{r} \bar{\bar{r}} s + r \bar{\bar{r}} \bar{s})/2) \in H.
        \end{gather*}
        Here $(s',r') = (\bar{s}, \bar{r})$ or $(\bar{\bar{s}}, \bar{\bar{r}})$.
        Note that $x_{[\gamma]}(-s) = \{x_{[\gamma]}(s)\}^{-1} \in H$. Hence if we replace $r$ (resp., $s$) by $-r$ (resp., $-s$), then 
        \begin{gather*}
            x_{[\beta]} ((r' s \pm r s')/2) x_{[\gamma] + [\beta]} (\pm(-r \bar{s} \bar{\bar{s}} - \bar{r} s \bar{\bar{s}} - \bar{\bar{r}} s \bar{s})/2) \\
            x_{2[\beta] - [\gamma]} (\pm(-r \bar{r} \bar{\bar{s}} - \bar{r} \bar{\bar{r}} s - r \bar{\bar{r}} \bar{s})/2) \in H.
        \end{gather*}
        Since $x_{[\beta]}(\cdot), x_{[\gamma] + [\beta]}(\cdot)$ and $x_{2[\beta] - [\gamma]}(\cdot)$ commutes with each other, we have
        \begin{gather*}
            x_{[\beta]} (r' s \pm r s') = \{ x_{[\beta]} ((r' s \pm r s')/2) x_{[\gamma] + [\beta]} (\pm(r \bar{s} \bar{\bar{s}} + \bar{r} s \bar{\bar{s}} + \bar{\bar{r}} s \bar{s})/2) \\ 
            x_{2[\beta] - [\gamma]} (\pm(r \bar{r} \bar{\bar{s}} + \bar{r} \bar{\bar{r}} s + r \bar{\bar{r}} \bar{s})/2) \} \{ x_{2[\beta] - [\gamma]} (\pm(-r \bar{r} \bar{\bar{s}} - \bar{r} \bar{\bar{r}} s - r \bar{\bar{r}} \bar{s})/2) \\
            x_{[\gamma] + [\beta]} (\pm(-r \bar{s} \bar{\bar{s}} - \bar{r} s \bar{\bar{s}} - \bar{\bar{r}} s \bar{s})/2) x_{[\beta]} ((r' s \pm r s')/2)\} \in H.
        \end{gather*}

        \item[\textbf{(E4)}] \textbf{$[\gamma]$ and $\mu$ are arbitrary.} Observe that, we can find a sequence $[\gamma_1], \dots, [\gamma_m]$ of roots in $\Phi'$ such that $2 \leq m \leq 12, [\gamma_1] = [\gamma], [\gamma_m] = [\beta]$ and the angle between $[\gamma_i]$ and $[\gamma_{i+1}]$ is $\pi / 6$ for every $i = 1, \dots, m-1$. 
        \begin{enumerate}[1.]
            \item If $[\gamma] \sim A_1$ and $[\beta] \sim A_1$ then $m$ must be odd. By applying Case (E1) to pairs $([\gamma_i], [\gamma_{i+1}])$ and Case (E2) to pairs $([\gamma_{i+1}], [\gamma_{i+2}])$ for $i=1, 3, \dots, m-2$, recursively, we get, $x_{[\beta]} ((r +\bar{r} + \bar{\bar{r}})s) \in H$ for every $r \in R$. Now by Lemma \ref{theta Rz}, we have $x_{[\beta]}(t) \in H,$ for every $t \in I_{[\beta]}$.

            \item If $[\gamma] \sim A_1$ and $[\beta] \sim A_1^3$ then $m$ must be even. Since $[\gamma_{m-1}] \sim A_1$, by the above Case, we have  $x_{[\gamma_{m-1}]}(s) \in H$. Now by applying Case (E1) to the pair $([\gamma_{m-1}], [\gamma_{m}])$, we get $x_{[\beta]} (rs) \in H$ for every $r \in R$. Hence we have $x_{[\beta]}(t) \in H,$ for every $t \in I_{[\beta]} = I$.
            
            \item If $[\gamma] \sim A_1^3$ and $[\beta] \sim A_1$ then $m$ must be even. By apply Case (E2) to pairs $([\gamma_i], [\gamma_{i+1}])$ and Case (E1) to pairs $([\gamma_{i+1}], [\gamma_{i+2}])$ for $i=1, 3, \dots, m-3$, recursively, and finally Case (E2) to the pair $([\gamma_{m-1}],[\gamma_m])$, we get $x_{[\beta]} (r s + \bar{r} \bar{s} + \bar{\bar{r}} \bar{\bar{s}}) \in H$ for every $r \in R$. Now by Lemma \ref{theta Rz}, we have $x_{[\beta]}(t) \in H,$ for every $t \in I_{[\beta]}$.
            
            \item If $[\gamma] \sim A_1^3$ and $[\beta] \sim A_1^3$ then $m$ must be odd and $m \geq 3$. We first show that $x_{[\gamma_3]}(t) \in H$ for every $t \in I$. Note that $[\gamma_3] \sim A_1^3$. By applying Case (E2) to the pair $([\gamma_{1}], [\gamma_{2}])$, we get $x_{[\gamma_2]} (r s + \bar{r} \bar{s} + \bar{\bar{r}} \bar{\bar{s}}) \in H$ for every $r \in R$. Next, by applying Case (E1) on the pair $([\gamma_2], [\gamma_3])$, we conclude that $x_{[\gamma_3]} (r s + \bar{r} \bar{s} + \bar{\bar{r}} \bar{\bar{s}}) \in H$ for every $r \in R$. Finally, by applying Case (E3) on the pair $([\gamma_1], [\gamma_3])$, we obtain $x_{[\gamma_3]}(r s \pm r' s') \in H$ for every $r \in R$ (where $a'$ denotes $\bar{a}$ or $\bar{\bar{a}}$). 
            
            Suppose $x_{[\gamma_3]}(r s + r' s') \in H$. Then $x_{[\gamma_3]}(r'' s'') \in H$ for every $r \in R$ (where $a'' = \bar{a}$ if $a' = \bar{\bar{a}}$ and vice-versa). B reversing the roles of $[\gamma_1]$ and $[\gamma_3]$ together with roles of $s$ and $s''$, we get $x_{[\gamma_1]}(r s + \bar{r} \bar{s} + \bar{\bar{r}} \bar{\bar{s}}) \in H$ and $x_{[\gamma_1]}(r's') \in H$. Since $x_{[\gamma_1]}(s) \in H$, it follows that $x_{[\gamma_1]}(\bar{s})$ and $x_{[\gamma_1]}(\bar{\bar{s}})$ are also in $H$. Applying the same process again with $\bar{s}$ (resp., $\bar{\bar{s}}$), we obtain $x_{[\gamma_3]}(r \bar{s}'') \in H$ (resp., $x_{[\gamma_3]}(r \bar{\bar{s}}'') \in H$) for every $r \in R$. In particular, we have $x_{[\gamma_3]}(t) \in H$ for every $t \in I_{[\gamma_3]}$. 
            
            Now, suppose $x_{[\gamma_3]}(rs - r's') \in H$. Then $x_{[\gamma_3]}(2rs + r'' s'') \in H$ for every $r \in R$. By reversing the roles of $[\gamma_1]$ and $[\gamma_3]$ together with roles of $s$ and $2s + s''$, we obtain $x_{[\gamma_1]}((2r+r')s + (2\bar{r}+ \bar{r'}) \bar{s} + (2 \bar{\bar{r}} + \bar{\bar{r'}}) \bar{\bar{s}}) \in H$ and $x_{[\gamma_1]}(4s + 4s'' + s') \in H$. Since the map $r \mapsto 2r + r'$ from $R$ to itself is surjective, we have $x_{[\gamma_1]}(rs + \bar{r}\bar{s} + \bar{\bar{r}} \bar{\bar{s}}) \in H$ for every $r \in R$. Consequently, $x_{[\gamma_1]} (3s') \in H$, and hence $x_{[\gamma_1]}(\bar{s})$ and $x_{[\gamma_1]}(\bar{\bar{s}}) \in H$ (as $1/3 \in R$). Applying the same process again with $\bar{s}$ and $\bar{\bar{s}}$, we get $x_{[\gamma_3]}(3rs) = x_{[\gamma_3]}(rs + \bar{r} \bar{s} + \bar{\bar{r}} \bar{\bar{s}}) x_{[\gamma_3]} (rs - \bar{r} \bar{s}) x_{[\gamma_3]} (rs - \bar{\bar{r}} \bar{\bar{s}}) \in H$ for every $r \in R$. Thus, we get $x_{[\gamma_3]}(rs) \in H$ for every $r \in R$. Similarly, we can show that $x_{[\gamma_3]}(r \bar{s}) \in H$ and $x_{[\gamma_3]}(r \bar{\bar{s}}) \in H$ for every $r \in R$. In particular, we have $x_{[\gamma_3]}(t) \in H$ for every $t \in I_{[\gamma_3]}$, as required. 
            
            Now if $m=3$, then we are done. If not, we repeat this process for the pair $([\gamma_{i}], [\gamma_{i+2}])$ for every $i = 3, \dots, m-2$ to obtain desired result.
        \end{enumerate}
    \end{enumerate}
    This completes the proof of our lemma.
\end{proof}

%%%%%%%%%%%%%%%%%%%%%%%%%%%%%%%%%%%%%%%%%%%%%%%%%%

\vspace{2mm}

\noindent \textit{Proof of Proposition \ref{prop:E(R,J) subset of H}.} Let $J$ be as in the hypothesis of Proposition \ref{prop:E(R,J) subset of H}. Let $t, u \in J$. Then there exists $[\alpha], [\beta] \in \Phi_\rho$ such that $t \in J_{[\alpha]}(H)$ and $u \in J_{[\beta]}(H)$. Let $[\gamma] \in \Phi_\rho$ be such that it is either of type $A_1^2$ or $A_1^3$. By Proposition \ref{z to Rz}, for every $r \in R$ we have $x_{[\gamma]}(rt), x_{[\gamma]}(\bar{t}), x_{[\gamma]}(u) \in H$ and hence $x_{[\gamma]} (t + u) \in H$. Therefore, if $t,u \in J$ and $r \in R$ then we have $t+u, rt, \bar{t} \in J$. Thus $J$ is a $\theta$-invariant ideal of $R$. Now for the second assertion, it follows from Proposition \ref{z to Rz} that $E'_\sigma (J) \subset H$. Since $H$ is normalized by $E'_\sigma (R)$, we conclude that $E'_\sigma (R,J) \subset H$, as desired. \qed

%%%%%%%%%%%%%%%%%%%%%%%%%%%%%%%%%%%%%%%%%%%%%%%%%%
%Section: Proof of Proposition \ref{prop:U(R) cap H subset U(J)}
%%%%%%%%%%%%%%%%%%%%%%%%%%%%%%%%%%%%%%%%%%%%%%%%%%

\section{Proof of Proposition \ref{prop:U(R) cap H subset U(J)}} \label{sec:Pf of prop 2}

%%%%%%%%%%%%%%%%%%%%%%%%%%%%%%%%%%%%%%%%%%%%%%%%%%

Let $\Phi_\rho$ be an irreducible root system. We fix a simple system $\Delta_\rho = \{ [\alpha_1], \dots, [\alpha_l] \}$ of $\Phi_\rho$. Recall that, for a root $[\alpha] = \sum_{i=1}^l m_{i} [\alpha_i] \in \Phi_\rho$, we defined $ht([\alpha])= \sum_{i=1}^l m_i$. We say a root $[\beta]$ is \textit{highest} if the height of $[\beta]$ is maximal, i.e., $ht([\beta]) = \max \{ ht([\alpha]) \mid [\alpha] \in \Phi_{\rho} \}$. Note that there is a unique highest root in an irreducible root system and it is a long positive root. 
Therefore we sometimes call it \textit{highest long root}. Similarly, we say $[\gamma]$ is a \textit{highest short root} if $ht([\gamma]) = \max \{ ht([\alpha]) \mid [\alpha] \in \Phi_{\rho} \text{ and } [\alpha] \text{ is short root} \}$. There is a unique highest short root in an irreducible root system.

\begin{lemma}\label{lemma:U cap H}
    \normalfont
    Let $x: = \prod_{[\alpha] \in \Phi_\rho^{+}} x_{[\alpha]}(t_{[\alpha]}) \in U(R) \cap H$ with $t_{[\alpha]} \in R_{[\alpha]}$ (the product is taken over disjoint roots in any fixed order). Then $x_{[\alpha]}(t_{[\alpha]}) \in H$ for all $[\alpha] \in \Phi_\rho^{+}$.
\end{lemma}

\begin{proof}
    The proof closely resembles that of Lemma 3.1 and Proposition 1 in Section 3 of \cite{EA3}. However, the calculations presented here are distinct from those in \cite{EA3}. For the convenience of the reader, we provide the full proof below.

    \vspace{2mm}
    
    \noindent \textbf{Case A. $\Phi_\rho \sim {}^2 A_3$:} In this case, after the twist, $\Phi_\rho$ becomes a root system of type $B_2$. Let $[\alpha]$ and $[\beta]$ be the simple roots, with $[\alpha]$ being the long root. We first claim that 
    \begin{equation}\label{eq_a}
        \textit{if } x = x_{[\alpha] + [\beta]}(t) x_{[\alpha]+ 2 [\beta]}(u) \in H, \textit{ then } x_{[\alpha] + [\beta]}(t), x_{[\alpha]+ 2 [\beta]}(u) \in H. 
    \end{equation}
    For any $r \in R_{[\alpha]} = R_\theta$, we have
    $$ H \ni [x_{-[\alpha]}(r), x] = x_{[\beta]} (\pm rt) x_{[\alpha] + 2 [\beta]} (\pm r t \bar{t}).$$
    Since $x \in H$ then so is $x^{-1} = x_{[\alpha] + [\beta]}(-t) x_{[\alpha]+ 2 [\beta]}(-u)$. Therefore we can replacing $t$ and $u$ by $-t$ and $-u$, respectively, and we get $x_{[\beta]} (\pm r(-t)) x_{[\alpha] + 2 [\beta]} (\pm r t \bar{t}) \in H$. But then 
    $$x_{[\beta]}(\pm 2rt) = \{ x_{[\beta]} (\pm rt) x_{[\alpha] + 2 [\beta]} (\pm r t \bar{t}) \} \{ x_{[\beta]} (\pm r(-t)) x_{[\alpha] + 2 [\beta]} (\pm r t \bar{t}) \}^{-1} \in H.$$
    Put $r = \pm 1/2$, we get $x_{[\beta]}(t) \in H$. By Proposition \ref{z to Rz}, we get $x_{[\alpha] + [\beta]}(t) \in H$ and hence $x_{[\alpha]+ 2 [\beta]}(u) \in H$. This proves (\ref{eq_a}).
    Now let $x = x_{[\beta]}(t) x_{[\alpha]}(u) x_{[\alpha] + [\beta]} (v) x_{[\alpha] + 2[\beta]}(w) \in H$. Then
    $$ [x_{[\alpha]}(1), x] = x_{[\alpha] + [\beta]} (\pm t) x_{[\alpha]+2[\beta]} (\pm t \bar{t}) \in H.$$ 
    By (\ref{eq_a}), we have $x_{[\alpha] + [\beta]} (\pm t) \in H$. Again by Proposition \ref{z to Rz}, we get $x_{[\beta]} (t) \in H$. Consequently, $x_1:= x_{[\alpha]}(u) x_{[\alpha] + [\beta]} (v) x_{[\alpha] + 2[\beta]}(w) \in H$. Now,
    $$ [x_{[\beta]}(1), x_1] = x_{[\alpha] + [\beta]} (\pm u) x_{[\alpha]+2[\beta]} (\pm u \pm (v + \bar{v})) \in H.$$
    Again, by (\ref{eq_a}), $x_{[\alpha] + [\beta]} (\pm u) \in H$, and hence $x_{[\alpha]} (u) \in H$ (by Proposition \ref{z to Rz}). But then $$x_{[\alpha] + [\beta]} (v) x_{[\alpha] + 2[\beta]}(w) \in H.$$ Finally, by (\ref{eq_a}), $x_{[\alpha] + [\beta]} (v) \in H$ and $ x_{[\alpha] + 2[\beta]}(w) \in H$, as desired.
    
    \vspace{2mm}
    
    \noindent \textbf{Case B. $\Phi_\rho \sim {}^2 A_4$:} In this case, after the twist, $\Phi_\rho$ becomes a root system of type $B_2$. Let $[\alpha]$ and $[\beta]$ be the simple roots, with $[\alpha]$ being the long root. We first claim that 
    \begin{equation}\label{eq_b}
        \textit{if } x = x_{[\alpha] + [\beta]}(t) x_{[\alpha]+ 2 [\beta]}(u) \in H, \textit{ then } x_{[\alpha] + [\beta]}(t), x_{[\alpha]+ 2 [\beta]}(u) \in H. 
    \end{equation}
    For any $r = (r_1, r_2) \in R_{[\alpha]} = \mathcal{A}(R)$, we have
    \begin{equation}\label{eq_c}
        H \ni [x_{[\beta]}(r), x] = x_{[\alpha] + 2 [\beta]} (\pm r'_1 t'_1),
    \end{equation}
    where $r'_1$ denotes either $r_1$ or $\bar{r_1}$, similar for $t'_1$.
    Now $$ H \ni [x_{-[\beta]}(r), x] = x_{[\alpha]}(\pm r'_1 t'_1) x_{[\alpha]}(\pm r'_2 u) x_{[\alpha] + [\beta]}(\pm r'_1 u', r'_2 u \bar{u}).$$
    By (\ref{eq_c}) and Proposition \ref{z to Rz}, $x_{[\alpha]}(\pm r'_1 t'_1) \in H$. But then $x_{[\alpha]}(\pm r'_2 u) x_{[\alpha] + [\beta]}(\pm r'_1 u', r'_2 u \bar{u}) \in H$. Now if we put $(r_1, r_2) = (1,1/2)$, then we get $x_{[\alpha]}(\pm u/2) x_{[\alpha] + [\beta]} (\pm u', u \bar{u}/2) \in H$ and if we put $(r_1, r_2) = (-1, 1/2)$, then we get $x_{[\alpha]}(\pm u/2) x_{[\alpha] + [\beta]} (\pm (-u'), u \bar{u}/2) \in H$. But then
    $$x_{[\alpha]}(\pm u) = \{ x_{[\alpha]}(\pm u/2) x_{[\alpha] + [\beta]} (\pm u', u \bar{u}/2) \} \{ x_{[\alpha]}(\pm u/2) x_{[\alpha] + [\beta]} (\pm (-u'), u \bar{u}/2) \} \in H.$$
    Again by Proposition \ref{z to Rz}, we have $x_{[\alpha] + 2[\beta]} (u) \in H$ and hence $x_{[\alpha]+ [\beta]}(t) \in H$. This proves (\ref{eq_b}).
    
    Now let $x = x_{[\beta]}(s) x_{[\alpha]}(t) x_{[\alpha] + [\beta]} (u) x_{[\alpha] + 2[\beta]}(v) \in H$. Then
    $$ [x_{[\alpha]}(1), x] = x_{[\alpha] + [\beta]} (\pm s'_1, s'_2) x_{[\alpha] + [\beta]} (0, \pm (v - \bar{v})) x_{[\alpha]+2[\beta]} (\pm s_2) \in H,$$ where $s = (s_1, s_2) \in \mathcal{A}(R)$ and $s_1' = s_1$ or $\bar{s_1}$, similar for $s'_2$. 
    By (\ref{eq_b}), we have 
    $$ x_{[\alpha] + [\beta]} (\pm s'_1, s'_2 \pm (v - \bar{v})) \in H \text{ and } x_{[\alpha]+2[\beta]} (\pm s_2) \in H.$$ 
    Again by Proposition \ref{z to Rz}, we get $x_{[\beta]} (s)$. But then $x_1:= x_{[\alpha]}(t) x_{[\alpha] + [\beta]} (u) x_{[\alpha] + 2[\beta]}(v) \in H$. 
    Now, $$ [x_{[\beta]}(1, 1/2), x_1] = x_{[\alpha] + [\beta]} (\pm t', t \bar{t}/2 \pm (t \bar{u_1} - \bar{t}u_1)) x_{[\alpha]+2[\beta]} (\pm t \pm u'_1 ) \in H.$$
    Again by (\ref{eq_b}), $x_{[\alpha] + [\beta]} (\pm t', t \bar{t}/2 \pm (t \bar{u_1} - \bar{t}u_1)) \in H$ and hence $x_{[\alpha]} (t) \in H$ (by Proposition \ref{z to Rz}). But then $x_{[\alpha] + [\beta]} (u) x_{[\alpha] + 2[\beta]}(v) \in H$ and again by (\ref{eq_b}), $x_{[\alpha] + [\beta]} (u) \in H$ and $ x_{[\alpha] + 2[\beta]}(v) \in H$, as desired.
    
    \vspace{2mm}
    
    \noindent \textbf{Case C. $\Phi_\rho \sim {}^3 D_4$:} In this case, after the twist, $\Phi_\rho$ becomes a root system of type $G_2$. Let $[\alpha]$ and $[\beta]$ be the simple roots, with $[\alpha]$ being the long root. 
    We first claim that if $$x = x_{[\alpha] + 2[\beta]} (t) x_{[\alpha] + 3[\beta]} (u) x_{2[\alpha] + 3 [\beta]}(v) \in H,$$ then $x_{[\alpha] + 2[\beta]} (t), x_{[\alpha] + 3[\beta]} (u), x_{2[\alpha] + 3 [\beta]}(v) \in H.$
    Note that $[x_{[\alpha]}(1), x] = x_{2[\alpha] + 3[\beta]} (\pm u) \in H$ and $[x_{-[\alpha]}(1), x] = x_{[\alpha] + 3[\beta]} (\pm v) \in H$. But then, by Proposition \ref{z to Rz}, we have $x_{2[\alpha] + 3[\beta]} (v) \in H$, $x_{[\alpha] + 3[\beta]} (u) \in H$ and hence $x_{[\alpha] + 2[\beta]}(t) \in H$, which proves the claim. 
    We next claim that if $$x = x_{[\alpha] + [\beta]}(t) x_{[\alpha] + 2[\beta]}(u) x_{[\alpha] + 3[\beta]}(v) x_{2[\alpha] + 3[\beta]}(w) \in H,$$ then $x_{[\alpha] + [\beta]}(t), x_{[\alpha] + 2[\beta]}(u), x_{[\alpha] + 3[\beta]}(v), x_{2[\alpha] + 3[\beta]}(w) \in H$. For any $r \in R_\theta$, we have 
    $$y(r) = [x_{-[\alpha]}(r), x] = x_{[\beta]} (\pm rt) x_{[\alpha] + 2[\beta]}(\pm r t t') x_{[\alpha] + 3 [\beta]}(\pm r^2 t \bar{t} \bar{\bar{t}} \pm r w) x_{2[\alpha] + 3 [\beta]} (\pm r t \bar{t} \bar{\bar{t}}) \in H,$$ where $t' = \bar{t}$ or $\bar{\bar{t}}$. For any $s \in R_\theta$, we have
    \begin{align*}
        y(r,s) &= [x_{[\alpha]}(s), y(r)] = x_{[\alpha] + [\beta]} (\pm srt) x_{[\alpha] + 2[\beta]}(\pm s r^2 t t') x_{[\alpha]+3[\beta]}(\pm s r^3 t \bar{t} \bar{\bar{t}}) \\
        & \hspace{35mm} x_{2[\alpha] + 3[\beta]}(\pm s^2 r^3 t \bar{t} \bar{\bar{t}} \pm s r^2 t \bar{t} \bar{\bar{t}} \pm srw) \in H.
    \end{align*}
    Let $x_1:= y(r,s)^{-1} y (1,rs) = x_{[\alpha]+2[\beta]}(\pm s (r^2 - r)tt') x_{[\alpha]+3[\beta]}(v') x_{2[\alpha]+3[\beta]}(u') \in H$. By above we have $x_{[\alpha] + 2[\beta]}(\pm s (r^2 - r)tt') \in H$. Put $r = -1$ and $s=1/2$, we have $x_{[\alpha]+2[\beta]}(tt') \in H$. 
    But then, by Proposition \ref{z to Rz}, we have $x_{[\alpha]+2[\beta]}(\pm t t'), x_{[\alpha]+3[\beta]}(\pm t\bar{t} \bar{\bar{t}}), x_{2[\alpha] + 3[\beta]} (\pm t \bar{t} \bar{\bar{t}}) \in H$. Hence,
    $$y(1) x_{2[\alpha] + 3[\beta]} (\pm t \bar{t} \bar{\bar{t}})^{-1} x_{[\alpha]+3[\beta]}(\pm t\bar{t} \bar{\bar{t}})^{-1} x_{[\alpha]+2[\beta]}(\pm t t')^{-1} = x_{[\beta]}(\pm t) x_{[\alpha] + 3 [\beta]}(\pm w) \in H.$$
    Further, $[x_{-2[\alpha]-3[\beta]}(1), x_{[\beta]}(\pm t) x_{[\alpha] + 3[\beta]}(\pm w)] = x_{-[\alpha]} (\pm w) \in H$. By Proposition \ref{z to Rz}, we have $x_{[\alpha] + 3[\beta]}(\pm w) \in H$ and hence $x_{[\beta]}(\pm t) \in H$. Again by Proposition \ref{z to Rz}, $x_{[\alpha] + [\beta]}(t) \in H$ and hence $x_{[\alpha] + 2[\beta]}(u) x_{[\alpha] + 3[\beta]}(v) x_{2[\alpha] + 3[\beta]}(w) \in H$. By above claim, we have $x_{[\alpha] + 2[\beta]}(u)$, $x_{[\alpha] + 3[\beta]}(v)$, $x_{2[\alpha] + 3[\beta]}(w) \in H$, as desired. 
    
    Finally, let $x = x_{[\beta]}(t_1) x_{[\alpha]}(t_2) x_{[\alpha]+[\beta]}(t_3) x_{[\alpha]+2[\beta]}(t_4) x_{[\alpha] + 3[\beta]}(t_5) x_{2[\alpha]+3[\beta]} (t_6) \in H$. Note that 
    $$[x_{[\alpha]}(1), x] = x_{[\alpha] + [\beta]} (\pm t_1) x_{[\alpha] + 2[\beta]}(\pm t_1 t'_1) x_{[\alpha]+3[\beta]} (\pm t_1 \bar{t}_1 \bar{\bar{t}}_1) x_{2[\alpha]+3[\beta]}(\pm t_1 \bar{t}_1 \bar{\bar{t}}_1) \in H,$$
    where $t'_1 = \bar{t}_1$ or $\bar{\bar{t}}_1$. By above claim we have $x_{[\alpha]+[\beta]}(\pm t_1) \in H$, hence $x_{[\beta]}(t_1) \in H$ (by Proposition \ref{z to Rz}). Therefore, we have $$x_1:= x_{[\alpha]}(t_2) x_{[\alpha]+[\beta]}(t_3) x_{[\alpha]+2[\beta]}(t_4) x_{[\alpha] + 3[\beta]}(t_5) x_{2[\alpha]+3[\beta]} (t_6) \in H.$$
    Note that $$ [x_{[\beta]}(1), x_1] = x_{[\alpha] + [\beta]}(\pm t_2) x_{[\alpha]+2[\beta]}(s_4) x_{[\alpha] + 3[\beta]}(s_5) x_{2[\alpha]+3[\beta]} (s_6) \in H,$$ for some $s_4, s_5, s_6 \in R$. Again by above claim we have $x_{[\alpha] + [\beta]} (\pm t_2) \in H$ and hence $x_{[\alpha]}(t_2) \in H$. Thus $x_{[\alpha]+[\beta]}(t_3) x_{[\alpha]+2[\beta]}(t_4) x_{[\alpha] + 3[\beta]}(t_5) x_{2[\alpha]+3[\beta]} (t_6) \in H$. But again by above claim we have $x_{[\alpha]+[\beta]}(t_3), x_{[\alpha]+2[\beta]}(t_4), x_{[\alpha] + 3[\beta]}(t_5), x_{2[\alpha]+3[\beta]} (t_6) \in H,$ as desired.
    
    \vspace{2mm}
    
    \noindent \textbf{Case D. The rank of $\Phi_\rho > 2$:} Let $[\beta]$ be the highest long root in $\Phi_\rho$ and $[\beta']$ be the highest short root in $\Phi_\rho$. For $x = \prod_{[\alpha] \in \Phi^{+}_\rho} x_{[\alpha]} (t_{[\alpha]})$ (product is taken over some fixed order on the roots), we set $\Phi(x) = \{ [\alpha] \in \Phi^{+}_\rho \mid t_{[\alpha]} \neq 0 \}$. We use induction on $n$ to prove the following statement.

    \vspace{2mm}

    \noindent \textbf{($P_n$):} If $\Phi (x)$ only contains the roots $[\beta], [\beta']$ or $[\alpha]$ with $ht ([\alpha]) \geq ht([\beta]) - n + 1$. Then the conclusion of the lemma holds.

    \vspace{2mm}

    \noindent \textit{Proof of $(P_1)$}: We will show that if $x= x_{[\beta]}(t) x_{[\beta']}(t') \in H,$ then all factors of $x$ is contained in $H$. The subsystem generated by $[\beta]$ and $[\beta']$ is of the type ${}^2 A_{3}$ if $\Phi_\rho \sim {}^2 A_{2n-1} \ (n \geq 3), {}^2 D_{n} \ (n \geq 4)$ or ${}^2 E_6$ and is of type ${}^2 A_4$ if $\Phi_\rho \sim {}^2 A_{2n} \ (n \geq 3)$. Thus we are done by Case $A$ and Case $B$, above.

    \vspace{2mm}

    \noindent \textit{Proof of $(P_{n}) \implies (P_{n+1})$}: Assume that $(P_{n})$ holds, that is, assume that if $\Phi(x)$ only contains the roots $[\alpha]$ with $ht([\alpha]) \geq ht ([\beta]) - n + 1$, $[\beta]$ or $[\beta']$, then all factors of $x$ are contained in $H$. To prove $(P_{n+1})$, let $x \in H$ be such that $\Phi(x)$ only contains the roots  $[\alpha]$ with $ht([\alpha]) \geq ht ([\beta]) - n$, $[\beta]$ or $[\beta']$. It is enough to show that if $[\delta] \in \Phi (x)$ be such that $ht([\delta]) = ht ([\beta]) - n$ and $[\delta] \neq [\beta], [\beta']$ then $x_{[\delta]}(t_{[\delta]}) \in H$. Note that there exists a simple root $[\alpha_i] \in \Delta_\rho$ such that $[\delta] + [\alpha_i] \in \Phi_\rho$ and $[\delta] - [\alpha_i] \not \in \Phi_\rho$ (see 3.6 of \cite{EA2}). Let $\Phi'$ be the subsystem generated by $[\alpha_i]$ and $[\delta]$.
    \begin{enumerate}
        \item Suppose $\Phi'$ is of type $A_2$. In this case, the pair $[\alpha_i]$ and $[\delta]$ is either of the type $(b-i)$ or of the type $(b-ii)$. Take
        \[
            H \ni [x_{[\alpha_i]}(1), x] = \begin{cases}
                x_{[\delta] + [\alpha_i]} (\pm t_{[\delta]}) x' & \text{if } [\alpha_i], [\delta] \text{ is of type } (b-i), \\
                x_{[\delta] + [\alpha_i]} (\pm t_{[\delta]}) x' \text{ or } x_{[\delta] + [\alpha_i]} (\pm \bar{t}_{[\delta]}) x' & \text{if } [\alpha_i], [\delta] \text{ is of type } (b-ii);
            \end{cases}
        \]
        where $x'$ is a product of elements $x_{[\alpha]}(t_{[\alpha]})$ with $[\alpha] \neq [\delta] + [\alpha_i]$ and $ht ([\alpha]) > ht ([\delta])$. Hence, by $(P_n)$, we have $x_{[\delta]+[\alpha_i]} (\pm t_{[\delta]})$ or $x_{[\delta]+[\alpha_i]} (\pm \bar{t}_{[\delta]}) \in H$. But then, by Proposition \ref{z to Rz}, $x_{[\delta]} (t_{[\delta]}) \in H$.

        \item Suppose $\Phi'$ is of type $B_2$ and $[\delta]$ is a short root. In this case, the pair $[\alpha_i]$ and $[\delta]$ is of the type $(d-i)$ if $\Phi_\rho \sim {}^2 A_{2n-1}, {}^2 D_{n+1}$ or ${}^2 E_6$ and of the type $(d-ii)$ if $\Phi_\rho \sim {}^2A_{2n}$ (with $[\alpha_i]$ being the long root). Take 
        \[
            H \ni [x_{[\alpha_i]}(1), x] = \begin{cases}
                x_{[\delta] + [\alpha_i]} (t_{[\delta]}) x' & \text{if } {}^2 A_{2n-1}, {}^2 D_{n+1} \text{ or } {}^2 E_6, \\
                x_{[\delta] + [\alpha_i]} (s_{[\delta]}) x' & \text{if } {}^2 A_{2n};
            \end{cases}
        \]
        where $x'$ is a product of elements $x_{[\alpha]}(t_{[\alpha]})$ with $[\alpha] \neq [\delta] + [\alpha_i]$, $ht ([\alpha]) > ht ([\delta])$ and $s_{[\delta]} = (\pm t_1, t_2)$ or $(\pm \bar{t_1}, \bar{t_2})$ if $t_{[\delta]}= (t_1, t_2).$ Hence, by $(P_n)$, we have $x_{[\delta]+[\alpha_i]} (t_{[\delta]})$ or $x_{[\delta]+[\alpha_i]} (s_{[\delta]}) \in H$. But then, by Proposition \ref{z to Rz}, $x_{[\delta]} (t_{[\delta]}) \in H$. 

        \item Suppose $\Phi'$ is of type $B_2$ and $[\delta_i]$ is a long root.  In this case, the pair $[\delta]$ and $[\alpha_i]$ is of the type $(d-i)$ if $\Phi_\rho \sim {}^2 A_{2n-1}, {}^2 D_{n+1}$ or ${}^2 E_6$ and of the type $(d-ii)$ if $\Phi_\rho \sim {}^2A_{2n}$. Take
        \[
            H \ni \begin{cases}
                [x_{[\alpha_i]}(1), x] = x_{[\delta] + [\alpha_i]} (\pm t_{[\delta]}) x' & \text{if } {}^2 A_{2n-1}, {}^2 D_{n+1} \text{ or } {}^2 E_6, \\
                [x_{[\alpha_i]}(1, 1/2), x] = x_{[\delta] + [\alpha_i]} (\pm t'_{[\delta]}, t_{[\delta]} \bar{t}_{[\delta]}/2) x' & \text{if } {}^2 A_{2n};
            \end{cases}
        \]
        where $x'$ is a product of elements $x_{[\alpha]}(t_{[\alpha]})$ with $[\alpha] \neq [\delta] + [\alpha_i]$, $ht ([\alpha]) > ht ([\delta])$ and $t'_{[\delta]} = t_{[\delta]}$ or $\bar{t}_{[\delta]}$. Hence, by $(P_n)$, we have $x_{[\delta]+[\alpha_i]} (\pm t_{[\delta]})$ or $x_{[\delta] + [\alpha_i]} (\pm t'_{[\delta]}, t_{[\delta]} \bar{t}_{[\delta]}/2) \in H$. But then, by Proposition \ref{z to Rz}, $x_{[\delta]} (t_{[\delta]}) \in H$.
    \end{enumerate}
    This proves the lemma.
\end{proof}

%%%%%%%%%%%%%%%%%%%%%%%%%%%%%%%%%%%%%%%%%%%%%%%%%%

We labelled the simple roots $[\alpha_1], [\alpha_2], \dots, [\alpha_l]$ from one end of the Dynkin diagram to the other end such that 
\[
    [\alpha_1] \sim \begin{cases}
        A_1^2 & \text{if } \Phi_\rho \sim {}^2A_{n} \ (n \geq 3); \\
        A_1 & \text{if } \Phi_\rho \sim {}^2D_{n} \ (n \geq 4), {}^3D_{4} \text{ or } {}^2E_{6}. \\
        
    \end{cases} 
\]
Let $[\beta]$ be the highest root in $\Phi_\rho$. Note that there is a unique simple root $[\gamma] \in \Delta_\rho$ such that $\langle [\beta], [\gamma] \rangle \neq 0.$ 
The following table give us the precious values of $[\beta]$ and $[\gamma]$:
\begin{center}
    \begin{tabular}{c|c|c}
        \textbf{Type of $\Phi_\rho$} & \textbf{$[\beta]$} & \textbf{$[\gamma]$} \\
        \hline
        ${}^2 A_{2n-1} \ (n \geq 2)$ & $2 [\alpha_1] + 2 [\alpha_2] + \dots + 2 [\alpha_{n-1}] + [\alpha_n]$ & $[\alpha_1]$ \\
        ${}^2 A_{2n} \ (n \geq 2)$ & $[\alpha_1] + 2[\alpha_2] + 2 [\alpha_3] + \dots + 2[\alpha_{n}]$ & $[\alpha_2]$ \\
        ${}^2 D_{n} \ (n \geq 4)$ & $[\alpha_1] + 2[\alpha_2] + 2 [\alpha_3] + \dots + 2[\alpha_{n-1}]$ & $[\alpha_2]$ \\
        ${}^3 D_{4}$ & $2[\alpha_1] + 3[\alpha_2]$ & $[\alpha_1]$ \\
        ${}^2 E_{6}$ & $2 [\alpha_1] + 3 [\alpha_2] + 4 [\alpha_3] + 2 [\alpha_4]$ & $[\alpha_1]$ \\
    \end{tabular}
\end{center}

%%%%%%%%%%%%%%%%%%%%%%%%%%%%%%%%%%%%%%%%%%%%%%%%%%

\begin{lemma}\label{lemma:UTV cap H}
    \normalfont
    Let $[\beta]$ and $[\gamma]$ be as above. If $z = x_{[\gamma]}(t) xhy \in U_\sigma (R) T_\sigma (R) U^{-}_\sigma (R) \cap H,$ where $x_{[\gamma]}(t) x \in U_\sigma (R),$ $x$ is a product of elements $x_{[\alpha]}(t_{[\alpha]})$ with $[\alpha] \neq [\gamma], [\alpha] \in \Phi^{+}_\rho$ and $t_{[\alpha]} \in R_{[\alpha]}, h \in T_\sigma (R)$ and $y \in U^{-}_\sigma (R)$. Then $x_{[\gamma]}(t) \in H$.
\end{lemma}

\begin{proof}
    We write ${}^a b$ for the conjugate $a b a^{-1}$. 
    
    \vspace{2mm}
    
    \noindent \textbf{Case A. The rank of $\Phi_\rho > 2$:} 
    Let 
    \begin{align*}
        H \ni z_1 &= [x_{-[\gamma]}(1), z] \\
        &= [x_{-[\gamma]} (1), x_{[\gamma]}(t)] \{ {}^{x_{[\gamma]}(t)} [x_{-[\gamma]}(1), x] \} \{ {}^{x_{[\gamma]}(t) x} [x_{-[\gamma]}(1), h] \} \{ {}^{x_{[\gamma]}(t) x h} [x_{-[\gamma]}(1), y] \} \\
        &= {}^{x_{[\gamma]}(t) x} \{ \{ {}^{x^{-1} x_{[\gamma]}(t)^{-1}} [x_{-[\gamma]} (1), x_{[\gamma]}(t)] \} \{ {}^{x^{-1}} [x_{-[\gamma]}(1), x] \} \{ [x_{-[\gamma]}(1), h] \} \{ {}^{h} [x_{-[\gamma]}(1), y] \} \}.
    \end{align*}
    Note that, ${}^{x^{-1} x_{[\gamma]}(t)^{-1}} [x_{-[\gamma]} (1), x_{[\gamma]}(t)] = {}^{x^{-1}} [x_{[\gamma]} (t)^{-1}, x_{-[\gamma]}(1)] = [x_{[\gamma]} (t)^{-1}, x_{-[\gamma]}(1)] x'$ with $x' \in U_{\sigma}(R),$ $[x_{-[\gamma]}(1),h] = x_{-[\gamma]}(a)$ for some $a \in R_{[\gamma]}$ and $x^{-1} [x_{-[\gamma]} (1), x] = [x^{-1}, x_{-[\gamma]} (1)] \in U_\sigma (R)$. Set $u_1 = [x_{[\gamma]}(t)^{-1}, x_{-[\gamma]} (1)], x_1= x' [x^{-1}, x_{-[\gamma]} (1)]$ and $y_1 = x_{-[\gamma]}(a) \{ {}^{h} [x_{-[\gamma]}(1), y] \}$. 
    Thus we have $z_1 = {}^{x_{[\gamma]}(t) x}(u_1 x_1 y_1)$ and since $z_1 \in H$ then so is $u_1 x_1 y_1$. 
    Observe that $x_1 \in U_\sigma (R)$ is a product of elements $x_{[\alpha]}(t_{[\alpha]})$ with $[\alpha] \neq [\gamma]$, and $y_1 \in U^{-}_\sigma (R)$ is a product of elements $x_{-[\alpha]}(s_{[\alpha]})$ with $m_{[\gamma]}([\alpha]) \geq 1$. 
    In this case, there exists a root $[\delta] \in \Delta_\rho$ such that $\{[\gamma], [\delta]\}$ is a base of a subsystem of $\Phi_\rho$ of type $A_2$ (note that $[\gamma] \sim A_1$ or $A_1^2$ then so is $[\delta]$). Now let
    \begin{align*}
        H \ni z_2 &= [x_{[\delta]}(1), u_1 x_1 y_1] \\
        &= \{ [x_{[\delta]}(1), u_1] \} \{ {}^{u_1} [x_{[\delta]}(1), x_1] \} \{ {}^{u_1 x_1} [x_{[\delta]}(1), y_1] \} \\
        &= {}^{u_1 x_1} \{ \{ {}^{x_1^{-1} u_1^{-1}} [x_{[\delta]}(1), u_1] \} \{ {}^{x_1^{-1}} [x_{[\delta]}(1), x_1] \} \{ [x_{[\delta]}(1), y_1] \} \}.
    \end{align*}
    Note that, 
    \begin{align*}
        {}^{x_1^{-1} u_1^{-1}} [x_{[\delta]}(1), u_1] &= {}^{x_1^{-1}} [u_1^{-1}, x_{[\delta]}(1)] \\
        &= x_1^{-1} (u_1^{-1} x_{[\delta]}(1) u_1 x_{[\delta]}(1)^{-1}) x_1 \\
        &= {}^{x_1^{-1} x_{-[\gamma]}(1)} \{ x_{[\delta]} (\pm t') x_{[\gamma] + [\delta]} (\pm t'^{2}) \},
    \end{align*}
    where $t' = t$ or $\bar{t}$. Observe that, there exists $x'_1 \in U_\sigma (R)$ such that $x_1^{-1} x_{-[\gamma]}(1) = x_{-[\gamma]}(1) x'_1$. Then,
    \begin{align*}
        z_2 &= {}^{u_1 x_1} \{ \{ {}^{x_1^{-1} u_1^{-1}} [x_{[\delta]}(1), u_1] \} \{ {}^{x_1^{-1}} [x_{[\delta]}(1), x_1] \} \{ [x_{[\delta]}(1), y_1] \} \} \\
        &= {}^{u_1 x_1 x_{-[\gamma]}(1)} \{ \{ {}^{x'_1} (x_{[\delta]}(\pm t') x_{[\gamma]+[\delta]}(\pm t'^2)) \} \{ {}^{x_{-[\gamma]}(1)^{-1} x_1^{-1}} [x_{[\delta]}(1), x_1] \} \{ {}^{x_{-[\gamma]}(1)^{-1}} [x_{[\delta]}(1), y_1] \} \} \\
        &= {}^{u_1 x_1 x_{-[\gamma]}(1)} \{ \{ x_{[\delta]}(\pm t') x_{[\gamma]+[\delta]}(\pm t'^2) x''_1 \} \{ {}^{x_{-[\gamma]}(1)^{-1} x_1^{-1}} [x_{[\delta]}(1), x_1] \} \{ {}^{x_{-[\gamma]}(1)^{-1}} [x_{[\delta]}(1), y_1] \} \},
    \end{align*}
    where $x''_1 \in U_\sigma (R)$. Set $x_2 = x''_1 \{ {}^{x_{-[\gamma]}(1)^{-1} x_1^{-1}} [x_{[\delta]}(1), x_1] \}$ and $y_2 = {}^{x_{-[\gamma]}(1)^{-1}} [x_{[\delta]}(1), y_1]$. Note that $x_2 \in U_\sigma (R)$ is a product of elements $x_{[\alpha]}(t_{[\alpha]})$ with $[\alpha] \in \Phi^{+}_\rho, [\alpha] \neq [\gamma], [\delta], [\gamma]+[\delta]$ and $y_2  \in U^{-}_\sigma (R)$ is a product of elements $x_{-[\alpha]}(s_{[\alpha]})$ with $[\alpha] \in \Phi^{+}_\rho$ and $m_{[\gamma]} ([\alpha]) \geq 1$. But then
    $$z_2 = [x_{[\delta]}(1), u_1 x_1 y_1] = {}^{u_1 x_1 x_{-[\gamma]}(1)} \{ x_{[\delta]}(\pm t') x_{[\gamma]+[\delta]}(\pm t'^2) x_2 y_2 \}.$$
    Since $z_2 \in H$ then so is $x_{[\delta]}(\pm t') x_{[\gamma]+[\delta]}(\pm t'^2) x_2 y_2$. 
    Let
    \begin{align*}
        H \ni z_3 &= [x_{-[\gamma]}(1), x_{[\delta]}(\pm t') x_{[\gamma]+[\delta]}(\pm t'^2) x_2 y_2] \\
        &= [x_{-[\gamma]}(1), x_{[\delta]}(\pm t')] \{ {}^{x_{[\delta]}(\pm t')} [x_{-[\gamma]}(1), x_{[\gamma]+[\delta]}(\pm t'^2)] \} \{ {}^{x_{[\delta]}(\pm t') x_{[\gamma]+[\delta]}(\pm t'^2)} [x_{-[\gamma]}(1), x_2] \} \\
        & \hspace{20mm} \{ {}^{x_{[\delta]}(\pm t') x_{[\gamma]+[\delta]}(\pm t'^2) x_2} [x_{-[\gamma]}(1), y_2] \} 
    \end{align*}
    Note that $[x_{-[\gamma]}(1), x_{[\delta]}(\pm t')]=1$ and $y_3 := [x_{-[\gamma]}(1), y_2] = 1$. The formal assertion is clear. To see the latter assertion, observe that the highest root $[\beta]$ is the only root with $m_{[\gamma]}([\beta]) = 2,$ therefore $x_{-[\beta]}(u)$ is the only possible factor of $y_3$. 
    Assume $y_3 = x_{-[\beta]}(u) \neq 1,$ then $-[\beta] = -[\alpha] - [\gamma]$ for some $[\alpha]$ such that $x_{-[\alpha]}(v)$ is factor of $y_2$. Hence, $-[\alpha] = -[\alpha_1] + [\delta]$ or $(-[\alpha_1]+[\delta]) -[\alpha_2]$ for some $[\alpha_1], [\alpha_2]$ with $m_{[\gamma]}([\alpha_i]) \geq 1 \ (i=1,2).$ But it is impossible for the first case, as $[\alpha_1] = [\beta] - [\gamma] + [\delta]$ is never a root, and for the second case, $2 \geq m_{[\gamma]}([\alpha_1] - [\delta] + [\alpha_2]) = m_{[\gamma]} ([\beta] - [\gamma]) = 1$. Which proves the claim. Finally, we have
    \begin{align*}
        z_3 &= \{ {}^{x_{[\delta]}(\pm t')} [x_{-[\gamma]}(1), x_{[\gamma]+[\delta]}(\pm t'^2)] \} \{ {}^{x_{[\delta]}(\pm t') x_{[\gamma]+[\delta]}(\pm t'^2)} [x_{-[\gamma]}(1), x_2] \} \\
        &= {}^{x_{[\delta]}(\pm t') x_{[\gamma]+[\delta]}(\pm t'^2)} \{ {}^{ x_{[\gamma]+[\delta]}(\pm t'^2)^{-1}} [x_{-[\gamma]}(1), x_{[\gamma]+[\delta]}(\pm t'^2)] \} \{ [x_{-[\gamma]}(1), x_2] \} \\
        &= {}^{x_{[\delta]}(\pm t') x_{[\gamma]+[\delta]}(\pm t'^2)} \{ [x_{[\gamma]+[\delta]}(\pm t'^2)^{-1}, x_{-[\gamma]}(1)]  [x_{-[\gamma]}(1), x_2] \} \\
        &= {}^{x_{[\delta]}(\pm t') x_{[\gamma]+[\delta]}(\pm t'^2)} \{ x_{[\delta]} (\pm t'^2) x_3 \},
    \end{align*}
    where $x_3 = [x_{-[\gamma]}(1), x_2] \in U_\sigma (R)$. Since $z_3 \in H$, then so is $x_{[\delta]}(\pm t'^2) x_3$. Therefore, by Lemma \ref{lemma:U cap H}, we have $x_{[\delta]} (\pm t'^2) \in H$. But then, by Proposition \ref{z to Rz}, $x_{[\gamma] + [\delta]} (\pm t'^2) \in H$. Hence $x_{[\gamma] + [\delta]} (\pm t'^2)^{-1} \{ x_{[\delta]}(\pm t') x_{[\gamma]+[\delta]}(\pm t'^2) x_2 y_2 \} = x_{[\delta]}(\pm t') x_2 y_2 \in H$. Now let 
    \begin{align*}
        H \ni z_4 &= [x_{-[\gamma] - [\delta]} (1), x_{[\delta]}(\pm t') x_2 y_2] \\
        &= [x_{-[\gamma] - [\delta]} (1), x_{[\delta]}(\pm t')] \{ {}^{x_{[\delta]}(\pm t')}[x_{-[\gamma] - [\delta]} (1), x_2] \} \{ {}^{x_{[\delta]}(\pm t') x_2} [x_{-[\gamma] - [\delta]} (1), y_2] \} \\
        &= x_{-[\gamma]}(\pm t') x_4 y_4,
    \end{align*}
    where $x_4 = {}^{x_{[\delta]}(\pm t')}[x_{-[\gamma] - [\delta]} (1), x_2]$ and $y_4 = {}^{x_{[\delta]}(\pm t') x_2} [x_{-[\gamma] - [\delta]} (1), y_2]$. Since $x_2$ does not have factors of the form $x_{[\alpha]}(t_{[\alpha]})$ with $[\alpha] = [\gamma], [\delta], [\gamma]+[\delta]$, we conclude that $x_4 \in U_\sigma(R)$ and it does not have a factor of the form $x_{[\delta]} (s_{[\delta]})$. Now we claim that $[x_{-[\gamma] - [\delta]} (1), y_2] = 1$ or $x_{-[\beta]}(s)$ for some $s \in R_{[\beta]}$. The latter case is possible only when $\Phi_\rho \sim {}^2 A_{2n+1}$. To see this, observe that the highest root $[\beta]$ is the only root with $m_{[\gamma]}([\beta]) = 2$, therefore $x_{-[\beta]}(u)$ is the only possible factor of $[x_{-[\gamma] - [\delta]} (1), y_2]$. Assume $\Phi_\rho \not\sim {}^{2} A_{2n+1}$ and $[x_{-[\gamma] - [\delta]} (1), y_2] = x_{-[\beta]}(s) \neq 1$, then $-[\beta] = -[\alpha] - [\gamma] - [\delta]$ for some root $[\alpha]$ such that $x_{-[\alpha]}(v)$ is a factor of $y_2$. Hence, $-[\alpha] = - [\alpha_1] + [\delta]$ or $-[\alpha] = (-[\alpha_1] + [\delta]) - [\alpha_2]$ for some $[\alpha_1], [\alpha_2]$ with $m_{[\gamma]}([\alpha_i]) \geq 1 (i=1,2).$ But it is impossible for the first case, as in this case $[\alpha_1] = [\beta] - 2[\gamma]$ is never a root, and for the second case, $2 \geq m_{[\gamma]}([\alpha_1] - [\delta] + [\alpha_2]) = m_{[\gamma]}([\beta] - [\gamma] - [\delta]) = 1$. Which proves the claim. 
    
    Suppose $\Phi_\rho \not\sim {}^2 A_{2n+1}$ then $z_4 = x_{-[\gamma]}(\pm t') x_4$. In this case, let
    \begin{align*}
        H \ni z_5 &= [x_{[\beta]}(1), x_{-[\gamma]} (\pm t') x_4] \\
        &= [x_{[\beta]}(1), x_{-[\gamma]}(\pm t')] \{ {}^{x_{-[\gamma]}(\pm t')} [x_{[\beta]}(1), x_4] \} \\
        &= x_{[\beta] - [\gamma]}(\pm t').
    \end{align*}
    Thus, by Proposition \ref{z to Rz}, we have $x_{[\gamma]}(t) \in H$. 
    
    Now suppose $\Phi_\rho \sim {}^2 A_{2n+1}$, then $ [x_{-[\gamma] - [\delta]}(1), y_2] = x_{-[\beta]}(s)$ for some $s \in R_{[\beta]}$. We can rewrite the expression of $z_4$ as follows:
    \begin{align*}
        z_4 &= [x_{-[\gamma] - [\delta]} (1), x_{[\delta]}(\pm t')] \{ {}^{x_{[\delta]}(\pm t')}[x_{-[\gamma] - [\delta]} (1), x_2] \} \{ {}^{x_{[\delta]}(\pm t') x_2} [x_{-[\gamma] - [\delta]} (1), y_2] \} \\
        &= {}^{x_{[\delta]}(\pm t') x_2} \{ \{ {}^{x_2^{-1} x_{[\delta]}(\pm t')^{-1}} [x_{-[\gamma] - [\delta]} (1), x_{[\delta]}(\pm t')] \} \{ {}^{x_2^{-1}} [x_{-[\gamma] - [\delta]} (1), x_2] \} \{ x_{-[\beta]}(s) \} \} \\
        &= {}^{x_{[\delta]}(\pm t') x_2} \{ \{ {}^{x_2^{-1}} [x_{[\delta]}(\pm t')^{-1}, x_{-[\gamma] - [\delta]} (1)] \} \{ [x_2^{-1}, x_{-[\gamma] - [\delta]} (1)] \} \{ x_{-[\beta]}(s) \} \} \\
        &= {}^{x_{[\delta]}(\pm t') x_2} \{ \{ {}^{x_2^{-1}} x_{-[\gamma]}(\pm t') \} \{ [x_2^{-1}, x_{-[\gamma] - [\delta]} (1)] \} \{ x_{-[\beta]}(s) \} \}
    \end{align*}
    Let $x'_2$ be such that ${}^{x_2^{-1}} x_{-[\gamma]}(\pm t') = x_{-[\gamma]}(\pm t') x'_2$. Clearly, $x'_2 \in U_\sigma (R)$. We set $x'_4 = x'_2 [x_2^{-1}, x_{-[\gamma] - [\delta]}(1)]$. Note that $x'_4 \in U_\sigma (R)$ such that it can not contain a factor of type $x_{[\delta]}(s_{[\delta]})$. But then 
    $$ z_4 = {}^{x_{[\delta]}(\pm t') x_2} \{ x_{-[\gamma]} (\pm t') x'_4 x_{-[\beta]}(s) \}.$$
    Since $z_4 \in H$, so is $x_{-[\gamma]} (\pm t')^{-1} x'_4 x_{-[\beta]}(s)$. Now, let
    \begin{align*}
        H \ni z'_5 &= [x_{-[\delta]}(1), x_{-[\gamma]} (\pm t') x'_4 x_{-[\beta]}(s)] \\
        &= [x_{-[\delta]}(1), x_{-[\gamma]} (\pm t')] \{ {}^{x_{-[\gamma]} (\pm t')} [x_{-[\delta]}(1), x'_4] \} \{ {}^{x_{-[\gamma]} (\pm t') x'_4} [x_{-[\delta]}(1), x_{-[\beta]}(s)] \} \\
        &= {}^{x_{-[\gamma]}(\pm t')} \{ \{{}^{x_{-[\gamma]}(\pm t')^{-1}} [x_{-[\delta]}(1), x_{-[\gamma]} (\pm t')] \} \{ [x_{-[\delta]}(1), x'_4] \} \} \\
        &= {}^{x_{-[\gamma]}(\pm t')} \{ \{[x_{-[\gamma]}(\pm t')^{-1}, x_{-[\delta]}(1)] \} \{ [x_{-[\delta]}(1), x'_4] \} \} \\
        &= {}^{x_{-[\gamma]}(\pm t')} \{ \{ x_{-[\gamma] - [\delta]} (\pm t') \} \{ [x_{-[\delta]}(1), x'_4] \} \} \\
        & = {}^{x_{-[\gamma]}(\pm t')} \{ x_{-[\gamma] - [\delta]} (\pm t') x_5 \},
    \end{align*}
    where $x_5 = [x_{-[\delta]}(1), x'_4] \in U_\sigma (R)$. Note that $x_{-[\gamma] - [\delta]} (\pm t') x_5 \in H$. Finally, let
    \begin{align*}
        H \ni z_6 &= [x_{[\beta]}(1), x_{-[\gamma] - [\delta]} (\pm t') x_5] \\
        &= [x_{[\beta]}(1), x_{-[\gamma] - [\delta]} (\pm t')] \{ {}^{x_{-[\gamma] - [\delta]} (\pm t')} [x_{[\beta]}(1), x_5] \} \\
        &= [x_{[\beta]}(1), x_{-[\gamma] - [\delta]} (\pm t')] \\
        &= x_{[\beta] - [\gamma] - [\delta]} (\pm t') x_{[\beta] - 2 [\gamma] - 2 [\delta]} (\pm t \bar{t})
    \end{align*}
    Now, by Lemma \ref{lemma:U cap H}, we have $x_{[\beta] - [\gamma] - [\delta]} (\pm t') \in H$ and hence, by Proposition \ref{z to Rz}, $x_{[\gamma]} (t) \in H$.

    \vspace{2mm}
    
    \noindent \textbf{Case B. $\Phi_\rho \sim {}^2 A_3$:} 
    Let $[\delta] \in \Delta_\rho$ be such that $\{[\gamma], [\delta]\}$ forms a base of $\Phi_\rho$. Note that $[\gamma]$ is a short root. The idea of the proof is the same as in Case A. We leave the details to the reader.
    Let $z = x_{[\gamma]}(t) x h y = x_{[\gamma]}(t) x h x_{-[\gamma]-[\delta]}(s_1) x_{-[\gamma]}(s_2) x_{-2[\gamma]-[\delta]}(s_3) x_{-[\delta]}(s_4) \in H$.
    We first show that $x_{-[\delta]}(s_4) \in H$. Write $x' = x^{-1} x_{[\gamma]}(t)^{-1}$ and $y' = \{ x_{-[\gamma]-[\delta]}(s_1) x_{-[\gamma]}(s_2) x_{-2[\gamma]-[\delta]}(s_3) \}^{-1}$ and consider 
    $$ H \ni z_1 = [x_{[\delta]}(1), z^{-1}] = {}^{x_{-[\delta]}(-s_4)y'} \{ u_1 y_1 x_1 \}, $$
    where $u_1 = [x_{-[\delta]}(s_4), x_{[\delta]}(1)], y_1 = x_{-[\gamma]}(s'_1) x_{-[\gamma]-[\delta]}(s'_2) x_{-2[\gamma]-[\delta]}(s'_3) \in U^{-}_\sigma (R)$ and $x_1 = x_{[\delta]}(t'_1) x_{[\gamma]+[\delta]} (t'_2) x_{2[\gamma] + [\delta]} (t'_3) \in U_\sigma (R)$. Now consider
    $$H \ni z_2 = [x_{-[\gamma]}(1), u_1 y_1 x_1] = {}^{u_1 y_1 x_{[\delta]}(1)} \{ x_{-[\gamma]}(\pm s_4) y_2 x_2 \},$$
    where $y_2 = x_{-[\gamma]-[\delta]}(s''_1) x_{-2[\gamma]-[\delta]}(s''_2) \in U^{-}_\sigma (R)$ and $x_2 = x_{[\delta]}(t''_1) x_{[\gamma]+[\delta]} (t''_2) \in U_\sigma (R)$. Now let $z'_2 = x_{-[\gamma]}(\pm s_4) y_2 x_2$ and consider
    $$ H \ni z_3 = [x_{[\delta]}(1), z'_2] = x_{-[\gamma]}(\pm s''_1) x_{-2[\gamma] - [\delta]}(s'''_2).$$
    But then $$w_{2[\gamma]+[\delta]}(1) \{ x_{-[\gamma]}(\pm s''_1) x_{-2[\gamma] - [\delta]}(s'''_2) \}w_{2[\gamma]+[\delta]}(1)^{-1} = x_{[\gamma] + [\delta]}(\pm s''_1) x_{2[\gamma] + [\delta]}(\pm s'''_2) \in H.$$ 
    By Lemma \ref{lemma:U cap H}, we have $x_{[\gamma] + [\delta]}(\pm s''_1) \in H$ and hence, by Proposition \ref{z to Rz}, we $x_{-[\gamma]-[\delta]}(s''_1) \in H$. Now $x_{-[\gamma]-[\delta]}(s''_1)^{-1} (z'_2) = x_{-[\gamma]}(\pm s_4) x_{-2[\gamma]-[\delta]}(s''_2) x_{[\delta]}(t''_1) x_{[\gamma]+[\delta]} (t''_2) \in H$. But then 
    \begin{gather*}
        w_{2[\gamma]+[\delta]}(1) w_{[\gamma] + [\delta]}(1) \{ x_{-[\gamma]}(\pm s_4) x_{-2[\gamma]-[\delta]}(s''_2) x_{[\delta]}(t''_1) x_{[\gamma]+[\delta]} (t''_2) \} w_{[\gamma] + [\delta]}(1)^{-1} w_{2[\gamma]+[\delta]}(1)^{-1} \\
        = x_{[\gamma] + [\delta]}(\pm s_4) x_{[\delta]}(s''_2) x_{2[\gamma] + [\delta]}(t''_1) x_{[\gamma]} (t''_2) \in H.
    \end{gather*}
    By Lemma \ref{lemma:U cap H}, we have $x_{[\gamma] + [\delta]}(\pm s_4) \in H$ and hence, by Proposition \ref{z to Rz}, we $x_{-[\delta]}(s_4) \in H$.
    Finally, $z' = z x_{-[\delta]}(s_4)^{-1} \in H$. Now consider, 
    $$ H \ni z_4 = [x_{[\delta]}(1), z'] = {}^{x_{[\gamma]}(t) x} \{ x_{[\gamma]+[\delta]} (\pm t) x_4 y_4 \}, $$
    where $x_4 =  x_{[\delta]}(t_{41}) x_{2[\gamma]+[\delta]}(t_{42})$ and $y_4 = x_{-[\gamma]}(s_{41}) x_{-2[\gamma]-[\delta]}(s_{42})$. Finally, we claim that $x_{-2[\gamma]-[\delta]}(s_{42}) \in H$. To see this consider
    $$ H \ni z_5 = [x_{[\gamma]+[\delta]}(1), x_{[\gamma]+[\delta]} (\pm t) x_4 y_4] = {}^{x_{[\gamma]+[\delta]} (\pm t) x_4 x_{-[\gamma]}(s_{41})} \{ x_{-[\gamma]}(\pm s_{42}) x_{[\delta]}(s_{52}) \}.$$
    But then $w_{[\gamma]}(1) \{ x_{-[\gamma]}(\pm s_{42}) x_{[\delta]}(s_{52}) \} w_{[\gamma]}(1)^{-1} = x_{[\gamma]}(\pm s_{42}) x_{2[\gamma] + [\delta]}(\pm s_{52}) \in H.$ By Lemma \ref{lemma:U cap H}, we have $x_{[\gamma]}(\pm s_{42}) \in H$, and hence, by Proposition \ref{z to Rz}, we have $x_{-2[\gamma]-[\delta]}(s_{42}) \in H$. Finally, we have $z'' = x_{[\gamma]+[\delta]} (\pm t) x_4 y_4 x_{[\gamma]}(\pm s_{42})^{-1} \in H$. But then 
    $$w_{[\gamma]}(1) z'' w_{[\gamma]}(1)^{-1} = x_{[\gamma]+[\delta]} (\pm t) x_{2[\gamma] + [\delta]} (\pm t_{41}) x_{[\delta]}(\pm t_{42}) x_{[\gamma]}(\pm s_{41}) \in H.$$ By Lemma \ref{lemma:U cap H}, $x_{[\gamma] + [\delta]} (t) \in H$ and hence, by Proposition \ref{z to Rz}, $x_{[\gamma]}(t) \in H$.
    
    \vspace{2mm}

    \noindent \textbf{Case C. $\Phi_\rho \sim {}^2 A_4$:} 
    Let $[\delta] \in \Delta_\rho$ be such that $\{[\gamma], [\delta]\}$ forms a base of $\Phi_\rho$. Note that $[\gamma]$ is a long root. The idea of the proof is the same as in Case A. We leave the details to the reader.
    We first consider 
    $$ z_1 = [x_{-[\gamma]}(1), z] = {}^{x_{[\gamma]}(t) x}(u_1 x_1 y_1),$$
    where $u_1 = [x_{[\gamma]}(t)^{-1}, x_{-[\gamma]}(1)],$ $x_1 = x_{[\delta]}(t_1) x_{[\gamma]+[\delta]} (t_2) x_{[\gamma] + 2[\delta]} (t_3)$ and 
    $y_1 = x_{-[\gamma]}(s_1)$ $x_{-[\gamma] - [\delta]}(s_2)$ $x_{-[\gamma] - 2[\delta]} (s_3)$.
    We than consider  
    \begin{align*}
        H \ni z_2 &= [x_{[\delta]}(1, 1/2), u_1 x_1 y_1] \\
        &= {}^{u_1 x_1 x_{-[\gamma]}(1)} \{ x_{[\delta]} (\pm t, (t \bar{t} \pm (t - \bar{t}))/2) x_2 y_2 \},
    \end{align*}
    where $x_2 = x_{[\gamma] + [\delta]}(t'_1) x_{[\gamma] + 2[\delta]} (t'_2)$ and $y_2 = x_{-[\gamma]}(s'_1) x_{-[\gamma] - [\delta]}(s'_2)$. Set $u_2 = x_{[\delta]} (\pm t, (t \bar{t} \pm (t - \bar{t}))/2)$ and $z'_2 = u_2 x_2 y_2$.
    Our next goal is to show that $x_{-[\gamma] - [\delta]}(s'_2) \in H$. Write $s'_2 = (a,b) \in R_{-[\gamma] - [\delta]}$. Now consider
    \begin{align*}
        H \ni z_3 &= [x_{[\delta]}(1, 1/2), z'_2] = {}^{u_2 x_2} \{ x_{[\gamma] + 2 [\delta]} (t''_1) x_{-[\gamma]} (\pm a') \}, 
    \end{align*}
    where $a'= a$ or $\bar{a}$ and $t''_1 \in R_{[\gamma] + 2 [\delta]}$. Since $z_3 \in H$ we have $x_{[\gamma] + 2 [\delta]} (t''_1) x_{-[\gamma]} (\pm a') \in H$. But then $w_{[\gamma]}(1) \{x_{[\gamma] + 2 [\delta]} (t''_1) x_{-[\gamma]} (\pm a')\} w_{[\gamma]}(1)^{-1} = x_{[\gamma] + 2 [\delta]} (t''_1) x_{[\gamma]} (\pm a') \in H$ (see Proposition \ref{prop wxw^{-1}}). By Lemma \ref{lemma:U cap H}, we have $x_{[\gamma]}(\pm a') \in H$. Next, we consider
    \begin{align*}
        H \ni z_4 &= [x_{[\gamma] + 2[\delta]}(1), z'_2] = {}^{u_2 x_2} \{ x_{[\delta]} (t''_2) x_{-[\gamma]} (\pm b') \}, 
    \end{align*}
    where $b'= b$ or $\bar{b}$ and $t''_2 \in R_{[\delta]}$. Since $z_3 \in H$ we have $x_{[\delta]} (t''_2) x_{-[\gamma]} (\pm b') \in H$. But then $w_{[\gamma]}(1) \{x_{[\delta]} (t''_2) x_{-[\gamma]} (\pm b')\} w_{[\gamma]}(1)^{-1} = x_{[\gamma] + [\delta]} (\pm t''_2) x_{[\gamma]} (\pm b') \in H$ (see Proposition \ref{prop wxw^{-1}}). By Lemma \ref{lemma:U cap H}, we have $x_{[\gamma]}(\pm b') \in H$. Since $x_{[\gamma]}(\pm a') \in H$ and $x_{[\gamma]}(\pm b') \in H$, by Proposition \ref{z to Rz}, we have $x_{-[\gamma]-[\delta]}(a,b) \in H$. But then 
    $$ z_2 x_{-[\gamma] - [\delta]}(a,b)^{-1} = x_{[\delta]}(\pm t, t \bar{t} \pm (t - \bar{t}))/2) x_{[\gamma] + [\delta]}(t'_1) x_{[\gamma] + 2[\delta]} (t'_2) x_{-[\gamma]}(s'_1) \in H.$$ 
    Finally, 
    \begin{gather*}
        w_{[\gamma]}(1) \{ x_{[\delta]}(\pm t, t \bar{t} \pm (t - \bar{t}))/2) x_{[\gamma] + [\delta]}(t'_1) x_{[\gamma] + 2[\delta]} (t'_2) x_{-[\gamma]}(s'_1) \} w_{[\gamma]}(1)^{-1} \\
        = x_{[\gamma] + [\delta]}(\pm t', t \bar{t} \pm (t - \bar{t}))/2) x_{[\delta]}(t'_{11}) x_{[\gamma] + 2[\delta]} (t'_{22}) x_{[\gamma]}(s'_{11}) \in H,
    \end{gather*}
    where $t' = t$ or $\bar{t}$, and $t'_{11} \in R_{[\delta]}, t'_{22} \in R_{[\gamma] + 2[\delta]}, s'_{11} \in R_{[\gamma]}$. Finally, by Lemma \ref{lemma:U cap H}, we have $x_{[\gamma] + [\delta]}(\pm t', t \bar{t} \pm (t - \bar{t}))/2) \in H$, and by Proposition \ref{z to Rz}, we have $x_{[\gamma]}(t) \in H$, as desired.
    
    \vspace{2mm}

    \noindent \textbf{Case D. $\Phi_\rho \sim {}^3 D_4$:} Let $[\delta] \in \Delta_\rho$ be such that $\{[\gamma], [\delta]\}$ forms a base of $\Phi_\rho$. Note that $[\gamma]$ is a long root. The idea of the proof is the same as in Case A. We leave the details to the reader.
    We first consider 
    $$ z_1 = [x_{-[\gamma]}(1), z] = {}^{x_{[\gamma]}(t) x}(u_1 x_1 y_1),$$
    where $u_1 = [x_{[\gamma]}(t)^{-1}, x_{-[\gamma]}(1)], x_1 = x_{[\delta]}(t_1) x_{[\gamma]+[\delta]} (t_2) x_{[\gamma] + 2[\delta]} (t_3) x_{[\gamma] + 3[\delta]} (t_4) x_{2[\gamma] + 3[\delta]} (t_5)$ and 
    $y_1 = x_{-[\gamma]}(s_1) x_{-[\gamma] - [\delta]}(s_2) x_{-[\gamma] - 2[\delta]} (s_3) x_{-[\gamma] - 3[\delta]} (s_4) x_{-2[\gamma] - 3[\delta]} (s_5)$.
    We than consider 
    \begin{align*}
        H \ni z_2 &= [x_{[\delta]}(1), u_1 x_1 y_1] \\
        &= {}^{u_1 x_1 x_{-[\gamma]}(1)} \{ x_{[\delta]} (\pm t) x_2 y_2 \},
    \end{align*}
    where $x_2 = x_{[\gamma] + [\delta]}(t'_1) x_{[\gamma] + 2[\delta]} (t'_2) x_{[\gamma] + 3[\delta]}(t'_3) x_{2[\gamma] + 3[\delta]} (t'_4)$ and $y_2 = x_{-[\gamma]}(s'_1)$ $x_{-[\gamma] - [\delta]}(s'_2)$ $x_{-[\gamma]-2[\delta]}(s'_3) x_{-2[\gamma] - 3[\delta]}(s'_4)$. Set $z'_2 = x_{[\delta]}(\pm t) x_2 y_2$.
    Our next goal is to show that $x_{-[\gamma] -2[\delta]}(s'_3) \in H$. Now consider 
    \begin{align*}
        H \ni z_3 &= [x_{[\gamma] + 3 [\delta]}(1), z'_2] \\
        &= {}^{x_{[\delta]}(\pm t) x_2 x_{-[\gamma]}(s'_1) x_{-[\gamma]-[\delta]}(s'_2) x_{-2[\gamma] - 3 [\delta]}(s'_4) } \{ x_{[\delta]} (\pm s'_3) x_{-[\gamma]} (\pm s'_4 \pm s'_3 \bar{s'_3} \bar{\bar{s'_3}}) \\
        & \hspace{20mm} x_{-[\gamma] - [\delta]} (\pm s'_3 \bar{s'_3}) x_{-2[\gamma] - 3[\beta]} (\pm s'_3 \bar{s'_3} \bar{\bar{s'_3}}) \}. 
    \end{align*}
    Since $z_3 \in H$ we have 
    $$x_{[\delta]} (\pm s'_3) x_{-[\gamma]} (\pm s'_4 \pm s'_3 \bar{s'_3} \bar{\bar{s'_3}}) x_{-[\gamma] - [\delta]} (\pm s'_3 \bar{s'_3}) x_{-2[\gamma] - 3[\beta]} (\pm s'_3 \bar{s'_3} \bar{\bar{s'_3}}) \in H.$$ 
    But then 
    \begin{gather*}
        w_{2[\gamma]+3[\delta]}(1) \{ x_{[\delta]} (\pm s'_3) x_{-[\gamma]} (\pm s'_4 \pm s'_3 \bar{s'_3} \bar{\bar{s'_3}}) x_{-[\gamma] - [\delta]} (\pm s'_3 \bar{s'_3}) x_{-2[\gamma] - 3[\beta]} (\pm s'_3 \bar{s'_3} \bar{\bar{s'_3}}) \} w_{2[\gamma]+3[\delta]}(1)^{-1} \\
        = x_{[\delta]} (\pm s'_3) x_{[\gamma] + 3[\delta]} (\pm s'_4 \pm s'_3 \bar{s'_3} \bar{\bar{s'_3}}) x_{[\gamma] + 2[\delta]} (\pm s'_3 \bar{s'_3}) x_{2[\gamma] + 3[\beta]} (\pm s'_3 \bar{s'_3} \bar{\bar{s'_3}}) \in H.
    \end{gather*}
    By Lemma \ref{lemma:U cap H}, we have $x_{[\delta]}(\pm s'_3), x_{[\gamma] + 3[\delta]} (\pm s'_4 \pm s'_3 \bar{s'_3} \bar{\bar{s'_3}}) \in H$ and hence, by Proposition \ref{z to Rz}, $x_{-[\gamma] - 2[\delta]} (s'_3), x_{[\gamma] + 3[\delta]}(\pm s'_3 \bar{s'_3} \bar{\bar{s'_3}}) \in H$. But then $x_{[\gamma] + 3[\delta]} (s'_4) \in H$, and again by Proposition \ref{z to Rz}, $x_{-2[\gamma]-3[\delta]}(s'_4) \in H$. Therefore
    \begin{align*}
        z''_2 &= z'_2 x_{-2[\gamma]-3[\delta]}(s'_4)^{-1} x_{-[\gamma]-2[\delta]}(s'_3)^{-1} \\
        &= x_{[\delta]}(\pm t) x_{[\gamma] + [\delta]}(t'_1) x_{[\gamma] + 2[\delta]} (t'_2) x_{[\gamma] + 3[\delta]}(t'_3) x_{2[\gamma] + 3[\delta]} (t'_4) x_{-[\gamma]}(s'_1) x_{-[\gamma] - [\delta]}(s'_2) \in H.
    \end{align*}
    Now replacing $z'_2$ by $w_{2[\gamma]+3[\delta]}(1) z'_2 w_{2[\gamma]+3[\delta]}(1)^{-1}$ in above process we can conclude that $x_{[\gamma]+[\delta]}(t'_1) \in H$. Therefore, we have $$z'''_2 = x_{[\gamma]+[\delta]}(t'_1)^{-1} z''_2 = x_{[\delta]}(\pm t) x_{[\gamma]+2[\delta]}(t''_2) x_{[\gamma] + 3[\delta]}(t''_3) x_{2[\gamma] + 3[\delta]} (t''_4) x_{-[\gamma]}(s'_1) x_{-[\gamma] - [\delta]}(s'_2) \in H.$$ 
    Finally, 
    \begin{align*}
        z_4 &= w_{[\gamma]}(1) w_{[\gamma]+[\delta]}(1) \{ z'''_2 \} w_{[\gamma]+[\delta]}(1)^{-1} w_{[\gamma]}(1)^{-1} \\ 
        &= x_{[\gamma]+2[\delta]}(\pm t) x_{[\gamma] + [\delta]}(\pm t''_2) x_{2[\gamma] + 3[\delta]}(t''_3) x_{[\gamma]} (t''_4) x_{[\gamma]+3[\delta]}(s'_1) x_{[\delta]}(s'_2) \in H,
    \end{align*}
    Thus, by Lemma \ref{lemma:U cap H}, we have $x_{[\gamma] + 2[\delta]}(\pm t) \in H$, and by Proposition \ref{z to Rz}, we have $x_{[\gamma]}(t) \in H$, as desired.
\end{proof}

\begin{rmk}
    The idea of the above proof is motivated by that of the Lemma in 3.4 of \cite{KS3}. However, that proof contains an error, which we rectify here. Specifically, the value of $s$ on page 11 is incorrect; the correct value should be $s = \pm a_1 a t^2$ instead of $s = \pm a_1 t + a_1 a t^2$. This leads to a mistake in the last paragraph of the proof of Case 1, where $x_{\beta}(\pm a_1 a t^2) \in H$ should appear instead of $x_{\beta}(\pm a_1 t \pm a_1 a t^2) \in H$. However, the corrected statement does not yield the required result.
\end{rmk}

%%%%%%%%%%%%%%%%%%%%%%%%%%%%%%%%%%%%%%%%%%%%%%%%%%

\vspace{2mm}

\noindent \textit{Proof of Proposition \ref{prop:U(R) cap H subset U(J)}.} The first part is clear from Lemma \ref{lemma:U cap H}. For the second part, let $z = xhy \in U_\sigma (rad(R)) T_\sigma (R) U^{-}_\sigma (R) \cap H$, where $x \in U_\sigma (rad(R)), h \in T_\sigma(R)$ and $y \in U^{-}_\sigma (R)$. 
% For $x= \prod_{[\alpha]>0} x_{[\alpha]}(t_{[\alpha]}) \in U_\sigma (R)$ (resp., $y = \prod_{[\alpha]>0} x_{-[\alpha]}(t_{[\alpha]}) \in U^{-}_\sigma (R)$), we define $\Phi(x)$ (resp., $\Phi(y)$) to be the set $\{ [\alpha] \mid t_{[\alpha]} \neq 0 \}$ (the product is taken over disjoint roots in any fixed order). We fix $[\beta], [\gamma], [\delta],$ etc. as in Lemma \ref{lemma:UTV cap H}.

First assume that $x = 1 = y$. Then $z = h = h(\chi) \in T_\sigma (R)$ for some $\chi$. Since $h(\chi) \in H$, for any root $\alpha \in \Phi$ we have $[x_{[\alpha]}(1), h(\chi)] = x_{[\alpha]}(1 - \chi(\alpha))$ is an element of $H$. Thus, $1 - \chi(\alpha) \in J$, that is, $\chi(\alpha) \equiv 1$ (mod $J$) for every $\alpha \in \Phi$. Therefore, by Lemma \ref{lemma on T(J)}, we have $h(\chi) \in T_\sigma (R, J)$. 

We now assume that $x \neq 1$ or $y \neq 1$. Observe that, if we prove that every factor of $x$ and $y$ is in $H$ then we are done. Furthermore, by conjugating with elements of the Weyl group, applying Lemma~\ref{inUHV}, and using the Chevalley commutator formulas, it suffices to prove the following: 
\textit{for a fixed short root (respectively, a fixed long root) \( [\alpha] \), any factor of the form \( x_{[\alpha]}(t) \) appearing in the expression of an element \( x' \) or \( y' \), where \( z' = x' h' y' \in U_\sigma(R)\, T_\sigma(R)\, U^{-}_\sigma(R) \cap H \), is contained in \( H \)}.

By the Lemma \ref{lemma:UTV cap H} and the above observation, if $x_{[\alpha]}(t_{[\alpha]})$ is a factor of $x$ or $y$ with $[\alpha]$ being the same type as $[\gamma]$ then $x_{[\alpha]}(t_{[\alpha]}) \in H$. Now it remains to show that $x_{[\alpha]}(t_{[\alpha]}) \in H$ for a root $[\alpha]$ not of the type $[\gamma]$. Note that we can assume that both $x$ and $y$ contains only factors $x_{[\alpha]}(t_{[\alpha]})$ where $[\alpha]$ is not of the type $[\gamma]$. 

Suppose $\Phi_\rho \not \sim {}^2 A_{2n+1}$, then $[\gamma]$ is a long root and hence it is of the same type as the highest root $[\beta]$. In particular, by our assumption, $x$ and $y$ does not contains a factor $x_{[\beta]}(t_{[\beta]})$ and $x_{-[\beta]}(t_{[\beta]})$, respectively. Let $x_{[\alpha_1]}(t)$ be a factor of $x$ or $y$. Then $[\alpha_1]$ is of different type then $[\gamma]$. Choose a root $[\alpha_2] \in \Phi_\rho^{+}$ such that $-[\alpha_2]$ and $[\beta]$ generates the subsystem of type $B_2$. Since $[\alpha_1]$ and $[\alpha_2]$ is of the same type, we can conjugate $z$ by an element of the Weyl group in such a way that $x_{-[\alpha_2]}(t)$ is a factor of that new element, say $z_1$. Now let $z_2 = [x_{[\beta]}(1), z_1]$. Then $z_2 \in U_\sigma(R) \cap H$ and it contains a factor $x_{[\beta] - [\alpha_2]} (t')$, where $t' = t$ or $\bar{t}$ if $[\alpha_2] \not \sim A_2$; and $t' = (t_1, t_2), (t_1, \bar{t}_2)$ or $(\bar{t}_1, t_2)$ if $[\alpha_2] \sim A_2$ and $t=(t_1, t_2)$. By Lemma \ref{lemma:U cap H}, we have $x_{[\beta] - [\alpha_2]} (t') \in H$ and hence, by Proposition \ref{z to Rz}, $x_{[\alpha_1]}(t) \in H$. 

Now suppose $\Phi_\rho \sim {}^2 A_{2n+1}$, then $[\gamma]$ is a short root. Let $x_{[\alpha_1]}(t)$ be a factor of $x$ or $y$ with $[\alpha_1]$ being a long root. Choose a root $[\alpha_2] \in \Phi_\rho$ such that $[\alpha_1]$ and $[\alpha_2]$ form a subsystem of type $B_2$. Let $z_1 = [x_{[\alpha_2]}(1), z]$. Then $z_1 \in H$ contains a factor $x_{[\alpha_1] + [\alpha_2]}(t)$ with $[\alpha_1] + [\alpha_2]$ being a short root, that is, it is of the same type as $[\gamma]$. Hence from above observation, $x_{[\alpha_1] + [\alpha_2]}(t) \in H$ and hence, by Proposition \ref{z to Rz}, $x_{[\alpha]}(t) \in H$, as desired.
\qed

%%%%%%%%%%%%%%%%%%%%%%%%%%%%%%%%%%%%%%%%%%%%%%%%%%
%Section: Proof of Propositions \ref{local:H subset G(R,J)}
%%%%%%%%%%%%%%%%%%%%%%%%%%%%%%%%%%%%%%%%%%%%%%%%%%

\section{Proof of Proposition \ref{local:H subset G(R,J)}}\label{sec:Pf of prop 3}

%%%%%%%%%%%%%%%%%%%%%%%%%%%%%%%%%%%%%%%%%%%%%%%%%%

Let the notations be as in Section \ref{sec:Pf of main thm}. Define \( k_\mathfrak{m} = R / \mathfrak{m} \cong R_{\mathfrak{m}} / \mathfrak{m} R_{\mathfrak{m}} \), and similarly for \( k_{\bar{\mathfrak{m}}} \) and \( k_{\bar{\bar{\mathfrak{m}}}} \). Clearly, \( R_{\mathfrak{m}} \cong R_{\bar{\mathfrak{m}}} \cong R_{\bar{\bar{\mathfrak{m}}}} \) and \( k_\mathfrak{m} \cong k_{\bar{\mathfrak{m}}} \cong k_{\bar{\bar{\mathfrak{m}}}} \). For $S=S_\mathfrak{m}$, we set 
    \[
        I_S = \begin{cases}
            \mathfrak{m} & \text{if } \mathfrak{m} = \bar{\mathfrak{m}}, \\
            \mathfrak{m} \cap \bar{\mathfrak{m}} & \text{if } \mathfrak{m} \neq \bar{\mathfrak{m}} \text{ and } o(\theta) = 2, \\
            \mathfrak{m} \cap \bar{\mathfrak{m}} \cap \bar{\bar{\mathfrak{m}}} & \text{if } \mathfrak{m} \neq \bar{\mathfrak{m}} \text{ and } o(\theta) = 3;
        \end{cases} 
        \text{ and }  
        k_S = \begin{cases}
            k_\mathfrak{m} & \text{if } \mathfrak{m} = \bar{\mathfrak{m}}, \\
            k_\mathfrak{m} \times k_{\bar{\mathfrak{m}}} & \text{if } \mathfrak{m} \neq \bar{\mathfrak{m}} \text{ and } o(\theta) = 2, \\
            k_\mathfrak{m} \times k_{\bar{\mathfrak{m}}} \times k_{\bar{\bar{\mathfrak{m}}}} & \text{if } \mathfrak{m} \neq \bar{\mathfrak{m}} \text{ and } o(\theta) = 3.
        \end{cases}
    \]
Then $R/I_S \cong R_S / (I_S R_S) \cong k_S$.
Note that the ring automorphism $\theta: R \longrightarrow R$ also induces an automorphism of $k_S$, which is also denoted by $\theta$.

%%%%%%%%%%%%%%%%%%%%%%%%%%%%%%%%%%%%%%%%%%%%%%%%%%

\begin{prop}\label{prop:U(RS) cap psi(H) subset U(JS)}
    \normalfont
    Let the notations be as in Section \ref{sec:Pf of main thm}.
    \begin{enumerate}[(a)]
        \item $U_\sigma (R_S) \cap \psi_\mathfrak{m}(H) \subset U_\sigma (J_S)$. 
        \item $U_\sigma (I_S R_S) T_\sigma (R_S) U^{-}_\sigma (R_S) \cap \psi_\mathfrak{m} (H) \subset U_\sigma (J_S) T_\sigma (R_S, J_S) U^{-}_\sigma (J_S).$
    \end{enumerate}
\end{prop}

At first glance, the above proposition may appear to be an immediate consequence of Proposition \ref{prop:U(R) cap H subset U(J)}, but this is not the case. The crucial point is that the subgroup $\psi_\mathfrak{m}(H)$ of $G_\sigma(R_S)$ may not be normalized by $E'_\sigma(R_S)$. However, once we prove the following lemma, the proof of the above lemma will be similar to the proof of Proposition \ref{prop:U(R) cap H subset U(J)} and hence is omitted.

\begin{lemma}
    \normalfont
    Let the notations be as in Section \ref{sec:Pf of main thm}. Let $z \in H$. 
    \begin{enumerate}[(a)]
        \item If $\psi_\mathfrak{m} (z) = \prod_{[\alpha] \in \Phi_\rho^{+}} x_{[\alpha]}(t_{[\alpha]}) \in U_\sigma (R_S)$ with $t_{[\alpha]} \in (R_S)_{[\alpha]}$, then for each $[\alpha] \in \Phi^{+}_\rho$ there exists $s_{[\alpha]} \in S_\theta$ such that $x_{[\alpha]}(s_{[\alpha]} \cdot t_{[\alpha]}) \in H$.
        \item Let $[\beta]$ and $[\gamma]$ be as in Lemma \ref{lemma:UTV cap H}. If $\psi_\mathfrak{m}(z) = x_{[\gamma]}(t) xhy \in U_\sigma (R_S) T_\sigma (R_S) U^{-}_\sigma (R_S),$ where $x_{[\gamma]}(t) x \in U_\sigma (R_S),$ $x$ is a product of elements $x_{[\alpha]}(t_{[\alpha]})$ with $[\alpha] \neq [\gamma], [\alpha] \in \Phi^{+}_\rho$ and $t_{[\alpha]} \in (R_S)_{[\alpha]}, h \in T_\sigma (R_S)$ and $y \in U^{-}_\sigma (R_S)$. Then there exists $s \in S_\theta$ such that $x_{[\gamma]}(s \cdot t) \in H$.
    \end{enumerate}
\end{lemma}

\begin{proof}
    The proofs of parts (a) and (b) follow a similar approach to Lemmas \ref{lemma:U cap H} and \ref{lemma:UTV cap H}, respectively, but also incorporate the method used in proving Proposition \ref{general:[x,g] in E(R,J)}. 
    Let $\psi''_\mathfrak{m}$ be as in the proof of Proposition \ref{general:[x,g] in E(R,J)}.
    
    \vspace{2mm}
    
    \noindent (a) Let $\psi_\mathfrak{m}(z) = \prod_{[\alpha] \in \Phi^+} x_{[\alpha]}(t_{[\alpha]}) \in U_\sigma(R_S) \subset U_\sigma(R_S[X])$, with $t_{[\alpha]} \in (R_S)_{[\alpha]} \subset (R_S[X])_{[\alpha]}$. We will first prove the result in the case where $\Phi_\rho \sim {}^2 A_3$. The other cases can be proven similarly using the same techniques (cf. Lemma \ref{lemma:U cap H}) and are therefore omitted.

    As in Lemma \ref{lemma:U cap H}, let $[\alpha]$ and $[\beta]$ be the simple roots, with $[\alpha]$ being the long root. We first claim that
    \begin{equation}\label{eq_aa}
        \begin{split}
            \text{if } \psi_\mathfrak{m}(z) = x_{[\alpha] + [\beta]}(t) x_{[\alpha] + 2[\beta]}(u) \in \psi_\mathfrak{m}(H), \text{ then } \\
            x_{[\alpha] + [\beta]}(s_1 \cdot t), x_{[\alpha] + 2[\beta]}(s_2 \cdot u) \in H \text{ for some } s_1, s_2 \in S_\theta.
        \end{split}
    \end{equation}
    Since $\psi''_\mathfrak{m} \mid_{G_\sigma(R_S)} = \psi_\mathfrak{m}$, we have $\psi_\mathfrak{m}(z) = \psi''_\mathfrak{m}(z) \in \psi''_\mathfrak{m}(H)$. For any $r \in (R_S[X])_{[\alpha]} = (R_S[X])_\theta$, we have
    \[ [x_{-[\alpha]}(r), \psi''_\mathfrak{m}(z)] = x_{[\beta]}(\pm rt) x_{[\alpha] + 2[\beta]}(\pm rt \bar{t}). \]
    Write $t = a/b$ where $a \in R$ and $b \in S$. Let $r = b \bar{b} r' X$ with $r' \in S_\theta$. Then
    \[ \psi''_\mathfrak{m}([x_{-[\alpha]}(b \bar{b} r' X), z]) = \psi''_\mathfrak{m}(x_{[\beta]} (\pm b \bar{b} r' t X) x_{[\alpha] + 2[\beta]} (\pm b \bar{b} r' t \bar{t} X)). \]
    Now let 
    \[ \epsilon(X) = [x_{-[\alpha]}(b \bar{b} r' X), z](x_{[\beta]} (\pm b \bar{b} r' t X) x_{[\alpha] + 2[\beta]} (\pm b \bar{b} r' t \bar{t} X))^{-1}. \]
    Then $\epsilon(X)$ satisfies the hypothesis of Lemma \ref{lemma:GT}, and hence there exists $s' \in S_\theta$ such that $\epsilon(s' X) = 1$. Thus we have
    \[ [x_{-[\alpha]}(b \bar{b} s' r' X), z] = x_{[\beta]} (\pm b \bar{b} s' r' t X) x_{[\alpha] + 2[\beta]} (\pm b \bar{b} s' r' t \bar{t} X). \]
    Setting $X = 1$, we get 
    \[ [x_{-[\alpha]}(b \bar{b} s' r'), z] = x_{[\beta]} (\pm b \bar{b} s' r' t) x_{[\alpha] + 2[\beta]} (\pm b \bar{b} s' r' t \bar{t}) \in H. \]
    Since $\psi''_\mathfrak{m}(z) \in \psi''_\mathfrak{m}(H)$, then so is $(\psi''_\mathfrak{m}(z))^{-1} = x_{[\alpha] + [\beta]}(-t) x_{[\alpha] + 2[\beta]}(-u)$. Therefore, we can replace $t$ and $u$ by $-t$ and $-u$ respectively, and we obtain 
    \[ x_{[\beta]} (\pm b \bar{b} s' r' (-t)) x_{[\alpha] + 2[\beta]} (\pm b \bar{b} s' r' t \bar{t}) \in H. \] 
    Thus, 
    \begin{equation*}
        \begin{split}
            x_{[\beta]}(\pm 2 b \bar{b} s' r' t) = \left\{ x_{[\beta]} (\pm b \bar{b} s' r' t) x_{[\alpha] + 2[\beta]} (\pm b \bar{b} s' r' t \bar{t}) \right\} \\
            \left\{ x_{[\beta]} (\pm b \bar{b} s' r' (-t)) x_{[\alpha] + 2[\beta]} (\pm b \bar{b} s' r' t \bar{t}) \right\}^{-1} \in H.
        \end{split}
    \end{equation*}
    Set $r' = \pm 1/2$ and let $s_1 = b \bar{b} s' \in S_\theta$, then $x_{[\beta]}(s_1 t) \in H$. By Proposition \ref{z to Rz}, we get $x_{[\alpha] + [\beta]}(s_1 t) \in H$. 
    Now, in $G_\sigma(R_S)$, for any $r_1 \in R$, we have 
    \[ [x_{-[\beta]}(r_1), \psi_\mathfrak{m}(z)] = x_{[\alpha]}(\pm (r_1 \bar{t} + \bar{r}_1 t)) x_{[\alpha] + [\beta]}(\pm r_1 u) x_{[\alpha]}(\pm r_1 \bar{r}_1 u) \in \psi_\mathfrak{m}(H). \]
    Put $r_1 = s_1$, then by the proof of Lemma \ref{z to Rz in phi'}, we have $x_{[\alpha]}(\pm (s_1 \bar{t} + \bar{s}_1 t)) \in \psi_\mathfrak{m}(H)$, and hence
    \begin{equation*}
        \begin{split}
            x_{[\alpha]}(\pm (s_1 \bar{t} + \bar{s}_1 t))^{-1} (x_{[\alpha]}(\pm (s_1 \bar{t} + \bar{s}_1 t)) x_{[\alpha] + [\beta]}(\pm s_1 u) x_{[\alpha]}(\pm s_1 \bar{s}_1 u)) \\
            = x_{[\alpha] + [\beta]}(\pm s_1 u) x_{[\alpha]}(\pm s_1 \bar{s}_1 u) \in \psi_\mathfrak{m}(H).
        \end{split}
    \end{equation*}
    Thus, 
    \begin{equation*}
        w_{[\beta]}(1) x_{[\alpha] + [\beta]}(\pm s_1 u) x_{[\alpha]}(\pm s_1 \bar{s}_1 u) w_{[\beta]}(1)^{-1} = x_{[\alpha] + [\beta]}(\pm s_1 u) x_{[\alpha] + 2[\beta]}(\pm s_1 \bar{s}_1 u) \in \psi_\mathfrak{m}(H).
    \end{equation*}
    Finally, by the above argument, we have $x_{[\alpha] + [\beta]}(s_2 u) \in H$ for some $s_2 \in S_\theta$. Therefore, by Proposition \ref{z to Rz}, we have $x_{[\alpha] + 2[\beta]}(s_2 u) \in H$. This proves (\ref{eq_aa}).

    Now let $x = x_{[\beta]}(t) x_{[\alpha]}(u) x_{[\alpha] + [\beta]}(v) x_{[\alpha] + 2[\beta]}(w) \in H$. By a similar argument as above and as in Lemma \ref{lemma:U cap H}, we can prove the desired result.

    \vspace{2mm}
    
    \noindent (b) This part can also be proven using the same reasoning as in the proofs of Lemma \ref{lemma:UTV cap H} and the methods of G. Taddei \cite{GT} (cf. Lemma \ref{lemma:GT}), similar to the approach taken for part (a). Therefore, we omit the detailed proof.
\end{proof}

\begin{rmk}
    To prove Proposition \ref{prop:U(RS) cap psi(H) subset U(JS)}, we must use not only the method of proof of Proposition \ref{prop:U(R) cap H subset U(J)}, but also the lemma by G. Taddei \cite[Lemma 3.14]{GT}, as we did in the proof of the previous lemma.
\end{rmk}

% \noindent \textbf{Outlines of the proof of part (b):}
%
% \noindent \textbf{Step 1:} Let $z \in H$ and $\psi_\mathfrak{m}(z) = xhy \in U_\sigma (I_S R_S) T_\sigma (R_S) U^{-}_\sigma (R_S)$. \\
% \noindent \textbf{Step 2:} We can assume that $x \neq 1$ or $y \neq 1$. To see this we suppose $\psi_\mathfrak{m}(z) = h$. Since $h = h(\chi) \not\in G_\sigma(R_S, J_S)$, by Lemma \ref{lemma on T(J)}, there exists $\alpha \in \Phi$ such that $\chi(\alpha) - 1 \not\in J_S$. But then there exists an element $t \in (R_S)_{[\alpha]}$ such that $[x_{[\alpha]}(t), h] = x_{[\alpha]}((\chi(\alpha) - 1) \cdot t) \not\in U_\sigma(J_S)$.
% Furthermore, there exists an element $s' \in S_\theta$ such that $t_1 := s' \cdot t \in R_{[\alpha]}$ and $t_2 := (s'(\chi(\alpha) - 1)) \cdot t \in R_{[\alpha]}$. Therefore, $\psi_\mathfrak{m}([x_{[\alpha]}(t_1), z]) = \psi_\mathfrak{m}(x_{[\alpha]}(t_2))$.
% By a similar argument as in Proposition \ref{general:[x,g] in E(R,J)}, we obtain an element $s \in S_\theta$ such that $[x_{[\alpha]}(s \cdot t_1), z] = x_{[\alpha]}(s \cdot t_2)$. Since $z \in H$, we have $x_{[\alpha]}(s \cdot t_2) \in H$ and hence $s \cdot t_2 \in J_{[\alpha]}$.
% Consequently, $\psi_\mathfrak{m}(x_{[\alpha]}(s \cdot t_2)) \in U_\sigma(J_S)$, which implies $x_{[\alpha]}((\chi(\alpha) - 1) \cdot t) \in U_\sigma(J_S)$, leading to a contradiction. Thus, we can assume that $x \neq 1$ or $y \neq 1$. \\
% \noindent \textbf{Step 3:} Next part follows similar lines to the proof of Proposition \ref{prop:U(R) cap H subset U(J)}.

%%%%%%%%%%%%%%%%%%%%%%%%%%%%%%%%%%%%%%%%%%%%%%%%%%

\vspace{2mm}

\noindent \textit{Proof of Proposition \ref{local:H subset G(R,J)}.} If $J_S = R_S$, the proof is complete. Therefore, we assume $J_S \neq R_S$. Since $J_S$ is a $\theta$-invariant proper ideal of $R_S$, we have $J_S \subset \operatorname{rad}(R_S) = I_S R_S$. By Corollary \ref{G(R,J)=UTV}, it follows that 
$$
G_\sigma(R_S, J_S) = U_\sigma(J_S) T_\sigma(R_S, J_S) U^{-}_\sigma(J_S).
$$ 
Assume, for the sake of contradiction, that $\psi_\mathfrak{m}(H) \not\subset G_\sigma(R_S, J_S)$. Under this assumption, we will show that there exists an element $z \in H$ such that $\psi_\mathfrak{m}(z) \not\in G_\sigma(R_S, J_S)$ and $\psi_\mathfrak{m}(z) \in U_\sigma(I_S R_S) T_\sigma(R_S) U^{-}_\sigma(R_S)$. This, however, leads to a contradiction with Proposition \ref{prop:U(RS) cap psi(H) subset U(JS)}.

Let $\pi: G_\sigma (R_S) \longrightarrow G_\sigma (k_S)$ be the canonical homomorphism. 
Suppose $\pi (\psi_\mathfrak{m} (H))$ is central. Then $\psi_\mathfrak{m} (H) \subset G_\sigma (R_S, I_S R_S)$ and hence, by Corollary~\ref{G(R,J)=UTV}, we are done. 
Now assume that $\pi(\psi_\mathfrak{m}(H))$ is non-central. Since $\pi \circ \psi_\mathfrak{m}$ is surjective on elementary subgroups, the subgroup $\pi(\psi_\mathfrak{m}(H))$ of $G_\sigma (k_S)$ is normalized by $E'_\sigma (k_S)$. We claim that $E'_\sigma (k_S) \subset \pi(\psi_\mathfrak{m}(H))$. Assuming the claim to be true for the moment, let us proceed to prove the rest of the result. 
For a given $x_{[\alpha]}(t + I_S) \in U^{-}_\sigma (k_S) \ (t \not \in I_S)$, by our claim, there exists $z \in H$ such that 
$$ \pi (\psi_\mathfrak{m} (z)) = x_{[\alpha]}(t + I_S) = \pi (\psi_\mathfrak{m} (x_{[\alpha]} (t))).$$
Therefore, $\psi_\mathfrak{m} (z) (\psi_\mathfrak{m} (x_{[\alpha]} (t)))^{-1} \in \ker{\pi} = G_\sigma (I_S R_S) = U_\sigma(I_S R_S) T_\sigma (I_S R_S) U^{-}_\sigma (I_S R_S)$, the last equality is due to Proposition \ref{G=UTV}. 
But then $$\psi_\mathfrak{m} (z) \in U_\sigma(I_S R_S) T_\sigma (I_S R_S) U^{-}_\sigma (R_S) \subset U_\sigma(I_S R_S) T_\sigma (R_S) U^{-}_\sigma (R_S),$$ as desired.

Now it only remains to prove the claim. To do this, we observe that $E'_\sigma(k_S) \cap \pi(\psi_\mathfrak{m}(H))$ is a normal subgroup of $E'_\sigma(k_S)$. However, the group $E'_\sigma(k_S)$ is simple over its center. 
To see this, assume first that $\mathfrak{m} = \bar{\mathfrak{m}}$. In this case, $k_S = k_\mathfrak{m}$ is a field, and our result follows from \cite[Theorem 34]{RS}. Now, consider the case where $\mathfrak{m} \neq \bar{\mathfrak{m}}$. By Proposition~\ref{prop:Chevalley as TChevalley}, we can deduce that $E'_\sigma(k_S)$ is isomorphic to $E_\pi(\Phi, k_\mathfrak{m})$. Therefore, applying \cite[Theorem 5]{RS}, we conclude that $E'_\sigma(k_S)$ is simple over its center.
Thus, we can conclude that either $E'_\sigma(k_S) \cap \pi(\psi_\mathfrak{m}(H)) \subset Z(G_\sigma(k_S))$ or $E'_\sigma(k_S) \cap \pi(\psi_\mathfrak{m}(H)) = E'_\sigma(k_S)$. Assume that $E'_\sigma(k_S) \cap \pi(\psi_\mathfrak{m}(H)) \subset Z(G_\sigma(k_S))$. Then, by Proposition \ref{G=G'}, we have $\pi(\psi_\mathfrak{m}(H)) \subset T_\sigma(k_S)$. Since $\pi(\psi_\mathfrak{m}(H))$ is non-central, there exists $h(\chi) \in \pi(\psi_\mathfrak{m}(H))$ with $\chi(\alpha) \neq 1$ for some $\alpha \in \Phi$. But then $[h(\chi), x_{[\alpha]}(1)] = x_{[\alpha]}(\chi(\alpha)-1) \in \pi(\psi_\mathfrak{m}(H))$, which is a contradiction to our assumption. Therefore, we must have $E'_\sigma(k_S) \cap \pi(\psi_\mathfrak{m}(H)) = E'_\sigma(k_S)$, that is, $E'_\sigma(k_S) \subset \pi(\psi_\mathfrak{m}(H))$. \qed

%%%%%%%%%%%%%%%%%%%%%%%%%%%%%%%%%%%%%%%%%%%%%%%%%%
%Appendix: $E'_\sigma (R)$ is a characteristic subgroup of $G_\sigma(R)$
%%%%%%%%%%%%%%%%%%%%%%%%%%%%%%%%%%%%%%%%%%%%%%%%%%

\appendix

\section{\texorpdfstring{$E'_\sigma (R)$}{E(R)} is a Characteristic Subgroup of \texorpdfstring{$G_\sigma(R)$}{G(R)}}\label{sec:char subgroup}

\begin{center}
  \textsc{by Pavel Gvozdevsky}\\
\end{center}

\vspace{2mm}

In this appendix, we present an application of the main theorems established in the paper. Specifically, we prove that \( E'_{\pi, \sigma}(\Phi, R) \) is a characteristic subgroup of \( G_{\pi, \sigma}(\Phi, R) \). A similar result for Chevalley groups was obtained by L.~N.~Vaserstein~\cite{LV}. In the twisted case, the result is established by the authors under the assumption that \( R \) is Noetherian. The extension to arbitrary commutative rings was later provided by Pavel Gvozdevsky, and this general case is the focus of the present appendix.

\begin{thm}
    \normalfont
    Let $R$ and $\Phi_\rho$ be as in Theorem~\ref{mainthm1}. 
	Let $H$ be a subgroup of $G_{\pi, \sigma} (\Phi, R)$ containing $E'_{\pi, \sigma} (\Phi, R)$.
    Then $E'_{\pi, \sigma} (\Phi, R)$ can be characterized as the smallest by inclusion among all the subgroups $K\le H$ that satisfy the following properties:
    \begin{enumerate}[(a)]
        \item $K$ is normal and is generated as a normal subgroup by a single element;
        \item $K=[K,K]$;
        \item the centralizer of $K$ in $H$ is abelian.
    \end{enumerate}
\end{thm}

\begin{proof}
    First let us show that $E'_{\pi, \sigma} (\Phi, R)$ satisfies the properties (a)--(c). It satisfies (a) by Proposition~\ref{z to Rz} applied to $z=1$; it satisfies (b) by Corollary~\ref{cor:normalized}; and it satisfies (c) by Theorem~\ref{thm:KS3}.
	
	Now let $K\le H$ be a subgroup that satisfies (a)--(c); we must prove that $E'_{\pi, \sigma} (\Phi, R)\le K$. Since by (a) $K$ is normal in $H$, it follows by Theorem~\ref{mainthm} that there exists a unique $\theta$-invariant ideal $J$ of $R$ such that 
	\[
	E'_\sigma(R, J) \subset K \subset G_\sigma(R, J).
	\]
	
	We claim that $J = R$. By (a) $K$ is generated as a normal subgroup by a single element $g_0$. Now, clearly $J$ is the smallest by inclusion $\theta$-invariant ideal such that $g_0\in G_{\sigma}(R,J)$; hence, $J$ is generated by all the entries of the matrices $\varpi_{ad}(g_0)-e$, $\overline{\varpi_{ad}(g_0)}-e$, and $\overline{\overline{\varpi_{ad}(g_0)}}-e$, where $\varpi_{ad}$ is the adjoint representation of the ambient Chevalley group, and $e$ is the identity matrix. Therefore, the ideal $J$ is finitely generated. Since by (b) the group $K$ is perfect, the uniqueness of the ideal $J$ implies that $J = JJ$. By Nakayama's Lemma, there exists $s \in R$ such that $s \equiv 1 \pmod{J}$ and $sJ = 0$. 
	Thus, $E'_\sigma(sR)$ is contained in the centralizer of $K$ in $H$; hence, (c) implies that $s=0$; hence, we have $J = R$; hence, we have  $E'_\sigma(R) \subset K$.
\end{proof}

\begin{cor}
    \normalfont
	Let $R$ and $\Phi_\rho$ be as in Theorem~\ref{mainthm1}. 
	Let $H$ be a subgroup of $G_{\pi, \sigma} (\Phi, R)$ containing $E'_{\pi, \sigma} (\Phi, R)$. 
    Then $E'_{\pi, \sigma} (\Phi, R)$ is a characteristic subgroup of $H$. 
    In particular, $E'_{\pi, \sigma} (\Phi, R)$ is a characteristic subgroup of $G_{\pi, \sigma} (\Phi, R)$.
\end{cor}

\begin{proof}
This is clear since the family of subgroups of $H$ satisfying conditions (a)--(c) is invariant under all automorphisms of $H$. 
\end{proof}

\begin{rmk}
	The Theorem above implies not only that $E'_\sigma(R)$ is characteristic, but also that any abstract isomorphism between (possibly different) twisted Chevalley groups must preserve elementary subgroups. 
\end{rmk}

%%%%%%%%%%%%%%%%%%%%%%%%%%%%%%%%%%%%%%%%%%%%%%%%%%
%Section: Acknowledgements
%%%%%%%%%%%%%%%%%%%%%%%%%%%%%%%%%%%%%%%%%%%%%%%%%%

\section*{Acknowledgements}
It is a pleasure to thank E. Bunina, P. Gvozdevsky, and D. Prasad for their comments on this paper. 
We also thank the referees for making many valuable suggestions that helped improve the paper. 

%%%%%%%%%%%%%%%%%%%%%%%%%%%%%%%%%%%%%%%%%%%%%%%%%%
%Section: Bibliography
%%%%%%%%%%%%%%%%%%%%%%%%%%%%%%%%%%%%%%%%%%%%%%%%%%

%%%%%%%%%%%%%%%%%%%%%%%%%%%%%%%%%%%%%%%%%%%%%%%%%%
%%%%%%%%%%%%%%%%%%%%%%%%%%%%%%%%%%%%%%%%%%%%%%%%%%


\begin{thebibliography}{AAAA}
%\bibliographystyle{alpha}
%\thispagestyle{empty}

\bibitem{EA1} Eiichi Abe, \emph{Coverings of twisted Chevalley groups over commutative rings}, Sci. Rep. Tokyo Kyoiku Daigaku Sect. A, Vol. 13 (1977), 194-218.

\bibitem{EA2} Eiichi Abe, \emph{Chevalley groups over local rings}, Tohoku Math. J. (2), Vol. 21 (1969), 474-494.

\bibitem{EA3} Eiichi Abe, \emph{Chevalley groups over commutative rings}, Radical theory, Proc. 1988, Sendai Conf., Vol. 83 (1989), 1-23.

\bibitem{EA&JH} Eiichi Abe and James F. Hurley, \emph{Centers of Chevalley groups over commutative rings}, Comm. Algebra, Vol. 16 (1988), 57-74.

\bibitem{EA&KS} Eiichi Abe and Kazuo Suzuki, \emph{On normal subgroups of Chevalley groups over commutative rings}, Tohoku Math. J. (2), Vol. 28 (1976), 185-198.

\bibitem{AB&RP} Anthony Bak and Raimund Preusser, \emph{The {E}-normal structure of odd dimensional unitary groups}, J. Pure Appl. Algebra, Vol. 222 (2018), No. 9, 2823--2880.

\bibitem{EB12:main} Elena I. Bunina, Automorphisms of Chevalley groups of different types over commutative rings, Journal of Algebra, 355, No. 1, pages 154-170 (2012).

\bibitem{RC} R. W. Carter, \emph{Simple Groups of Lie Type}, 2nd edition, Wiley, London (1989).

\bibitem{JH} James E. Humphreys, \emph{Introduction to Lie Algebras and Representation Theory}, Graduate Texts in Mathematics, Springer-Verlag, New York-Berlin (1972).

\bibitem{VP&AS} V. A. Petrov and A. K. Stavrova, \emph{Elementary subgroups in isotropic reductive groups}, St. Petersburg Math. J., Vol. 20 (2009), no. 4, 625--644.

\bibitem{EP&NV} Eugene Plotkin and Nikolai A. Vavilov, \emph{Chevalley groups over commutative rings. \Romannum{1}. {E}lementary calculations}, Acta Appl. Math., Vol. 45 (1996), 73--113. 

\bibitem{RP1} Raimund Preusser, \emph{The {E}-normal structure of {P}etrov's odd unitary groups over commutative rings}, Communications in Algebra, Vol. 48 (2020), No. 3, 1114--1131.

\bibitem{RP2} Raimund Preusser, \emph{Sandwich classification for {${\rm GL}_n(R)$}, {${\rm O}_{2n}(R)$} and {${\rm U}_{2n}(R,\Lambda)$} revisited}, Journal of Group Theory, Vol. 21 (2018), No. 1, 21--44.

\bibitem{RP3} Raimund Preusser, \emph{Sandwich classification for {$O_{2n+1}(R)$} and {$U_{2n+1}(R,\Delta)$} revisited}, Journal of Group Theory, Vol. 21 (2018), No. 4, 539--571.

\bibitem{RP4} Raimund Preusser, \emph{Structure of hyperbolic unitary groups {II}: {C}lassification of {E}-normal subgroups}, Algebra Colloquium, Vol. 24 (2017), No. 2, 195--232.

\bibitem{AS&AS} Anastasia Stavrova and Alexei Stepanov, \emph{Normal structure of isotropic reductive groups over rings}, Journal of Algebra, Vol. 656 (2024), 486--515.

\bibitem{RS} Robert Steinberg, \emph{Lectures on Chevalley Groups}, Yale University Press (1968).

\bibitem{RSTCG} Robert Steinberg, \emph{Variations on a Theme of Chevalley}, Pacific J. Math., Vol. 9 (1959), 875-891.

\bibitem{KS1} Kazuo Suzuki, \emph{On Normal Subgroups of Twisted Chevalley Groups over Local Rings}, Sci. Rep. Tokyo Kyoiku Daigaku Sect. A, Vol. 13 (1977), 238-249.

\bibitem{KS2} Kazuo Suzuki, \emph{Normality of the Elementary Subgroups of Twisted Chevalley Groups over Commutative Rings}, J. Algebra, Vol. 175 (1995), 526-536.

\bibitem{KS3} Kazuo Suzuki, \emph{Centers of Twisted Chevalley Groups over Commutative Rings}, Kumamoto J. Math., Vol. 6 (1993), 1-9.

\bibitem{GT} Giovanni Taddei, \emph{Normalit\'{e} des groupes \'{e}l\'{e}mentaires dans les groupes de {C}hevalley sur un anneau}, Contemp. Math. (2), Vol. 55 (1986), 693-710.

\bibitem{LV} Leonid N. Vaserstein, \emph{On normal subgroups of Chevalley groups over commutative rings}, Tohoku Math. J. (2), Vol. 38 (1986), 219-230.

\bibitem{NV1} Nikolai A. Vavilov, \emph{Structure of Chevalley groups over commutative rings}, Nonassociative algebras and related topics (Hiroshima, 1990), World Sci. Publ., River Edge, NJ (1991), 219--335.



\end{thebibliography}
\end{document}